\documentclass[10pt,oneside,leqno]{amsart}
\usepackage{amsxtra}
\usepackage{amsopn}
\usepackage{color}
\usepackage{amsmath,amsthm,amssymb}
\usepackage{amscd}
\usepackage{amsfonts}
\usepackage[utf8]{inputenc}
\usepackage{latexsym}
\usepackage{verbatim}
\usepackage{graphicx}
\usepackage{xypic}
\usepackage{booktabs}
\usepackage{array}
\usepackage{multirow}
\usepackage{caption}
\usepackage{multicol}
\usepackage{lscape}
\usepackage{xypic}
\newcommand{\referenza}{}

\theoremstyle{plain}
\newtheorem{theorem}{Theorem}[section]
\newtheorem*{theorem*}{Theorem \referenza}
\newtheorem{definition}[theorem]{Definition}
\newtheorem{lemma}[theorem]{Lemma}
\newtheorem{proposition}[theorem]{Proposition}
\newtheorem{corollary}[theorem]{Corollary}
\newtheorem{remark}[theorem]{Remark}

\newtheorem{remark-question}[section]{Remark-Question}

\newcommand\C{{\mathbb C}}

\newcommand\N{{\mathbb N}}
\newcommand\Q{{\mathbb Q}}
\newcommand\R{{\mathbb R}}
\newcommand\Z{{\mathbb Z}}

\newcommand\frg{{\mathfrak g}}
\newcommand\frh{{\mathfrak h}}

\newcommand\frs{{\mathfrak s}}

\newcommand\Real{{\mathfrak R}{\mathfrak e}\,} %%%%% Real and Imaginary parts %%%%%
\newcommand\Imag{{\mathfrak I}{\mathfrak m}\,}

\newcommand\db{{\bar{\partial}}}

\newcommand\GL{{\hbox{\rm Aut}}}

\definecolor{fondo}{rgb}{0.93,0.93,0.93}

\usepackage[hypertexnames=false,
backref=page,
    pdftex,
    pdfpagemode=UseNone,
    breaklinks=true,
    extension=pdf,
    colorlinks=true,
    linkcolor=blue,
    citecolor=blue,
    urlcolor=blue,
]{hyperref}

\allowdisplaybreaks[1]

\title{Complex structures of splitting type}
\thanks{
During the preparation of the work, the first author has been granted by a research fellowship by Istituto Nazionale di Alta Matematica INdAM and by a Junior Visiting Position at Centro di Ricerca ``Ennio de Giorgi''; he is also supported by the Project PRIN ``Varietà reali e complesse: geometria, topologia e analisi armonica'', by the Project FIRB ``Geometria Differenziale e Teoria Geometrica delle Funzioni'', by SNS GR14 grant ``Geometry of non-Kähler manifolds'',  by SIR2014 project RBSI14DYEB ``Analytic aspects in complex and
hypercomplex geometry'' and by GNSAGA of INdAM.
This work has been partially supported by the projects MINECO (Spain) MTM2014-58616-P,
and Gobi\-er\-no de Arag\'on/Fondo Social Europeo, grupo consolidado E15-Geometr\'{\i}a.
The first author would like to warmly thank for the kind hospitality during his stay at Departamento de Matem\'aticas
of the Universidad de Zaragoza.  The last author would like also to thank the Dipartimento di Matematica e Informatica
of Universit\`a degli Studi di Parma (Italia).
We wish to thank D. Andriot and A. Fino for their suggestions and remarks.
We also thank the referees for many useful comments and suggestions that have helped us
to improve the final version of the paper.}

\author{Daniele Angella}
\address[D. Angella]{Dipartimento di Matematica e Informatica ``Ulisse Dini''\\
Universit\`a degli Studi di Firenze\\
viale Morgagni 67/a\\
50134 Firenze, Italy
}
\email{daniele.angella@unifi.it}

\author{Antonio Otal}
\address[A. Otal]{Centro Universitario de la Defensa\,-\,I.U.M.A.\\
Academia General Militar\\
Ctra. de Huesca s/n. 50090 Zaragoza\\
Spain}
\email{aotal@unizar.es}

\author{Luis Ugarte}
\address[L. Ugarte]{Departamento de Matem\'aticas\,-\,I.U.M.A.\\
Universidad de Zaragoza\\
Campus Plaza San Francisco\\
50009 Zaragoza, Spain}
\email{ugarte@unizar.es}

\author{Raquel Villacampa}
\address[R. Villacampa]{Centro Universitario de la Defensa\,-\,I.U.M.A.\\
Academia General Militar\\
Ctra. de Huesca s/n. 50090 Zaragoza\\
Spain}
\email{raquelvg@unizar.es}

\begin{document}

\begin{abstract}
We study the six-dimensional solvmanifolds that admit complex structures of splitting type  classifying the underlying solvable Lie algebras.
In particular, many complex structures of this type exist on the Nakamura manifold $X$, and they allow
us to construct a countable family of compact complex non-$\partial\db$ manifolds $X_k$, $k\in\Z$, that admit a small
holomorphic deformation $\{(X_{k})_{t}\}_{t\in\Delta_k}$ satisfying the $\partial\db$-Lemma for any $t\in\Delta_k$
except for the central fibre.
Moreover, a study of the existence of special Hermitian metrics is also carried out on
six-dimensional solvmanifolds with splitting-type complex structures.
\end{abstract}

\maketitle

\section*{Introduction}

\noindent Let $\frg$ be a real Lie algebra of even dimension. A complex structure on $\frg$ is an endomorphism $J\colon \frg\to\frg$ satisfying $J^2=-\text{Id}_\frg$ and the \emph{Nijenhuis} condition
\begin{equation}\label{nij}
\text{Nij}_J(X,Y):=[JX,JY]-J[JX,Y]-J[X,JY]-[X,Y]=0,\quad\text{for all }X,Y\in\frg.
\end{equation}
An important problem is to find the Lie algebras that admit such a structure. They allow
to construct many interesting examples of compact complex manifolds whenever the simply-connected Lie group $G$ of $\frg$ has a lattice $\Gamma$ of maximal rank. Indeed,
by extending $J$ to the group $G$ and then passing $J$ to the quotient $G/\Gamma$ one obtains nilmanifolds, resp. solvmanifolds, when $G$ is nilpotent, resp. solvable,
endowed with $G$-left-invariant complex structures.
In real dimension four, the solvable Lie algebras admitting a complex structure
have been classified by Ovando in \cite{Ovando},
however no  general result is known in higher dimension.
Focused in six dimensions, Salamon \cite{salamon} classifies the nilpotent Lie algebras that admit a complex structure,
finding eighteen non-isomorphic Lie algebras (see also \cite{ceballos-otal-ugarte-villacampa}).
In \cite{andrada-barberis-dotti} Andrada, Barberis and Dotti obtain
the Lie algebras endowed with a complex structure $J$ of \emph{abelian} type,
i.e. $J$ satisfies $[JX,JY]=[X,Y]$ for all $X,Y\in\frg$. More recently, Fino and the second and third authors \cite{fino-otal-ugarte} classify the 6-dimensional unimodular solvable Lie algebras admitting a complex structure $J$ with non-zero closed $(3,0)$-form $\Psi$.

The existence of a closed nowhere vanishing $(n,0)$-form $\Psi$ on a $2n$-dimensional almost complex manifold automatically implies
the Nijenhuis condition \eqref{nij}, and such complex manifolds have holomorphically trivial canonical bundle.
Nilmanifolds with $G$-left-invariant complex structures are examples of this kind;
in fact, by \cite[Theorem 1.3]{salamon}, for any basis $\{\omega^j\}_{j=1}^n$ of (1,0)-forms on the
underlying nilpotent Lie algebra, the $(n,0)$-form $\Psi=\omega^1\wedge\cdots\wedge\omega^n$ is closed.
However, this is no longer true for general $G$-left-invariant complex structures on solvmanifolds.
In \cite[Proposition 2.1]{fino-otal-ugarte} it is proved that for solvmanifolds,
the existence of a $G$-left-invariant complex structure with holomorphically trivial canonical bundle
is equivalent to the existence of a non-zero closed $(n,0)$-form on the Lie algebra underlying the solvmanifold.
The second author finds in \cite[Chapter 4]{otal-phd} that several of the complex structures with holomorphically trivial canonical bundle on $6$-dimensional solvmanifolds are also of splitting type, i.e. they satisfy \cite[Assumption 1.1]{kasuya-mathz} (see Definition~\ref{def-splitting-type} for details), but that there are other complex structures
that are not of splitting type. %%%(indeed the underlying Lie algebras can not admit such structure).

In addition to providing an important source of examples of compact complex manifolds
with unusual and interesting properties, the complex structures of splitting type have also
interest because they constitute a natural solvable extension of complex nilmanifolds,
as they are certain semi-direct products of the latter by $\mathbb{C}^n$.
In this sense, they allow to investigate to what extent geometric properties of nilmanifolds
still survive in this larger class of homogeneous spaces.  See, e.g., the deformation limits
constructed in \cite{angella-kasuya-2}; compare also the
observation~\cite{kasuya-vaisman} that Oeljeklaus-Toma manifolds are
solvmanifolds of real splitting type endowed with a left-invariant
complex structure, and as such they do not admit Vaisman metrics.
Furthermore, some complex cohomological invariants of the manifold
can be obtained explicitly, which allows to study several aspects of
their complex~\cite{kasuya-jdg, kasuya-mathz, kasuya-hodge,
angella-kasuya-1} and Hermitian~\cite{kasuya-vaisman,
kasuya-geom-formality, FKV} geometry.

One of these invariants are the Dolbeault cohomology groups.
For nilmanifolds, several steps have been done in \cite{CF,CFGU2,R1,R2} towards
the (still open) conjecture that the Dolbeault cohomology of a nilmanifold
with $G$-left-invariant complex structure $J$ can be computed in terms of invariant forms on $G$, i.e. in terms of the pair $(\frg,J)$.
Concerned with the calculus of the Dolbeault cohomology of solvmanifolds, Kasuya \cite{kasuya-mathz} provides a technique to compute such complex invariants when the complex structure is of splitting type. The Dolbeault cohomology groups are obtained by means of a certain finite-dimensional subalgebra of the de Rham complex and, more recently, the first author and Kasuya
develop in \cite{angella-kasuya-1} a technique to compute the Bott-Chern cohomology
by means of another finite-dimensional subalgebra.
These techniques have allowed to study the deformation limits of compact complex $\partial\db$-manifolds \cite{angella-kasuya-2}
and of compact balanced manifolds~\cite{fino-otal-ugarte}.

Our objective in this paper is the complex geometry of 6-dimensional solvmanifolds endowed with
a ($G$-left-invariant) complex structure of splitting type.
The paper is structured as follows. In Section \ref{sec:classification}, we obtain the
solvable Lie algebras that may support a splitting-type complex structure. More concretely, in
Theorem~\ref{thm:main-thm} we prove that if $G/\Gamma$ is a 6-dimensional solvmanifold
endowed with a complex structure $J$ of splitting type, then the Lie algebra $\frg$ of $G$
is isomorphic to $\frs_k$ for some $1 \leq k\leq 12$ (see the list in Theorem~\ref{thm:main-thm}
for a description of the Lie algebras $\frs_k$). Since six of the Lie algebras $\frs_k$ have parameters in their description,
the number of non-isomorphic Lie algebras underlying the solvmanifolds
with splitting-type complex structure is not finite.
In Remark~\ref{remark-lattices} we discuss the existence of lattices.

In Section \ref{sec:hermitian}, we investigate the existence of Hermitian metrics,
with special attention to strong K\"ahler with torsion (SKT) and balanced metrics.
In particular, we obtain SKT structures on solvmanifolds corresponding to $\mathfrak{s}_1$
and we show the existence of balanced structures on the other Lie algebras $\mathfrak{s}_k$ for $2 \leq k\leq 12$
(see Table~\ref{table:summary-metrics}).
A conjecture of Fino and Vezzoni \cite{fino-vezzoni} states that in the compact non-K\"ahler case it is never possible to find
an SKT metric and also a balanced one. In \cite{fino-vezzoni-nilmanifolds} they prove the conjecture for nilmanifolds
and in \cite{fino-vezzoni} for 6-dimensional solvmanifolds having holomorphically trivial canonical bundle.
As a consequence of our study in Section \ref{sec:hermitian}, it turns out that the conjecture also holds for any splitting-type
complex structure on a 6-dimensional solvmanifold.
On the other hand, Popovici proposes in \cite{Pop-2015} a conjecture relating the balanced and the Gauduchon cones of $\partial\db$-manifolds, and he observes that, if proved to hold, the conjecture would imply
the existence of a balanced metric on any $\partial\db$-manifold.
Since solvmanifolds corresponding to $\mathfrak{s}_1$ do not satisfy the $\partial\db$-Lemma,
as another consequence of our study in Section~\ref{sec:hermitian}, one has that
balanced metrics exist on any $\partial\db$-solvmanifold of dimension~6 endowed with a splitting-type
complex structure (see Corollary~\ref{cor-consec}).

Finally, Section \ref{sec:nakamura} is devoted to the complex geometry of the Nakamura manifold
and to the construction of some analytic families of compact complex structures on it.
The Lie algebra underlying the Nakamura manifold is $\mathfrak{s}_{12}$, and the complex-parallelizable
structure given in \cite{nakamura} and the abelian complex structure found in \cite{andrada-barberis-dotti}
are particular examples of splitting-type complex structures.
After classifying, up to equivalence, the splitting-type complex structures on
the Nakamura manifold (see Proposition~\ref{prop:moduli_Nk}),
we prove in Theorem~\ref{canonico_no_abierta}, by an appropriate deformation of its abelian complex structure,  that
the property of \emph{having holomorphically trivial canonical bundle}
and the property of \emph{being of splitting type} are not stable under holomorphic deformations.

Moreover, in Theorem~\ref{thm:non-closedness} we construct, for each $k\in\Z$,
a compact complex manifold $X_k$ that does not satisfy the $\partial\bar\partial$-Lemma, and we prove that
$X_k$ admits a small holomorphic deformation $\{(X_{k})_{t}\}_{t\in\Delta_k}$, $\Delta_k$ being an open disc in $\mathbb{C}$ around $0$,
such that $(X_{k})_{t}$ is a
compact complex $\partial\bar\partial$-manifold for any $t\neq 0$.
For the proof of this result we make use of the complex geometry on $\mathfrak{s}_{12}$,
since the compact complex manifolds $X_k$, $k\in\Z$, and all of their small holomorphic deformations $(X_{k})_{t}$,
$t\in\Delta_k$, are solvmanifolds corresponding to $\mathfrak{s}_{12}$ endowed with complex structures of splitting type.
Furthermore, they all have holomorphically trivial canonical bundle and admit a balanced metric.

When we consider the case $k=-1$, then we recover the main result in \cite{angella-kasuya-2}
because it corresponds precisely to the
complex-parallelizable structure.
So our Theorem~\ref{thm:non-closedness} shows that the result extends to a countable family of complex structures.
%{\color{red}{Indeed, the complex structures on $X_k$ are not equivalent!}}
Since one of the complex structures (concretely $k=0$) is the abelian one \cite{andrada-barberis-dotti},
we have in particular that the abelian complex structure can be deformed to complex structures satisfying the $\partial\bar\partial$-Lemma.
In other words, the abelian complex structure on the Nakamura manifold (which does not satisfy the $\partial\bar\partial$-Lemma)
is the central limit of an analytic family
of compact complex $\partial\bar\partial$-manifolds.

\section{The Lie algebras underlying the solvmanifolds with complex structures of splitting type}\label{sec:classification}

\noindent
We are concerned with solvmanifolds $X=G/\Gamma$ endowed with a complex structure of \emph{splitting type} in the following sense:

\begin{definition}\label{def-splitting-type}{\cite[Assumption 1.1]{kasuya-mathz}}
A solvmanifold $X = G/\Gamma$ endowed with a $G$-left-invariant complex structure $J$ is said to be of {\em splitting type} if $G$ is a semi-direct product $G=\C^{n}\ltimes_{\varphi}N$ such that:
\begin{enumerate}
 \item\label{item:ass-1} $N$ is a connected simply-connected $2k$-dimensional nilpotent Lie group endowed with an $N$-left-invariant complex structure $J_N$;
 \item\label{item:ass-2} for any $\mathbf{z}\in \C^{n}$, it holds that $\varphi(\mathbf{z})\in \GL(N)$ is a holomorphic automorphism of $N$ with respect to $J_N$;
 \item\label{item:ass-3} $\varphi$ induces a semi-simple action on the Lie algebra $\mathfrak n$ associated to $N$;
 \item\label{item:ass-4} $G$ has a lattice $\Gamma$ (then $\Gamma$ can be written as $\Gamma = \Gamma_{\C^n} \ltimes_{\varphi} \Gamma_{N}$ such that $\Gamma_{\C^n}$ and $\Gamma_{N}$ are  lattices of $\C^{n}$ and $N$, respectively, and, for any $\mathbf{z}\in \Gamma_{\mathbb C^n}$, it holds $\varphi(\mathbf{z})\left(\Gamma_N\right)\subseteq\Gamma_N$);
 \item\label{item:ass-5} the inclusion $\wedge^{\bullet,\bullet}\left(\mathfrak n\otimes_\R\C\right)^* \hookrightarrow \wedge^{\bullet,\bullet}\left(N\slash \Gamma_{N} \right)$ induces the isomorphism in cohomology
$$ H^{\bullet,\bullet}_{\bar\partial}\left(\wedge^{\bullet,\bullet}\left(\mathfrak n\otimes_\R\C\right)^*\right) \stackrel{\cong}{\to} H^{\bullet,\bullet}_{\bar\partial }\left(N\slash \Gamma_{N}\right) \;. $$
\end{enumerate}
\end{definition}

We recall the construction of the complex structure (for further details see \cite{kasuya-mathz}). Let $G=\C^n\ltimes_\varphi N$;
taking $\mathbf{z}=(z_1,\ldots, z_n) \in \C^n$,
we consider $\{dz_1,\ldots,dz_n\}$ the standard $(1,0)$-basis of $\C^n$.
Consider $\{\varphi^1,\ldots,\varphi^k\}$ the $N$-invariant $(1,0)$-basis
such that the induced action  is given by the diagonal matrix
$$
%%%\begin{equation}\label{phi_gen}
\varphi(\mathbf{z})=\begin{pmatrix}\alpha_1&&\\&\ddots&\\&&\alpha_k\end{pmatrix},
%%%\end{equation}
$$
where $\alpha_j\in\text{Hom}(\C^n;\C^*)$ are characters of $\C^n$, $j=1,\ldots,k$.
Then $\{dz_1,\ldots,dz_n, \alpha_1^{-1}\varphi^1,\ldots,\alpha_k^{-1}\varphi^k\}$
is a $G$-invariant $(1,0)$-basis for the complex structure on $G=\C^n\ltimes_\varphi N$.

\subsection{Reduced equations of splitting-type complex structures in dimension~\texorpdfstring{$6$}{6}}\label{subsec:classification-complex}

If the complex dimension of the solvmanifold is $n+k=3$, then we have the following cases:
$G=\C^2\ltimes_\varphi\C$
or $G=\C\ltimes_\varphi N$, where the nilpotent factor $N$ in the semi-direct product has real dimension $4$ and it is endowed with a left-invariant complex structure.
There are only two possibilities for $N$, namely the complex surface $\C^2$ or the real 4-dimensional nilpotent Lie group $KT$ with Lie algebra $\mathfrak{Kt} = \frh_3\oplus\R$ (we denote by $\frh_3$ the real $3$-dimensional Heisenberg Lie algebra) endowed with the left-invariant complex structure defined by a basis of $(1,0)$-forms $\{\tau,\sigma\}$ satisfying
\begin{equation}\label{kod_thurs}
\left\{\begin{array}{l}
d\tau=0,\\[3pt]d\sigma=\tau\wedge\bar\tau.
\end{array}\right.
\end{equation}
The nilmanifold $KT/\Gamma$ endowed with the complex structure \eqref{kod_thurs} is the well-known Kodaira-Thurston manifold.

For the case $\C\ltimes_\varphi N$, either for $N=\C^2$ or $KT$,
the action $\varphi\colon\C\to {\GL}(N)$ will be represented for every $z_3\in\C$ by a diagonal matrix of the form
\begin{equation}\label{phi_dos}
 \varphi(z_3)=\begin{pmatrix}e^{Az_3+B\bar z_3}&0\\0&e^{Cz_3+D\bar z_3}\end{pmatrix},
\end{equation}
where $A,B,C,D\in\C$. For the case $\C^2\ltimes_\varphi\C$, the action is %%%always diagonal and simply
given for every $(z_2,z_3)\in\C^2$ by
$$
\varphi(z_2, z_3) = e^{Az_2+B\bar z_2+Cz_3+D\bar z_3},
$$
where $A,B,C,D\in\C$.

\begin{proposition}\label{prop:eq-struttura-1}
Let $X=G/\Gamma$ be a $6$-dimensional solvmanifold endowed with a complex structure of splitting type, and suppose that
$G=\C^2\ltimes_\varphi\C$ or $G=\C\ltimes_\varphi\C^2$. If $\frg$ is the Lie algebra of $G$, then
there is a basis $\{\omega^1,\omega^2,\omega^3\}$ for $(\frg^{1,0})^{\ast}$ satisfying the complex structure equations
\begin{equation*}
\left\{\begin{array}{l}
 d\omega^1=A\,\omega^{13}+B\omega^{1\bar3},\\
 d\omega^2= -(A+\bar B+\varepsilon)\,\omega^{23}+\varepsilon\,\omega^{2\bar3},\\
d\omega^3=0,
\end{array}\right.
\end{equation*}
for some $A,B\in\C$ and $\varepsilon\in\{0,1\}$.  (Here, and in what follows, $\omega^{\bar k}$ stands for $\overline{\omega^k}$.)
\end{proposition}

\begin{proof}
Let $G=\C\ltimes_\varphi N$ be the semi-direct product where the action $\varphi\colon\C\to{\GL}(N)$ is given by the matrix \eqref{phi_dos}, once fixed a $(1,0)$-coframe for $N$.
We are considering the case $N=\C^2$. Hence, $\varphi(z_3)$ is automatically an automorphism of $\C^2$ and the complex structure
on $G$ is determined by the global $G$-invariant $(1,0)$-basis
$\{\omega^1=e^{-Az_3-B\bar z_3}dz_1,\,\omega^2=e^{-Cz_3-D\bar z_3}dz_2,\,\omega^3=dz_3\}$.
The complex structure equations in the basis $\{\omega^1,\omega^2,\omega^3\}$ are
$$
d\omega^1=A\omega^{13}+B\omega^{1\bar3},\quad
d\omega^2=C\omega^{23}+D\omega^{2\bar3},\quad
d\omega^3=0.
$$
The unimodularity of $G$ is equivalent to the condition $d(\wedge^{3,2}\frg^*\oplus\wedge^{2,3}\frg^*)=\{0\}$,
which forces $A+\bar B+C+\bar D=0$.
Clearly, if $D=0$, then $C=-A-\bar B$.
Now, if $D\neq 0$ then, up to scaling $\omega^3$, we can suppose that $D$ is equal to 1 and
so $C=-A-\bar B-1$, arriving at the desired structure equations.

\smallskip

Consider next the case $G=\C^2\ltimes_\varphi\C$.
In this case we have a $(1,0)$-coframe $\{\eta^1,\eta^2,\eta^3\}$
given by $\{\eta^1=e^{-Az_2-B\bar z_2-Cz_3-D\bar z_3}dz_1,\,\eta^2=dz_2,\,\eta^3=dz_3\}$. Hence, the structure equations are
$$
d\eta^1 = A\eta^{12}+B\eta^{1\bar2}+C\eta^{13}+D\eta^{1\bar3}, \quad\
d\eta^2 = d\eta^3 = 0 .
$$
The unimodularity condition is equivalent to $A + \bar B=0$ and $C + \bar D=0$.
Thus, we can consider $(A,C)\neq (0,0)$, because otherwise $\varphi$ is trivial.
Now, if $A\neq 0$ (similarly for $C\neq 0$ when $A=0$) then the change of basis $\{\omega^1=\eta^1,\,\omega^2=\eta^3,\,\omega^3=A\eta^2+C\eta^3 \}$
provides the structure equations $d\omega^1 = \omega^{13}-\omega^{1\bar3}$ and $d\omega^2 = d\omega^3 = 0$.
\end{proof}

\begin{proposition}\label{prop:eq-struttura-2}
Let $X=G/\Gamma$ be a $6$-dimensional solvmanifold endowed with a complex structure
of splitting type, and suppose that $G=\C\ltimes_\varphi KT$.
Then, there is a $(1,0)$-basis $\{\omega^1,\omega^2,\omega^3\}$ satisfying the complex structure equations
\begin{equation*}%\label{split_kt}
\left\{\begin{array}{l}
 d\omega^1=\varepsilon\,(\omega^{13}-\omega^{1\bar3}),\\[3pt]
 d\omega^2=\omega^{1\bar1},\\[3pt]
 d\omega^3=0,
\end{array}\right.
\end{equation*}
where $\varepsilon\in\{0,1\}$.
\end{proposition}

\begin{proof}
%If $N=KT$, consider $\psi\colon KT\to\C^2$ a global chart on $N$ such that $\psi(g)=(z_1,z_2)\in\C^2$. Hence, we can see $KT$ as the Lie group %$(\C^2,*)$ where $*$ is the induced product on $\C^2$ given by $(a_1,a_2)*(z_1,z_2)=(a_1+z_1,a_2+z_2+\bar a_1z_1)$. Attending at the product law it is %direct to check that $\varphi(z_3)\in\text{GL}(KT)$ for any $z_3\in\C$ if and only if $C=A+\bar B$ and $D=\bar A+B$ in~\eqref{phi_dos}.
%
%Now, taking the $G$-invariant $(1,0)$-basis $\{\eta^1=e^{Az_2+B\bar z_2}\varphi^1,\, \eta^2=e^{(A+\bar B)z_3+(\bar A+B)\bar z_3}\varphi^2,\, %\eta^3=dz_3\}$, the complex structure equations in this basis are:
%\begin{equation}\label{eqs_dos}
%\left\{\begin{array}{l}
%d\eta^1=A\eta^{31}-B\eta^{1\bar3},\\[3pt]
%d\eta^2=(A+\bar B)\eta^{32}-(\bar A+B)\eta^{2\bar3}+\eta^{1\bar1},\\[3pt]
%d\eta^3=0.
%\end{array}\right.
%\end{equation}
%Finally, the unimodularity condition implies $B=-\bar A$. Hence,  if $A\neq 0$ we can choose the new $(1,0)$-basis  %$\omega^1=\eta^1,\,\omega^2=\eta^2,\,\omega^3=A\eta^3$ to get the reduced equations of the proposition. %arriving to \eqref{split_kt}.
%
%It remains to show that this case covers all the possibilities.  In fact, in general,
%
%
The semisimple action induced by $\varphi$ on $\mathfrak{Kt}$ assures the existence of a basis for $\mathfrak{Kt}$
such that the action is diagonal. So, we can take a basis of the form
$$ P \cdot \begin{pmatrix}\tau\\ \sigma\end{pmatrix} , \qquad \text{ where } \qquad
P\;=\;\begin{pmatrix}p_{11}&p_{12}\\p_{21}&p_{22}\end{pmatrix}\in \text{GL}(2,\mathbb C) \;$$
and $\{\tau,\sigma\}$ is the preferred basis of $(1,0)$-forms with structure equations \eqref{kod_thurs}.
%In the following, we show that it is necessary that $p_{12}=p_{21}=0$.

Denote also
$$Q \;:=\; P^{-1} \;=\;\begin{pmatrix}q_{11}&q_{12}\\q_{21}&q_{22}\end{pmatrix}
\;=\; \frac{1}{p_{11}p_{22}-p_{12}p_{21}}
\begin{pmatrix}p_{22}&-p_{12}\\-p_{21}&p_{11}\end{pmatrix}
\;.$$

With respect to this basis, we can assume that the action $\varphi$ is diagonal and given by the inverse of the matrix~\eqref{phi_dos},
which we will denote simply by $\alpha$.
So, the invariant basis we choose is
$$ \left\{\omega^1, \omega^2, \omega^3:= dz_3  \right\} , \qquad \text{ where }
\qquad \begin{pmatrix}\omega^1\\
\omega^2\end{pmatrix}=\alpha\cdot P \cdot \begin{pmatrix}\tau\\ \sigma\end{pmatrix} \;.$$

Since
$$d\alpha \;=\; -\alpha \cdot \mathcal{E}, \qquad \text{ where } \qquad
\mathcal{E} \;:=\; \begin{pmatrix}A\,\omega^3 + B\,\omega^{\bar 3}&0\\0&C\,\omega^3 + D\,\omega^{\bar 3}\end{pmatrix} \;,$$
%$${\color{red} d\alpha \;=\; \alpha \cdot \mathcal{E} \cdot \begin{pmatrix}\omega^3\\ \omega^{\bar 3}\end{pmatrix}} \footnote{{\color{blue} REMARK 4}} , \qquad \text{ where } \qquad
%\mathcal{E} \;:=\; \begin{pmatrix}A&B\\C&D\end{pmatrix} \;,$$
whence we get the structure equations (here $\wedge$ is intended componentwise):
\begin{eqnarray*}
d\begin{pmatrix}\omega^1\\\omega^2\end{pmatrix}&=& d\alpha\wedge P \cdot
\begin{pmatrix}\tau\\ \sigma\end{pmatrix}+ \alpha\cdot P \cdot d
\begin{pmatrix}\tau\\\sigma\end{pmatrix} \\[5pt]
&=& -\alpha\cdot\mathcal E \wedge \alpha^{-1} \cdot
\begin{pmatrix}\omega^1\\\omega^2\end{pmatrix}+\alpha\cdot P \cdot
\begin{pmatrix}0\\\tau\wedge\bar\tau\end{pmatrix}\\[5pt]
&=& -\begin{pmatrix} \left(A\omega^3+B \omega^{\bar3}\right) \wedge \omega^1\\
\left( C\omega^3+D\omega^{\bar3}\right) \wedge \omega^2\end{pmatrix}
+\alpha\cdot P \cdot \left(
\begin{pmatrix}0\\\tau\end{pmatrix}\wedge\begin{pmatrix}0\\\bar\tau\end{pmatrix}
\right) \\
&=& \begin{pmatrix} \omega^1 \wedge \left(A\omega^3+B \omega^{\bar3}\right) \\
\omega^2 \wedge \left( C\omega^3+D\omega^{\bar3}\right) \end{pmatrix}
+\alpha\cdot P \cdot \left(
\begin{pmatrix}0\\\tau\end{pmatrix}\wedge\begin{pmatrix}0\\\bar\tau\end{pmatrix}
\right).
\end{eqnarray*}

Since there is no dependence on $\alpha$ in the first term, it is well-defined for any value of the parameters $A,B,C,D\in \C$.

As for the second term:
\begin{eqnarray*}
\lefteqn{\alpha\cdot P \cdot \left( \begin{pmatrix}0\\\tau\end{pmatrix}\wedge \begin{pmatrix}0\\ \bar\tau\end{pmatrix}\right)}\\[5pt]
&=& \alpha\cdot \left(P \cdot \begin{pmatrix}0\\ q_{11}\alpha_1^{-1}\omega^1+q_{12}\alpha_2^{-1}\omega^2 \end{pmatrix} \right) \wedge \left( \bar P \cdot \begin{pmatrix}0\\\bar q_{11}\bar\alpha_1^{-1}\omega^{\bar1}+\bar q_{12}\bar\alpha_2^{-1}\omega^{\bar2}\end{pmatrix} \right) \\[5pt]
&=& \alpha\cdot \left(\begin{pmatrix}p_{12}\cdot\left(q_{11}\alpha_1^{-1}\omega^1+q_{12}\alpha_2^{-1}\omega^2\right)\\ p_{22}\cdot\left(q_{11}\alpha_1^{-1}\omega^1+q_{12}\alpha_2^{-1}\omega^2\right) \end{pmatrix} \wedge \begin{pmatrix}\bar p_{12}\cdot\left(\bar q_{11}\bar\alpha_1^{-1}\omega^{\bar1}+\bar q_{12}\bar\alpha_2^{-1}\omega^{\bar2}\right)\\\bar p_{22}\cdot\left(\bar q_{11}\bar\alpha_1^{-1}\omega^{\bar1}+\bar q_{12}\bar\alpha_2^{-1}\omega^{\bar2}\right)\end{pmatrix} \right) \\[5pt]
&=&\frac{1}{\left|p_{11}p_{22}-p_{12}p_{21}\right|^2} \begin{pmatrix}
  \alpha_1 \cdot |p_{12}|^2
 \\[5pt]
 \alpha_2 \cdot |p_{22}|^2
 \end{pmatrix} \cdot \omega\;,
\end{eqnarray*}
where $$\omega =
 \alpha_1^{-1}\bar\alpha_1^{-1}|p_{22}|^2\omega^{1\bar1}
 -
 \alpha_1^{-1}\bar\alpha_2^{-1}p_{22}\bar p_{12}\omega^{1\bar2}
 -
 \bar\alpha_1^{-1}\alpha_2^{-1}p_{12}\bar p_{22}\omega^{2\bar1}
 +
 \alpha_2^{-1}\bar\alpha_2^{-1}|p_{12}|^2\omega^{2\bar2}.
$$

So,
\begin{equation*}\label{dif_kt}
d\begin{pmatrix}\omega^1\\\omega^2\end{pmatrix}= \begin{pmatrix} \left(A\omega^3+B \omega^{\bar3}\right) \wedge \omega^1\\
\left( C\omega^3+D\omega^{\bar3}\right) \wedge \omega^2\end{pmatrix} + \frac{1}{\left|p_{11}p_{22}-p_{12}p_{21}\right|^2} \begin{pmatrix}
  \alpha_1 \cdot |p_{12}|^2
 \\[5pt]
 \alpha_2 \cdot |p_{22}|^2
 \end{pmatrix} \cdot \omega.
\end{equation*}

Now, we have to take care about the dependence on $z_3$ of the terms in the expression above. Note that the case $(A,B,C,D)=(0,0,0,0)$ is trivial, that is, it yields just the product. %(case 0 in Table \ref{table:cases-kt}\label{case0}).
Let us assume $(A,B)\neq (0,0)$.  The term
$$\alpha_1\cdot \frac{|p_{12}|^2}{\left|p_{11}p_{22}-p_{12}p_{21}\right|^2} \cdot\alpha_1^{-1}\bar\alpha_1^{-1}|p_{22}|^2\omega^{1\bar1} \;=\; \bar\alpha_1^{-1} \cdot \frac{|p_{12}|^2|p_{22}|^2}{\left|p_{11}p_{22}-p_{12}p_{21}\right|^2} \omega^{1\bar1} \; $$
contains $\bar\alpha_1^{-1}$ that depends on $z_3$. So either $p_{12}=0$ or $p_{22}=0$.
If $p_{22}=0$, then $\omega = \alpha_2^{-1}\bar\alpha_2^{-1}|p_{12}|^2\omega^{2\bar2}$ and therefore we have to assume
$ \alpha_1\alpha_2^{-1}\bar\alpha_2^{-1}$ to be constant.
Up to rescale $p_{12}$, we may assume
$$ \alpha_1\alpha_2^{-1}\bar\alpha_2^{-1} \cdot \frac{|p_{12}|^2}{\left|p_{21}\right|^2} \;=\; 1 \;, $$
getting in this case the structure equations %%%(case 1 in Table \ref{table:cases-kt}\label{case1})
\begin{equation}\label{kt:case1}
 d
 \begin{pmatrix}
         \omega^1\\\omega^2
        \end{pmatrix}
 =
\begin{pmatrix} \left(A\omega^3+B \omega^{\bar3}\right) \wedge \omega^1\\
\left( C\omega^3+D \omega^{\bar3}\right) \wedge \omega^2\end{pmatrix}+
\begin{pmatrix}
  \omega^{2\bar2}
 \\[5pt]
 0
 \end{pmatrix}\;.
\end{equation}

On the other hand, if $p_{12}=0$, then $\omega =  \alpha_1^{-1}\bar\alpha_1^{-1}|p_{22}|^2\omega^{1\bar1}$ and we get the necessary assumption that
$\alpha_2 \alpha_1^{-1}\bar\alpha_1^{-1}$ is constant. Moreover, up to rescaling, we may assume
$$
\alpha_2 \alpha_1^{-1}\bar\alpha_1^{-1}\cdot \frac{|p_{22}|^2}{\left|p_{11}\right|^2} \;=\; 1 \;,
$$
which reduces the structure equations to %%%(case 2 in Table \ref{table:cases-kt}\label{case2})
\begin{equation}\label{kt:case2}
 d
\begin{pmatrix}
         \omega^1\\\omega^2
        \end{pmatrix}
 =
\begin{pmatrix} \left(A\omega^3+B\omega^{\bar3}\right) \wedge \omega^1\\
\left( C\omega^3+D \omega^{\bar3}\right) \wedge \omega^2\end{pmatrix}+
\begin{pmatrix}
  0
 \\[5pt]
 \omega^{1\bar1}
 \end{pmatrix} \;.
\end{equation}

%%%\item
In case $(C,D)\neq(0,0)$, we look at the term
$$ \alpha_2 \cdot \frac{|p_{22}|^2}{\left|p_{11}p_{22}-p_{12}p_{21}\right|^2} \cdot\alpha_2^{-1}\bar\alpha_2^{-1}|p_{12}|^2\omega^{2\bar2} \;=\; \bar\alpha_2^{-1} \cdot\frac{|p_{12}|^2|p_{22}|^2}{\left|p_{11}p_{22}-p_{12}p_{21}\right|^2} \omega^{2\bar2} $$
and we argue in the same way as before.
%%% it contains $\bar\alpha_3^{-1}$ depending on $z_1$ (unless $C=D=0$ but this should be trivial..) so either $p_{22}=0$ or $q_{12}=0$.
If $p_{22}=0$, then
we are reduced to the structure equations~\eqref{kt:case1}, whereas if
$p_{12}=0$, then
we are reduced to the structure equations~\eqref{kt:case2}.

Note that, with reference, e.g., to the second case~\eqref{kt:case2}, the Jacobi condition yields the equations
$$ A+\bar B-C \;=\; D-\bar C \;=\;0 \;. $$
Now the unimodularity condition is then equivalent to the equation
$\bar A + B = 0$. Finally, if $A\not=0$ then we can suppose that it is equal to 1 after rescaling $\omega^3$.
\end{proof}

For the sake of clearness, we summarize Proposition \ref{prop:eq-struttura-1}
and Proposition \ref{prop:eq-struttura-2} in the following statement.

\begin{theorem}\label{prop:eq-struttura}
Let $X=G/\Gamma$ be a $6$-dimensional solvmanifold endowed with a complex structure of splitting type.
%%%without canonical trivial bundle
Then, there is a co-frame $\{\omega^1,\omega^2,\omega^3\}$  of invariant $(1,0)$-forms satisfying the complex structure equations
\begin{equation}\label{split_C2}
\left\{\begin{array}{l}
 d\omega^1=A\,\omega^{13}+B\,\omega^{1\bar3},\\[3pt]
 d\omega^2= -(A+\bar B+\varepsilon)\,\omega^{23}+\varepsilon\,\omega^{2\bar3},\\[3pt]
d\omega^3=0
\end{array}\right.
\end{equation}
or
\begin{equation}\label{split_kt}
\left\{\begin{array}{l}
 d\omega^1=\varepsilon\,(\omega^{13}-\omega^{1\bar3}),\\[3pt]
d\omega^2=\omega^{1\bar1},\\[3pt]
 d\omega^3=0,
\end{array}\right.
\end{equation}
where $A,B\in\C$ and $\varepsilon\in\{0,1\}$.
\end{theorem}

% \begin{remark}
%  {\rm Let $G=\C\ltimes_\varphi\C^2$ endowed with a complex structure given by \eqref{eqs},  imposing the invariance of the 1-forms  $L_g^*(\omega^j)=\omega^j$ by an element $g\in G$ with coordinates $(a_1,a_2,a_3)\in\C^3$ we get the multiplication  law of the group
% $$
% (a_1,a_2,a_3)\cdot (z_1,z_2,z_3)=(e^{Aa_3+B\bar a_3}z_1+a_1,e^{Ca_3+D\bar a_3}z_2+a_2,z_3+a_3)
% $$
% and hence we can give a matrix representation of the corresponding Lie group endowed with the complex structure as
% $$
% G_{A,B,C,D}=\left\{\begin{pmatrix}e^{Az_3+B\bar z_3}&0&0&z_1\\0&e^{Cz_3+D\bar z_3}&0&z_2\\0&0&1&z_3\\0&0&0&1\end{pmatrix}\ |\ z_1,z_2,z_3\in\C\right\}.
% $$
% Similary, if $G=\C\ltimes_\varphiKT$ is endowed with a complex structure given by \eqref{split_kt}  the multiplication  law of the group is given by
% $$
% (a_1,a_2,a_3)\cdot (z_1,z_2,z_3)=(z_1+a_1,e^{\epsilon(-a_1+\bar a_1)}z_2+a_2,z_3+a_3+\bar a_2e^{\epsilon(-a_1+\bar a_1)}z_2)
% $$
% and this yields the following matrix representation:
% $$
% G_\epsilon=\left\{\begin{pmatrix}1&0&0&z_1\\0&e^{\epsilon(-z_1+\bar z_1)}&0&z_2\\0&\bar z_2e^{\epsilon(-z_1+\bar z_1)}&1&z_3\\0&0&0&1\end{pmatrix}\ |\ z_1,z_2,z_3\in\C\right\}.
% $$
% It is clear that the complex manifolds $G_{A,B,C,D}$ and $G_\epsilon$ are both biholomorphic to~$\C^3$. }
% \end{remark}

% \newpage

\medskip

\begin{remark}\label{remark-trivialbundle}
{\rm
We note that for a complex structure in~\eqref{split_C2},
the canonical bundle is holomorphically trivial if and only if $B=-\varepsilon$.
Indeed, by \cite[Proposition 2.1]{fino-otal-ugarte}, since the complex structure is left-invariant,
a nowhere vanishing holomorphic $(3, 0)$-form on $X=G/\Gamma$
is necessarily invariant, but a direct calculation shows that $d\omega^{123}=(B+\varepsilon)\omega^{123 \bar 3}$.
%Bellow we show that in this case the solvmanifolds have underlying Lie algebra
%${\mathfrak s}_4$, ${\mathfrak s}_8^{\alpha}$ or ${\mathfrak s}_{12}$,
%accordingly to \cite[Theorem 2.8]{fino-otal-ugarte}.
%%%
%%% ${\mathfrak s}_4 = {\mathfrak g}_1$,
%%% ${\mathfrak s}_8^{\alpha} = {\mathfrak g}_2^{\alpha}$
%%% ${\mathfrak s}_{12} = {\mathfrak g}_8$
%%%
Similarly, for a complex structure in~\eqref{split_kt}, one has that $d\omega^{123}=-\varepsilon\, \omega^{123 \bar 3}$,
so the canonical bundle is holomorphically trivial if and only if $\varepsilon=0$.
%%%i.e. $X$ is a nilmanifold.
We show below which are the Lie algebras underlying such solvmanifolds.
}
\end{remark}

\subsection{Six-dimensional solvable Lie algebras with complex structures of splitting type}\label{subsec:classification-algebras}

In this section we determine the $6$-dimensional real Lie algebras underlying the reduced equations of splitting-type complex structures
obtained in the previous section. For simplicity, we introduce the following definition.

\begin{definition}\label{def-alg-Lie-splitting-type}
{\rm  We will say that}
$\frg$ admits a complex structure of splitting type
{\rm if $\frg$ is a real Lie algebra underlying the complex
equations~\eqref{split_C2} or~\eqref{split_kt} in Theorem~\ref{prop:eq-struttura}.}
\end{definition}

Recall that those Lie algebras underlying the complex equations~\eqref{split_kt} correspond to Lie groups of the form $\C\ltimes_\varphi KT$,  whereas the Lie algebras underlying~\eqref{split_C2} correspond to
$\C^2\ltimes_\varphi\C$ or $\C\ltimes_\varphi\C^2$.

%Recall that all the Lie algebras admitting a splitting-type complex structure are unimodular and solvable,
%and those underlying the complex equations~\eqref{split_kt} correspond to Lie groups of the form $\C\ltimes_\varphi KT$,
%whereas the Lie algebras underlying~\eqref{split_C2} correspond to
%$\C^2\ltimes_\varphi\C$ or $\C\ltimes_\varphi\C^2$.

The main result in this section is the following theorem.

\begin{theorem}\label{thm:main-thm}
Let $\frg$ be a unimodular (non-nilpotent) solvable Lie algebra of dimension $6$.
Then, $\frg$ admits a complex structure of splitting type if and only if it is isomorphic to one in the following list:

\begin{itemize}
\item[] $\mathfrak s_1= (e^{23}, e^{34}, -e^{24},0,0,0)$,\\[-5pt]
\item[] $\mathfrak s_2= (0, -e^{13}, e^{12}, 0,0,0)$,\\[-5pt]
\item[] $\mathfrak s_3= (0, -e^{13}, e^{12}, 0,-e^{46},e^{45})$,\\[-5pt]
\item[] $\mathfrak s_4=(e^{15}, -e^{25}, -e^{35},e^{45},0,0)$,\\[-5pt]
\item[] $\mathfrak s_5^{\alpha}=(e^{15}, e^{25}, -e^{35}+\alpha\,e^{45},-\alpha\,e^{35}-e^{45},0,0)$,\quad $\alpha>0$,\\[-5pt]
\item[] $\mathfrak s_6^{\alpha,\beta}= (\alpha\,e^{15} + e^{25}, -e^{15}+\alpha\,e^{25}, -\alpha\,e^{35} + \beta\,e^{45}, -\beta\,e^{35}-\alpha\,e^{45},0,0)$,\\ \phantom{$\mathfrak s_6^{\alpha,\beta}=$} $\alpha>0,\ \, 0<\beta<1$,\\[-5pt]
\item[] $\mathfrak s_7^{\alpha}= (e^{25}, -e^{15}, \alpha\,e^{45}, -\alpha\,e^{35},0,0)$,\quad $0<\alpha\leq 1$,\\[-5pt]
\item[] $\mathfrak s_8^{\alpha}= (\alpha\,e^{15} + e^{25}, -e^{15}+\alpha\,e^{25}, -\alpha\,e^{35} + e^{45}, -e^{35}-\alpha\,e^{45},0,0)$,\quad $\alpha> 0$,\\[-5pt]
\item[] $\mathfrak s_{9}= (-e^{16}, -e^{26}, e^{36}-e^{45} ,e^{35}+e^{46},0,0)$,\\[-5pt]
\item[] $\mathfrak s_{10}^{\alpha, \beta}=(e^{15}+\beta\,e^{16}-e^{26}, e^{16}+e^{25}+\beta\,e^{26}, -e^{35}-\beta\,e^{36}-\alpha\,e^{45}, \alpha\,e^{35}-e^{45}-\beta\,e^{46},0,0)$,\\
\phantom{$\mathfrak s_{10}^{\alpha, \beta}=$} $\alpha\neq 0$,\quad $\beta\in\mathbb R$,\\[-5pt]
\item[] $\mathfrak s_{11}^{\alpha}=(e^{16}-e^{25}, e^{15}+e^{26}, -e^{36}-\alpha\,e^{45}, \alpha\,e^{35}-e^{46},0,0)$,\quad $\alpha\in (0,1)$,\\[-5pt]
\item[] $\mathfrak s_{12}=(e^{16}-e^{25}, e^{15}+e^{26}, -e^{36}+e^{45}, -e^{35}-e^{46},0,0)$.\\[-5pt]
\end{itemize}
\end{theorem}

Here we follow the notation in \cite{salamon}. For example, by writing $(e^{23}, e^{34}, -e^{24},0,0,0)$ we mean that there exists a basis $\{e^1, \dots, e^6\}$ of the dual of the Lie algebra satisfying $de^1 = e^2\wedge e^3$, $de^2 = e^3\wedge e^4$, $de^3 = -e^2\wedge e^4$, and $de^4 = de^5 = de^6=0$. 

\medskip

\begin{remark}
{\rm For detailed explanations on the values of the parameters in the list above, see Appendix~\ref{apendice}.
}
\end{remark}

Let us start by determining the Lie algebras underlying the equations~\eqref{split_kt}.

\begin{proposition}\label{prop_KT}
The Lie algebras $\frg$ that admit a complex structure of splitting type corresponding to $\C\ltimes_\varphi KT$
are $(0,0,0,0,0,e^{12})$ if $\frg$ is nilpotent, and $\mathfrak s_1$ if $\frg$ is solvable but non-nilpotent.
\end{proposition}

\begin{proof}
It is clear that one obtains $(0,0,0,0,0,e^{12})$ if $\varepsilon=0$ in~\eqref{split_kt}.
%%%(we recall that $\frh_8 = (0,0,0,0,0,e^{12})$ in the notation of \cite{salamon}).
For $\varepsilon=1$,
if we consider the basis $\{e^1,\ldots,e^6 \}$ given by $\omega^1=e^3-i\,e^2$, $\omega^2=2(e^5-i\,e^1)$ and $\omega^3=\frac12(e^6-i\,e^4)$,
then it is immediate to see that the real Lie algebra is $\mathfrak s_1$.
\end{proof}

%In Table~\ref{tabla} we summarize the results for equations~\eqref{split_kt}.

%%%%  TABLA CON OMEGAS UNA EN CADA LINEA %%%%%%%%%%%%

%\begin{center}
%\renewcommand{\arraystretch}{1.2}
%\begin{table}[h]
%\begin{tabular}{|c|l|c|}
%\hline
%&Real Basis&Real Lie algebra\\[4pt]
%\hline
%&$\omega^1=e^1+ie^2$&\\
%$\varepsilon=0$&$\omega^2=e^3+ie^4$&$\mathfrak h_{8}$\\
%&$\omega^3=e^5+i e^6$&\\[4pt]
%\hline
%&$\omega^1=\frac12 (-e^6+ie^4)$&\\
%$\varepsilon=1$&$\omega^2=e^3-ie^2$&$\mathfrak s_{1}$\\
%&$\omega^3=2(e^5-i e^1)$&\\[4pt]
%\hline
%\end{tabular}
%\medskip
%\caption{Lie algebras underlying equations~\eqref{split_kt}}
%\end{table}
%\medskip
%\end{center}

%%%%%%%%%%%%% TABLA CON OMEGAS EN UNA LINEA  %%%%%%%%

%\begin{center}
%\renewcommand{\arraystretch}{1.2}
%\begin{table}[!hb]
%\begin{tabular}{|c|l|c|}
%\hline
%&Real Basis&Real Lie algebra\\[4pt]
%\hline
%$\varepsilon=0$&$\omega^1=e^1+ie^2,\,\, \omega^2=e^3+ie^4,\,\, \omega^3=e^5+i e^6$&$\mathfrak h_{8}$\\[4pt]
%\hline
%$\varepsilon=1$&$\omega^1=e^3-ie^2,\,\, \omega^2=2(e^5-i e^1),\,\, \omega^3=\frac12 (e^6-ie^4)$&$\mathfrak s_{1}$\\[4pt]
%\hline
%\end{tabular}
%\medskip
%\caption{Lie algebras underlying equations~\eqref{split_kt}.}\label{tabla}
%\end{table}
%\end{center}

We divide the study of equations~\eqref{split_C2} according to the vanishing of coefficient $\varepsilon$. As a result we present several tables (see Tables~\ref{tabla2}, \ref{tabla3}, \ref{tabla_imA=imBneq0}, \ref{tabla5} and \ref{tabla6}).  There, the real basis $\{e^1,\ldots, e^6\}$ is the one that corresponds to the real structure equations in Theorem~\ref{thm:main-thm}.

%%%%%%%%%%%%%%%%%\subsection{Case \texorpdfstring{$\varepsilon = 0$}{}}

\begin{proposition}\label{prop2}
The %%%unimodular solvable $6$-dimensional
Lie algebras underlying equations~\eqref{split_C2} with $\varepsilon=0$
are $\mathfrak s_2$, $\mathfrak s_{9}$, $\mathfrak s_{10}^{\alpha, \beta}$, $\mathfrak s_{11}^{\alpha}$, $\mathfrak s_{12}$.%of types $N_{6,13},\, N_{6,15},\, N_{6,18}$.%$N_{6,13}^{0,-1,0,-1}$,\, $N_{6,15}^{-1,2b_2,\frac{1+2b_1}{2b_2},-\frac{1+2b_1}{2b_2}}$ (if $b_2\neq 0$),\, $N_{6,18}^{0,-(1+2b_1),-1}$ (if $b_2=0$), where $B=b_1+i\,b_2$.
\end{proposition}

\begin{proof}
%%%Using previous notation, $C=-\bar A - B-\varepsilon$.
Suppose first that $B=-\bar A$ in~\eqref{split_C2}, so we can suppose that $A\neq 0$.
Moreover, taking  $\{\omega'^1= \omega^1,\, \omega'^2= \omega^2,\,\omega'^3= A\omega^3\}$  we can suppose that $A=-B=1$.  If we set $\omega^1=e^3+ie^2,\, \omega^2=e^4+ie^5,\,\omega^3=e^6+\frac{i}{2}e^1$, then we obtain the structure equations of $\frs_2$.

On the other hand, if $B\neq -\bar A$, observe that we can normalize the coefficient in $\omega^{23}$ just
by taking a new basis $\{\omega'^1= \omega^1,\, \omega'^2= \omega^2,\,\omega'^3= -(A + \bar B)\omega^3\}$.  %%%By \eqref{uni}, $A+\bar B = -1$.

If we denote $\omega'^1=\alpha^1 + i\alpha^2,\, \omega'^2=\alpha^3 + i\alpha^4,\, \omega'^3=\alpha^5 + i\alpha^6$, then the real structure equations become $d\alpha^5 = d\alpha^6=0$ and
\begin{equation*}%\label{split_C2_real}
\left\{\begin{array}{lll}
d\alpha^1&=&-\alpha^{15}- 2\,\Imag B\,\alpha^{25}+(1+2\,\Real B)\,\alpha^{26},\\[5pt]
d\alpha^2&=&2\,\Imag B\,\alpha^{15}- (1+2\,\Real B)\,\alpha^{16}-\alpha^{25},\\[5pt]
d\alpha^3&=&\alpha^{35}-\alpha^{46},\\[5pt]
d\alpha^4&=&\alpha^{36}+\alpha^{45}.\end{array}\right.
\end{equation*}

It is straightforward to see that if $B=-\frac12$, the Lie algebra is isomorphic to $\frs_{9}$ (take $e^i = \alpha^i,\, i=1,2,3,4,\, e^5 = \alpha^6,\, e^6 = \alpha^5$).  If we consider the real basis $e^1 = \alpha^3,\, e^2 = \alpha^4,\, e^3 = \alpha^1,\, e^4 = \alpha^2,\, e^5 = \alpha^6,\, e^6 = \alpha^5$, then the Lie algebra is isomorphic to $\frs_{12}$ if $B=0$ and  $\frs_{11}^{\alpha'}$, for $\alpha' = -1-2 B$%\footnote{segun la lista de algebras, $0<|\beta|\leq 1$, lo que indica que $B\in [-1, -\frac12)\cup (-\frac12, 0]$.  Asi, quedaria que $B\in[-1,0)-\{-\frac12\}$}
, if $B\in\R\setminus\{-\frac12, 0\}$.
Notice that $\alpha'\in \R\setminus\{-1, 0\}$ and hence the Lie algebra $\frs_{11}^{\alpha'}$
is isomorphic to the Lie algebra $\mathfrak s_{11}^{\alpha}$ for some $\alpha \in (0,1)$, 
as it appears in Theorem~\ref{thm:main-thm} (see Appendix~\ref{apendice} for details).
If $B=-1$, taking  $e^1 = -\alpha^3,\, e^2 = -\alpha^4,\, e^3 = \alpha^2,\, e^4 = \alpha^1,\, e^5 = \alpha^6,\, e^6 = \alpha^5$ we obtain the Lie algebra $\mathfrak s_{12}$.
Finally, if $\Imag B\neq 0$, then with respect to the real basis $e^1 = \alpha^3,\, e^2 = \alpha^4,\, e^3 = \alpha^1,\, e^4 = \alpha^2,\, e^5 = \alpha^5-\frac{1+2\,\Real B}{2\,\Imag B}\alpha^6,\, e^6 = \alpha^6$, we obtain $\frs_{10}^{\alpha, \beta}$ where $\alpha = 2\,\Imag B\neq 0$ and $\beta = \frac{1+2\,\Real B}{2\,\Imag B}$.
\end{proof}

In Table~\ref{tabla2} we summarize the results obtained in the previous proposition.

\renewcommand{\arraystretch}{1.3}

\renewcommand{\arraystretch}{1.3}
\begin{table}[!ht]
\resizebox{\textwidth}{!}{
\begin{tabular}{|c|c|c|l|c|}
\hline
\multicolumn{3}{|c|}{$A,B \in \mathbb{C}$} & Real basis $\{e^1,\ldots,e^6\}$ & Lie algebra\\
\hline
\multicolumn{3}{|c|}{\multirow{2}{*}{$A=-\bar B\not=0$}}&$\omega^1=e^3+ie^2,\, \omega^2=e^4+ie^5$,&\multirow{2}{*}{$\frs_2$}\\
\multicolumn{3}{|c|}{}&$\,\omega^3=e^6+\frac{i}{2}e^1$&\\
\cline{1-5}
\multirow{9}{*}{$A=-1-\bar B$}&\multirow{8}{*}{$B\in\R$}&\multirow{2}{*}{$B=-1$}&$\omega^1 = e^4 + ie^3,\, \omega^2 = -e^1 - ie^2,$&\multirow{2}{*}{$\frs_{12}$}\\
&&&$\omega^3 = e^6+ ie^5$&\\ \cline{3-5}
&&\multirow{2}{*}{$B=-\frac12$}&$\omega^1 = e^1 + ie^2,\,\, \omega^2 = e^3 + ie^4,$&\multirow{2}{*}{$\frs_9$}\\
&\multirow{5}{*}{     }&&$\omega^3 = e^6+ ie^5$&\\ \cline{3-5}
&&\multirow{2}{*}{$B=0$}&$\omega^1 = e^3 + ie^4,\, \omega^2 = e^1 + ie^2,$&\multirow{2}{*}{$\frs_{12}$}\\
&&&$\omega^3 = e^6+ ie^5$&\\ \cline{3-5}
&&\multirow{2}{*}{$B\neq -1, -\frac12, 0$}&$\omega^1 = e^3 + ie^4,\, \omega^2 = e^1 + ie^2,$&$\frs_{11}^{\alpha'}$\\
&&&$\omega^3 = e^6+ ie^5$&$\alpha' = -1-2B$ \\ \cline{2-5}
&\multicolumn{2}{c|}{$\Imag B\neq 0$}&$\omega^1 = e^3 + ie^4,\, \omega^2 = e^1 + ie^2,$&$\frs_{10}^{\alpha, \beta}$\\
&\multicolumn{2}{c|}{   }&$\omega^3=\left(e^5+\frac{1+2\,\Real B}{2\,\Imag B}\,e^6\right)+ie^6$&$\alpha = 2\,\Imag B,\,
\beta=\frac{1+2\,\Real B}{2\,\Imag B}$\\[4pt]\hline
\end{tabular}}
\medskip
\caption{Lie algebras underlying equations~\eqref{split_C2} with $\varepsilon=0$
(Proposition~\ref{prop2}).} \label{tabla2}  
\end{table}

%%%%%%%%%%%%%%\subsection{Case \texorpdfstring{$\varepsilon = 1$}{}}

From now on, we focus on the equations~\eqref{split_C2} with $\varepsilon=1$. Let us consider
the basis of real 1-forms $\{\alpha^1,\ldots,\alpha^6\}$ given by 
\begin{equation}\label{base-alfa}\omega^1=\alpha^1 + i\alpha^2,\quad \omega^2=\alpha^3 + i\alpha^4,\quad \omega^3=\alpha^5 + i\alpha^6.\end{equation}
Hence, in terms of this basis the real structure equations become $d\alpha^5 = d\alpha^6 = 0$ and
\begin{equation}\label{split_C2_real_caso2}
\left\{\begin{array}{lll}
d\alpha^1&=&(\Real A + \Real B)\,\alpha^{15}- (\Imag A - \Imag B)\,\alpha^{16} \\[2pt] &&-(\Imag A + \Imag B)\,\alpha^{25}-(\Real A - \Real B)\,\alpha^{26},\\[5pt]
d\alpha^2&=&(\Imag A + \Imag B)\,\alpha^{15}+ (\Real A - \Real B)\,\alpha^{16}\\[2pt]&&+(\Real A + \Real B)\,\alpha^{25}-(\Imag A - \Imag B)\,\alpha^{26},\\[5pt]
d\alpha^3&=&-(\Real A + \Real B)\,\alpha^{35}+ (\Imag A - \Imag B)\,\alpha^{36}\\[2pt]&&+(\Imag A - \Imag B)\,\alpha^{45}+(2+\Real A + \Real B)\,\alpha^{46},\\[5pt]
d\alpha^4&=&-(\Imag A - \Imag B)\,\alpha^{35}- (2+\Real A + \Real B)\,\alpha^{36}\\[2pt]&&-(\Real A + \Real B)\,\alpha^{45}+(\Imag A - \Imag B)\,\alpha^{46}.
\end{array}\right.
\end{equation}

We need to consider different cases in order to identify all the possible real Lie algebras underlying
these equations.
Concretely, we focus our attention at the expression $\Imag A - \Imag B$ in~\eqref{split_C2_real_caso2} distinguishing three cases, namely:
$\Imag A = \Imag B = 0$, $\Imag A = \Imag B \neq 0$, or $\Imag A \neq  \Imag B$.

%\newpage

%%%%%%%%%%%%%%%%%%%%%%%%%%%%%%%%%%%%%%%%%%%%%%%%
%%%%%%%%%%%%%%%%%%%%%%%%%%%%%%%%%%%%%%%%%%%%%%%%
\subsubsection{Case \texorpdfstring{$\varepsilon = 1$, $\Imag A=\Imag B=0$}{}}
%%%%%%%%%%%%%%%%%%%%%%%%%%%%%%%%%%%%%%%%%%%%%%%%
%%%%%%%%%%%%%%%%%%%%%%%%%%%%%%%%%%%%%%%%%%%%%%%%

\begin{lemma}\label{lema imA=imB=0}
The %%%unimodular solvable $6$-dimensional
Lie algebras underlying equations~\eqref{split_C2} with $\varepsilon = 1$ and
$A, B\in \R$ are
$\mathfrak s_{2}$,
$\mathfrak s_{4}$, $\mathfrak s_{7}^{\alpha}$, $\mathfrak s_{9}$,
$\mathfrak s_{11}^{\alpha}$, $\mathfrak s_{12}$.%types $A_{5,7}\oplus\R$,\,\,$A_{5,17}\oplus\R$,\, $N_{6,13}$,\,$N_{6,18}$.
\end{lemma}

\begin{proof}
Imposing condition  $\Imag A=\Imag B=0$ in~\eqref{split_C2_real_caso2}, the equations simplify as %%%$d\alpha^5 = d\alpha^6 =~0$ and:
\begin{equation*}%\label{split_C2_balanced3_real}
\begin{array}{ll}
d\alpha^1=(A +B)\,\alpha^{15}-(A -B)\,\alpha^{26},\quad & d\alpha^3=-(A +B)\,\alpha^{35}+(2+A +B)\,\alpha^{46},\\[5pt]
d\alpha^2=(A - B)\,\alpha^{16}+(A +B)\,\alpha^{25},\quad &
d\alpha^4=- (2+A +B)\,\alpha^{36}-(A + B)\,\alpha^{45}.
\end{array}
\end{equation*}

Now, it suffices to consider different cases depending on the vanishing of the coefficients in the previous structure equations.
Concretely, we divide our analysis in the subcases $A=-B$, $A=B\neq 0$ and $A\neq \pm B$. The results appear in Table~\ref{tabla3}.
Notice that in the case of the Lie algebra $\frs_7^{\alpha'}$,
if $\alpha' = |A|>1$ then it is isomorphic to $\frs_7^{\alpha}$ with $\alpha=1/\alpha'$, so that $0<\alpha\leq 1$ according
to Theorem~\ref{thm:main-thm}.
Similarly, $\frs_{11}^{\alpha'}$
is isomorphic to the Lie algebra $\mathfrak s_{11}^{\alpha}$ for some $\alpha \in (0,1)$,
as it appears in Theorem~\ref{thm:main-thm}. 

 For each case in Table~\ref{tabla3}, we need to apply a change of real basis between the initial one $\{\alpha^1,\ldots, \alpha^6\}$ and the final one $\{e^1,\ldots, e^6\}$.  These changes are given simply by equalling the expression of $\omega^i$'s given in \eqref{base-alfa} and their corresponding expressions given in Table~\ref{tabla3}.

\end{proof}

\begin{center}
\renewcommand{\arraystretch}{1.3}
\begin{table}[h]\
\begin{tabular}{|c|c|l|c|}
\hline
\multicolumn{2}{|c|}{$A, B\in \R$}& Real basis $\{e^1,\ldots,e^6\}$ & Lie algebra\\
\hline
\multirow{5}{*}{$A=-B$}&\multirow{2}{*}{$A=0$}&$\omega^1=e^4 + ie^5,\,\, \omega^2=e^3+ie^2,$&\multirow{2}{*}{$\frs_2$}\\
&&$\omega^3=-e^6-\frac i2e^1$&\\[3pt]
\cline{2-4}

&\multirow{2}{*}{$A\neq 0$}&$\omega^1=-\frac{A}{|A|}\,e^3 + ie^4,$&$\frs_7^{\alpha'}$\\
&&$\omega^2=e^1+ie^2,\,\,\omega^3=e^6+\frac i2 e^5$&$\alpha' = |A|$\\[3pt]
\cline{1-4}

\multirow{4}{*}{$A=B$}&\multirow{2}{*}{$A=-1$}&$\omega^1=e^1 + ie^4,\,\, \omega^2=e^3+ie^2,$&\multirow{2}{*}{$\frs_4$}\\
&&$\omega^3=-\frac12 e^5+ i e^6$&\\[3pt]
\cline{2-4}

&\multirow{2}{*}{$A\neq 0, -1$}&$\omega^1=e^1 + ie^2,\,\, \omega^2=e^3+ie^4,$&\multirow{2}{*}{$\frs_9$}\\
&&$\omega^3=-\frac{1}{2A}\,e^6-\frac{i}{2(A+1)}\,e^5$&\\[3pt]
\cline{1-4}

\multirow{9}{*}{$A\neq \pm B$}&\multirow{2}{*}{$A=-1$}&$\omega^1=e^1 + ie^2,\, \omega^2=e^4+ie^3,$&\multirow{2}{*}{$\frs_{12}$}\\
&&$\omega^3=\frac{1}{B-1}\,e^6-\frac{i}{B+1}\,e^5$&\\[3pt]
\cline{2-4}

&\multirow{2}{*}{$B=-1$}&$\omega^1=e^1 + ie^2,\, \omega^2=e^3+ie^4,$&\multirow{2}{*}{$\frs_{12}$}\\
&&$\omega^3=\frac{1}{A-1}\,e^6+\frac{i}{A+1}\,e^5$&\\[3pt]
\cline{2-4}

&\multirow{2}{*}{$A+B=-2$}&$\omega^1=e^3+ie^4,\,\, \omega^2=e^1 + ie^2,$&\multirow{2}{*}{$\frs_9$}\\
&&$\omega^3=-\frac{1}{2}\,e^6+\frac{i}{2(A+1)}\,e^5$&\\[3pt]
\cline{2-4}

&$A+B\neq -2$&$\omega^1=e^1 + ie^2,\, \omega^2=e^3+ie^4,$&$\frs_{11}^{\alpha'}$\\
&$A, B\neq -1$&$\omega^3=\frac{1}{A+B}\,e^6+\frac{i}{A-B}\,e^5$&$\alpha' = \frac{2+A+B}{B-A}$\\[3pt]
\cline{1-4}

\end{tabular}
\medskip
\caption{Lie algebras underlying equations~\eqref{split_C2} with $\varepsilon=1$ and $\Imag A=\Imag B=0$
(Lemma~\ref{lema imA=imB=0}).}  \label{tabla3}
\end{table}  

\end{center}

%%%%%%%%%%%%%%%%%%%%%%%%%%%%%%%%%%%%%%%%%%%%%%%%
%%%%%%%%%%%%%%%%%%%%%%%%%%%%%%%%%%%%%%%%%%%%%%%%
\subsubsection{Case \texorpdfstring{$\varepsilon = 1$, $\Imag A=\Imag B\neq 0$}{}}
%%%%%%%%%%%%%%%%%%%%%%%%%%%%%%%%%%%%%%%%%%%%%%%%
%%%%%%%%%%%%%%%%%%%%%%%%%%%%%%%%%%%%%%%%%%%%%%%%

\begin{lemma}\label{lema imA=imB no0}
The %%%unimodular solvable $6$-dimensional
Lie algebras underlying equations~\eqref{split_C2} with $\varepsilon = 1$ and $\Imag A=\Imag B\neq 0$ are
$\mathfrak s_{3}$, $\mathfrak s_{5}^{\alpha}$, $\mathfrak s_{9}$, $\mathfrak s_{10}^{\alpha, \beta}$.%types $\mathfrak e(2)\oplus \mathfrak e(2)$,\, $N_{6,13}$,\,$N_{6,15}$.
\end{lemma}

\begin{proof}
Taking  $\Imag A=\Imag B\neq 0$, the equations~\eqref{split_C2_real_caso2} transform into %%%$d\alpha^5 = d\alpha^6=0$ and
$$
%%%\begin{equation}\label{split_C2_caso3_real}
\left\{\begin{array}{lll}
d\alpha^1&=&(\Real A + \Real B)\,\alpha^{15}-2\,\Imag A\,\alpha^{25}-(\Real A - \Real B)\,\alpha^{26},\\[5pt]
d\alpha^2&=&2\,\Imag A\,\alpha^{15}+ (\Real A - \Real B)\,\alpha^{16}+(\Real A + \Real B)\,\alpha^{25},\\[5pt]
d\alpha^3&=&-(\Real A + \Real B)\,\alpha^{35}+(2+\Real A + \Real B)\,\alpha^{46},\\[5pt]
d\alpha^4&=&- (2+\Real A + \Real B)\,\alpha^{36}-(\Real A + \Real B)\,\alpha^{45}.
\end{array}\right.
%%%\end{equation}
$$
We consider the following cases according to the vanishing of some coefficients in the equations above, namely
$\Real A = -\Real B$, $\Real A = \Real B\neq 0$ and $\Real A\neq \pm \Real B$,
obtaining the results that appear in Table~\ref{tabla_imA=imBneq0}.   The changes of basis between $\{\alpha^i\}_{i=1}^6$ and $\{e^i\}_{i=1}^6$ follow directly from Table~\ref{tabla_imA=imBneq0}, taking into account  \eqref{base-alfa}.
\end{proof}

\begin{center}
\renewcommand{\arraystretch}{1.2}
\begin{table}[!ht]
{\resizebox{\textwidth}{!}{
\begin{tabular}{|c|l|c|}
\hline
$\Imag A=\Imag B\neq 0$& Real basis $\{e^1,\ldots,e^6\}$& Lie algebra\\[4pt]
%%%$C=-(\Real A+\Real B + 1)$&&\\
\hline

\multirow{3}{*}{$\Real A=-\Real B$}&$\omega^1=e^2-ie^3,\quad \omega^2=e^5+ie^6,$&\multirow{3}{*}{$\mathfrak s_3$}\\[4pt]
&$\omega^3=\frac{1}{2\,\Imag A}(e^1-\Real A\,e^4) + \frac{i}{2} e^4$&\\[4pt]
\hline

\multirow{3}{*}{$\Real A=\Real B=-1$}&$\omega^1=-\frac{\Imag A}{|\Imag A|}e^3+ie^4,$&$\mathfrak s_5^{\alpha}$\\[4pt]
&$\omega^2=e^1+ie^2,\quad \omega^3=\frac 12\,e^5 + ie^6$&$\alpha= |\Imag A|$\\
\hline

\multirow{3}{*}{$\Real A=\Real B\neq 0, -1$}&$\omega^1=e^3+ie^4,\quad \omega^2=e^1+ie^2,$&$\mathfrak s_{10}^{\alpha, 0}$\\[5pt]
&$\omega^3=-\frac{1}{2\,\Real A}\,e^5-\frac {i}{2(\Real A+1)}\,e^6$&$\alpha=-\frac{\Imag A}{\Real A}$\\[5pt]%$N_{6,15}^{-1,-\frac{\Imag A}{\Real A},0,0}$\\[5pt]
\hline
$\Real A\neq \pm \Real B$&$\omega^1=e^3+ie^4,\quad \omega^2=e^1+ie^2,$&\multirow{3}{*}{$\mathfrak s_{9}$}\\[5pt]
$\Real A+\Real B=-2$&$\omega^3=-\frac{1}{2}\,e^6+\frac {i}{2(\Real A+1)}\,(e^5+\Imag A\, e^6)$&\\[5pt]%$N_{6,13}^{0,-1,0,-1}$\\
\hline

\multirow{3}{*}{$\Real A\neq \pm \Real B$}&$\omega^1=e^1+ie^2,$&$\mathfrak s_{10}^{\alpha,\beta}$\\[5pt]
\multirow{3}{*}{$\Real A + \Real B\neq-2$}&$\omega^2=e^3+ie^4,$&$\alpha=\frac{2\,\Imag A(2+\Real A+\Real B)}{\Real^2 A - \Real^2 B}$\\[5pt]%$N_{6,15}^{-1,\beta,\gamma,-\gamma}$\\[5pt]
&$\omega^3=\frac{1}{\Real A + \Real B}e^5+\frac{1}{2\,\Imag A}e^6 -i\,\frac{2\,\Imag A}{\Real^2 A - \Real^2 B}\,e^5$& $\beta=\frac{\Real A + \Real B}{2\,\Imag A}$\\[4pt]
\hline
\end{tabular}}}
\medskip
\caption{Lie algebras underlying equations~\eqref{split_C2} with $\varepsilon=1$ and $\Imag A=\Imag B\neq 0$
(Lemma~\ref{lema imA=imB no0}).}\label{tabla_imA=imBneq0}
\end{table}
\end{center}

%%%%%%%%%%%%%%%%%%%%%%%%%%%%%%%%%%%%%%%%%%%%%%%%
%%%%%%%%%%%%%%%%%%%%%%%%%%%%%%%%%%%%%%%%%%%%%%%%
\subsubsection{Case \texorpdfstring{$\varepsilon = 1$, $\Imag A\neq \Imag B$}{}}
%%%%%%%%%%%%%%%%%%%%%%%%%%%%%%%%%%%%%%%%%%%%%%%%
%%%%%%%%%%%%%%%%%%%%%%%%%%%%%%%%%%%%%%%%%%%%%%%%

Starting from \eqref{split_C2_real_caso2}, let us consider the new basis $\{\beta^1,\ldots,\beta^6\}$ given by
$$\beta^i = \alpha^i,\, i=1,2,3,4,\quad  \beta^5 = (\Imag A-\Imag B)\,\alpha^5,\quad \beta^6 = (\Imag A-\Imag B)\,\alpha^6.$$
In terms of this basis, the structure equations \eqref{split_C2_real_caso2} are %%%$d\beta^5 = d\beta^6=0$ and
%\begin{equation*}
%\left\{\begin{array}{lll}
%d\beta^1&=&\frac{\Real A + \Real B}{\Imag A-\Imag B}\,\beta^{15}- \beta^{16}-\frac{\Imag A + \Imag B}{\Imag A-\Imag B}\,\beta^{25}-\frac{\Real A - \Real B}{\Imag A-\Imag B}\,\beta^{26},\\[5pt]
%d\beta^2&=&\frac{\Imag A + \Imag B}{\Imag A-\Imag B}\,\beta^{15}+ \frac{\Real A - \Real B}{\Imag A-\Imag B}\,\beta^{16}+\frac{\Real A + \Real B}{\Imag A-\Imag B}\,\beta^{25}-\beta^{26},\\[5pt]
%d\beta^3&=&-\frac{\Real A + \Real B}{\Imag A-\Imag B}\,\beta^{35}+\beta^{36}+\beta^{45}+\frac{2+\Real A + \Real B}{\Imag A-\Imag B}\,\beta^{46},\\[5pt]
%d\beta^4&=&-\beta^{35}- \frac{2+\Real A + \Real B}{\Imag A-\Imag B}\,\beta^{36}-\frac{\Real A + \Real B}{\Imag A-\Imag B}\,\beta^{45}+\beta^{46},\\[5pt]
%d\beta^5 &=& d\beta^6=0.
%\end{array}\right.
%\end{equation*}
\begin{equation}\label{split_C2_real}
\left\{\begin{array}{lll}
d\beta^1&=&-\beta^1\wedge\left(\beta^6-\frac{\Real A + \Real B}{\Imag A-\Imag B}\,\beta^{5}\right) -
\beta^2\wedge\left(\frac{\Imag A + \Imag B}{\Imag A - \Imag B }\,\beta^{5}+\frac{\Real A - \Real B}{\Imag A - \Imag B }\,\beta^{6}\right),\\[5pt]
d\beta^2&=&-\beta^2\wedge\left(\beta^6-\frac{\Real A + \Real B}{\Imag A-\Imag B}\,\beta^{5}\right) +
\beta^1\wedge\left(\frac{\Imag A + \Imag B}{\Imag A - \Imag B }\,\beta^{5}+\frac{\Real A - \Real B}{\Imag A - \Imag B }\,\beta^{6}\right),\\[5pt]
d\beta^3&=&\beta^3\wedge\left(\beta^6-\frac{\Real A + \Real B}{\Imag A-\Imag B}\,\beta^{5}\right)
+\beta^{4}\wedge\left(\beta^{5}+\frac{2+\Real A + \Real B}{\Imag A-\Imag B}\,\beta^{6}\right),\\[5pt]
d\beta^4&=&-\beta^{3}\wedge\left(\beta^{5}+\frac{2+\Real A + \Real B}{\Imag A-\Imag B}\,\beta^{6}\right)
+\beta^4\wedge\left(\beta^6-\frac{\Real A + \Real B}{\Imag A-\Imag B}\,\beta^{5}\right).
\end{array}\right.
\end{equation}
%where $p= \frac{\Imag A + \Imag B}{\Imag A-\Imag B}$ and $q = \frac{\Real A - \Real B}{\Imag A-\Imag B}$.

We define the 1-forms
$$
\nu^5 = \beta^{5}+\frac{2+\Real A + \Real B}{\Imag A-\Imag B}\,\beta^{6},\quad  \nu^6 = \beta^6-\frac{\Real A + \Real B}{\Imag A-\Imag B}\,\beta^{5}.
$$
The linear dependence of $\nu^5$ and $\nu^6$ will play a key role in our study of the underlying Lie algebras.
Let us define
$$
%%\begin{eqnarray}\label{delta}
\begin{array}{lll}
\Delta=\Delta(A,B)& =& (\Imag A-\Imag B)^2 + (2+\Real A + \Real B)(\Real A + \Real B)\\[4pt]
\nonumber&=&|A|^2+|B|^2 + 2(\Real A + \Real B + \Real A\Real B - \Imag A\Imag B).
\end{array}
%%\end{eqnarray}
$$
It is straightforward to check that $\nu^5$ and $\nu^6$ are linearly independent if and only if $\Delta\not=0$.
In the following lemmata we study the cases $\Delta=0$ and $\Delta\not=0$.

\begin{lemma}\label{caso_dep}
%%%Let  $A,B\in\C$ such that  $\Imag A\neq \Imag B$ and  $\nu^5$ and $\nu^6$ are linearly dependent.
The Lie algebras underlying equations~\eqref{split_C2} with $\varepsilon = 1$, $\Imag A\neq \Imag B$ and $\Delta(A,B)=0$ are
$\mathfrak s_{5}^{\alpha},\,\mathfrak s_{6}^{\alpha,\beta},\, \mathfrak s_{8}^{\alpha},\,\mathfrak s_{10}^{\alpha,0}$.%types $A_{5,13}\oplus\R$,\, $A_{5,17}\oplus\R$,\, $N_{6,13}$,\,$N_{6,15}$,\,$N_{6,18}$.
\end{lemma}

\begin{proof}
Notice first that the condition $\Delta=0$ implies that $\Real A + \Real B\neq 0, -2$.
Since $\nu^5$ and $\nu^6$ are linearly dependent, we have that
$$\nu^5 = \theta\,\nu^6,\quad \text{where }\quad
\theta=\frac{2+\Real A + \Real B}{\Imag A-\Imag B} = -\frac{\Imag A-\Imag B}{\Real A + \Real B}\neq 0.$$
%%%This condition is equivalent to the following system:
%%%\begin{equation}\label{condicion_dependiente}
%%%\begin{cases}
%%%|A|^2 + |B|^2 + 2(\Real A + \Real B + \Real A \Real B - \Imag A \Imag B) =0,\\
%%%\Real A + \Real B\neq 0, -2,\\
%%%\Imag A\neq \Imag B.
%%%\end{cases}
%%%\end{equation}
Let us consider the new basis $\{\gamma^1,\ldots,\gamma^6\}$ given by
$\gamma^i = \beta^i$, $1\leq i\leq 5$, and $\gamma^6 =\nu^6= \beta^6+\frac1\theta\,\beta^5$.
With respect to this basis, the structure equations \eqref{split_C2_real} are %%%$d\gamma^5=d\gamma^6=0$ and
\begin{equation}\label{dif_gamma}
\left\{\begin{array}{lll}
d\gamma^1&=&-\gamma^{16} + \gamma^2\wedge
\left[\left(\frac{|B|^2-|A|^2}{(\Imag A - \Imag B)^2}\right)\gamma^{5}
-\left(\frac{\Real A - \Real B}{\Imag A-\Imag B}\right)\gamma^{6}\right],\\[5pt]
d\gamma^2&=&-\gamma^{26} - \gamma^1\wedge
\left[\left(\frac{|B|^2-|A|^2}{(\Imag A - \Imag B)^2}\right)\gamma^{5}
-\left(\frac{\Real A - \Real B}{\Imag A-\Imag B}\right)\gamma^{6}\right],\\[5pt]
d\gamma^3&=&\gamma^{36}+\theta\,\gamma^{46},\\[5pt]
d\gamma^4&=&\gamma^{46}-\theta\,\gamma^{36}.
\end{array}\right.
\end{equation}

In order to determine the Lie algebras underlying the equations~\eqref{dif_gamma},
we distinguish the cases when $|A| = |B|$ or $|A|\neq |B|$ (see Table~\ref{tabla5} for details).
Notice that the Lie algebras $\mathfrak s_5^{\alpha'}$, $\mathfrak s_6^{\alpha',\beta'}$
and $\mathfrak s_8^{\alpha'}$ in Table~\ref{tabla5} are isomorphic to the Lie algebras $\mathfrak s_5^{\alpha}$, $\mathfrak s_6^{\alpha,\beta}$
and $\mathfrak s_8^{\alpha}$ with the values of the parameters $\alpha$ and $\beta$ that appear in Theorem~\ref{thm:main-thm}.

 Observe that the relation between the bases $\{\gamma^i\}_{i=1}^6$ and $\{e^i\}_{i=1}^6$ can be deduced from the following diagram:
$$
\xymatrix{
\alpha
\ar@/^1.5pc/[rrr]^{\txt{Table 4}}
\ar[r]
& \beta \ar[r]
& \gamma \ar[r]
& e.
}
$$ \vspace{-0.3cm}
\end{proof}

\begin{center}
\renewcommand{\arraystretch}{1.3}
\begin{table}[!ht]
{\resizebox{\textwidth}{!}{
\begin{tabular}{|c|c|c|l|c|}
\hline
\multicolumn{3}{|c|}{$\Imag A\neq \Imag B$,\ \ $\Delta(A,B)=0$}& Real basis $\{e^1,\ldots,e^6\}$& Lie algebra\\[4pt]
\hline

\multirow{11}{*}{$|A|=|B|$}&\multicolumn{2}{c|}{\multirow{3}{*}{$B=\bar A$}}&$\omega^1=e^1+ie^2,\quad \omega^2 = e^3 + i e^4,$& $\mathfrak s_5^{\alpha'}$\\[5pt]
&\multicolumn{2}{c|}{}&$\omega^3=\frac{1}{2\,\Imag A}\left[e^6 - i\left(e^5+\frac{\Imag A}{1+\Real A}\,e^6\right)\right]$&$\alpha' = -\frac{1+\Real A}{\Imag A}$\\[4pt] \cline{2-5}

&\multirow{8}{*}{$B\neq \bar A$}&$B=-1$&$\omega^1=e^1+ie^2,\quad \omega^2 = e^4 + i e^3,$& $\mathfrak s_8^{\alpha'}$\\[5pt]
&&$\Imag A\neq 0$&$\omega^3=\frac{1}{\Imag A}e^6- \frac{i}{1+\Real A}(e^5+e^6)$&$\alpha' = \frac{\Imag  A}{1+\Real A}$\\[4pt]
 \cline{3-5}

&&$A=-1$&$\omega^1=e^1+ie^2,\quad \omega^2 = e^3 + i e^4,$& $\mathfrak s_8^{\alpha'}$\\[5pt]
&&$\Imag B\neq 0$&$\omega^3=\frac{-1}{\Imag B}e^6+ \frac{i}{1+\Real B}(e^5-e^6)$&$\alpha' = \frac{\Imag  B}{1+\Real B}$\\[4pt] \cline{3-5}

&&$\Real A\neq \Real B$&$\omega^1=e^1+ie^2,\quad \omega^2 = e^3 + i e^4,$& $\mathfrak s_6^{\alpha',\beta'}$\\[5pt]
&&$\Real A, \Real B\neq -1$&$\omega^3=\frac{1}{\Imag A-\Imag B}e^6 - i\frac{1}{\Real A - \Real B}e^5$&$\alpha'=\frac{\Imag A-\Imag B}{\Real A - \Real B}$\\[5pt]
&&$(\Imag A)(\Imag B)\neq 0$&$ \qquad + i\frac{\Real A + \Real B}{(\Imag A-\Imag B)^2}\,e^6$&$\beta'=\frac{-(2+\Real A + \Real B)}{\Real A-\Real B}$\\[4pt]
\cline{1-5}

\multicolumn{3}{|c|}{\multirow{4}{*}{$|A|\neq |B|$}} &$\omega^1=e^1+ie^2,\quad \omega^2=e^3+ie^4,$& \multirow{2}{*}{$\mathfrak s_{10}^{\alpha,0}$}\\[5pt]
\multicolumn{3}{|c|}{}&$\omega^3=\frac{\Real A-\Real B}{|A|^2-|B|^2}\,e^5+\frac{\Imag A-\Imag B}{|A|^2-|B|^2}\,e^6$&\\[5pt]
\multicolumn{3}{|c|}{}&$\phantom{\omega^3=}-i\frac{\Imag A+\Imag B}{|A|^2-|B|^2}\,e^5+i\frac{\Real A + \Real B}{|A|^2-|B|^2}\,e^6$&
$\alpha=\frac{2+\Real A + \Real B}{\Imag A-\Imag B}$\\[4pt]
\cline{1-5}

\end{tabular}}}

\medskip
\caption{Lie algebras underlying equations~\eqref{split_C2} with $\varepsilon=1$,
$\Imag A\neq \Imag B$ and $\Delta(A,B)=0$ (Lemma~\ref{caso_dep}).}\label{tabla5}
\end{table}  
\end{center}

\begin{lemma}\label{caso_indep}
%%%Let $A,B\in\C$ such that  $\Imag A\neq \Imag B$ and  $\nu^5$ and $\nu^6$ are linearly independent.
The Lie algebras underlying equations~\eqref{split_C2} with $\varepsilon = 1$, $\Imag A\neq \Imag B$ and $\Delta(A,B)\not=0$ are
$\mathfrak s_{9},\, \mathfrak s_{10}^{\alpha, \beta},\, \frs_{11}^{\alpha},\, \frs_{12}$.%types $A_{5,13}\oplus\R$,\, $A_{5,17}\oplus\R$,\, $N_{6,13}$,\,$N_{6,15}$,\,$N_{6,18}$.
\end{lemma}

\begin{proof}
Since $\Delta\not=0$, the 1-forms $\nu^5$ and $\nu^6$ are linearly independent. Hence, we consider the basis
$\{\nu^1,\ldots,\nu^6\}$ given by
$$\nu^i = \beta^i,\, i=1,2,3,4,\quad  \nu^5= \beta^{5}+\frac{2+\Real A + \Real B}{\Imag A-\Imag B}\,\beta^{6},\quad \nu^6 = \beta^6-\frac{\Real A + \Real B}{\Imag A-\Imag B}\,\beta^{5}.$$
The structure equations~\eqref{split_C2_real} transform into %%%$d\nu^5=d\nu^6=0$ and
$$
%%%\begin{equation}\label{ec_ind}
\left\{\begin{array}{lll}
d\nu^1&=&-\nu^{16}-\frac{\Imag A-\Imag B}{\Delta}\left(X\,\nu^{25} - Y\,\nu^{26}\right),\\[5pt]
d\nu^2&=&-\nu^{26}+\frac{\Imag A-\Imag B}{\Delta}\left(X\,\nu^{15} - Y\,\nu^{16}\right),\\[5pt]
d\nu^3&=&\nu^{36}+\nu^{45},\\[5pt]
d\nu^4&=&-\nu^{35}+\nu^{46},
\end{array}\right.
%%%\end{equation}
$$
where $$X=\frac{|A|^2 - |B|^2}{\Imag A - \Imag B},\quad\quad Y=2\,\frac{\Imag A (1+\Real B) + \Imag B (1+\Real A)}{\Imag A - \Imag B}.$$
%%%and
%%%\begin{eqnarray}\label{delta}
%%%\Delta& =& -[(\Imag A-\Imag B)^2 + (2+\Real A + \Real B)(\Real A + \Real B)]\\
%%%\nonumber&=&-[|A|^2+|B|^2 + 2(\Real A + \Real B + \Real A\Real B - \Imag A\Imag B)]\neq 0.
%%%\end{eqnarray}

Now, the study is divided according to the vanishing of coefficients $X$ and $Y$ (see Table~\ref{tabla6} for details).  For the sake of clarity, we see what happens when $X=Y=0$:
let us define $p=\frac{\Imag A + \Imag B}{\Imag A - \Imag B}$ and $q=\frac{\Real A - \Real B}{\Imag A - \Imag B}$, and   consider the following system of equations in variables $p$ and $q$: $$\begin{cases}
\begin{array}{l}
X = p(\Imag A - \Imag B) + q(\Real A + \Real B),\\
Y = p(2+\Real A + \Real B) - q(\Imag A - \Imag B).
\end{array}
\end{cases}$$
Observe that the determinant associated to the system is $-\Delta$.  Since $\Delta \neq 0$, if $X=Y=0$, the system has trivial solution and therefore $B=\bar A$ and, in particular, $\Delta = 4(|A|^2 + \Real A)\neq~0$.  

 Finally,  the relation between the bases $\{\nu^i\}_{i=1}^6$ and $\{e^i\}_{i=1}^6$ can be deduced from
$$
\xymatrix{
\alpha
\ar@/^1.5pc/[rrr]^{\txt{Table 5}}
\ar[r]
& \beta \ar[r]
& \nu \ar[r]
& e.}
$$
\vspace{-0.3cm}
\end{proof}

\begin{center}
\renewcommand{\arraystretch}{1.3}
\begin{table}[!ht]
{\resizebox{\textwidth}{!}{
\begin{tabular}{|c|c|c|l|c|}
\hline
\multicolumn{3}{|c|}{$\Imag A\neq \Imag B$,\ \ $\Delta(A,B)\not=0$}& Real basis $\{e^1,\ldots,e^6\}$ & Lie algebra\\[4pt]
\hline

\multirow{8}{*}{$|A|=|B|$}&\multicolumn{2}{c|}{\multirow{4}{*}{$Y=0$}}&$\omega^1=e^1+ie^2,\quad \omega^2=e^3+ie^4,$&\multirow{4}{*}{$\frs_9$}\\[4pt]
&\multicolumn{2}{c|}{}&$\omega^3=\frac{-\Imag A}{2 (|A|^2 + \Real A)}\left(e^5+\frac{1+\Real A}{\Imag A}\,e^6\right)$&\\[4pt]
&\multicolumn{2}{c|}{}&$\phantom{\omega^3=}- i \frac{\Imag A}{2 (|A|^2 + \Real A)} \left(\frac{\Real A}{\Imag A}\,e^5 - e^6\right)$&\\[4pt]  \cline{2-5}

&\multicolumn{2}{c|}{\multirow{4}{*}{$Y\neq 0$}}&$\omega^1=e^3+ie^4,\quad \omega^2=e^1-ie^2,$&\multirow{3}{*}{$\frs_{10}^{\alpha,0}$}\\[4pt]
&\multicolumn{2}{c|}{}&$\omega^3=\frac{\Imag A-\Imag B}{\Delta}\left(e^6-\frac{2+\Real A + \Real B}{\Imag A-\Imag B}\,e^5\right)$&\multirow{3}{*}{$\alpha=\frac{-Y(\Imag A-\Imag B)}{\Delta}$}\\[4pt]
&\multicolumn{2}{c|}{}&\phantom {$\omega^3=$} $+\frac{i(\Imag A-\Imag B)}{\Delta} \left(e^5 + \frac{\Real A + \Real B}{\Imag A-\Imag B}\,e^6\right)$&\\[4pt] \cline{1-5}

\multirow{12}{*}{$|A|\neq |B|$}&\multirow{8}{*}{$Y=0$}&\multirow{4}{*}{$\Delta = \pm(|A|^2 - |B|^2) $}&$\omega^1=e^3+ie^4,\quad \omega^2=e^1-ie^2,$&\multirow{4}{*}{$\frs_{12}$}\\[4pt]
&&&$\omega^3=\frac{\Imag A-\Imag B}{\Delta}\left(e^5-\frac{2+\Real A + \Real B}{\Imag A-\Imag B}\,e^6\right)$&\\[4pt]
&&&$\phantom{\omega^3=}+ \frac{i(\Imag A-\Imag B)}{\Delta}\left(\frac{\Real A + \Real B}{\Imag A-\Imag B}\,e^5 + e^6\right)$&\\[4pt] \cline{3-5}

&&\multirow{4}{*}{$\Delta \neq  \pm(|A|^2 - |B|^2) $}&$\omega^1=e^3+ie^4,\quad \omega^2=e^1-ie^2,$&\multirow{3}{*}{$\frs_{11}^{\alpha'}$}\\[4pt]
&&&$\omega^3=\frac{\Imag A-\Imag B}{\Delta}\left(e^5-\frac{2+\Real A + \Real B}{\Imag A-\Imag B}\,e^6\right)$&\multirow{3}{*}{$\alpha' = \frac{-X(\Imag A-\Imag B)}{\Delta}$}\\[4pt]
&&&$\phantom{\omega^3=}+ \frac{i(\Imag A-\Imag B)}{\Delta}\left(\frac{\Real A + \Real B}{\Imag A-\Imag B}\,e^5 + e^6\right)$&\\[4pt] \cline{2-5}

&\multicolumn{2}{c|}{\multirow{5}{*}{$Y\neq 0$}}&$\omega^1=e^1-ie^2,\quad \omega^2=e^3+ie^4,$&\multirow{3}{*}{$\frs_{10}^{\alpha, \beta}$}\\[4pt]
&\multicolumn{2}{c|}{}&$\omega^3=\frac{X(2+\Real A + \Real B)-Y(\Imag A-\Imag B)}{X\, \Delta}\,e^5$ &\multirow{3}{*}{$\alpha=\frac YX$}\\[4pt]
&\multicolumn{2}{c|}{}&$\phantom{\omega^3 = }+ \frac{2+\Real A+\Real B}{Y(\Imag A-\Imag B)}\,e^6  - \frac{i}{Y}\,e^6$&\multirow{3}{*}{$\beta = \frac{\Delta}{Y(\Imag A-\Imag B)}$}\\[5pt]
&\multicolumn{2}{c|}{}&\qquad $ - i\,\frac{X(\Imag A-\Imag B) + Y(\Real A + \Real B)}{X\,\Delta}\,e^5$&\\[4pt]
\hline
\end{tabular}}}
\medskip
\caption{Lie algebras underlying equations~\eqref{split_C2} with $\varepsilon=1$,
$\Imag A\neq \Imag B$ and $\Delta(A,B)\not=0$ (Lemma~\ref{caso_indep}).}\label{tabla6}
\end{table}
\end{center}

The previous lemmata provide the following

\begin{proposition}\label{prop3}
The unimodular solvable $6$-dimensional Lie algebras underlying equations~\eqref{split_C2} with $\varepsilon = 1$ are $\mathfrak s_2$,
$\mathfrak s_3$, $\mathfrak s_4$, $\mathfrak s_5^{\alpha}$, $\mathfrak s_6^{\alpha,\beta}$,
$\mathfrak s_7^{\alpha}$, $\mathfrak s_{8}^{\alpha}$, $\mathfrak s_{9}$, $\mathfrak s_{10}^{\alpha, \beta}$, $\mathfrak s_{11}^{\alpha}$,
$\mathfrak s_{12}$.
\end{proposition}

\bigskip

As a consequence of the previous propositions, we prove the main result of this section:

\bigskip

\noindent
\emph{Proof of Theorem~\ref{thm:main-thm}.} 
The ``only if'' part of the theorem follows from Propositions~\ref{prop_KT}, \ref{prop2} and~\ref{prop3}.

For the proof of the ``if'' part, we must show that all the Lie algebras in the list admit a
splitting-type complex structure.
This is clear for the Lie algebras ${\mathfrak s}_{1}$, ${\mathfrak s}_{2}$,
${\mathfrak s}_{3}$, ${\mathfrak s}_{4}$,
${\mathfrak s}_{9}$ and ${\mathfrak s}_{12}$ from Proposition~\ref{prop_KT}
and Tables~\ref{tabla2}, \ref{tabla3}
and~\ref{tabla_imA=imBneq0}. The remaining Lie algebras in the list depend on parameters,
so we will show next
particular appropriate values of $A$ and $B$ that define a complex structure of splitting type
on each one of the Lie
algebras ${\mathfrak s}_{5}^{\alpha}$, ${\mathfrak s}_{6}^{\alpha,\beta}$,
${\mathfrak s}_{7}^{\alpha}$,
${\mathfrak s}_{8}^{\alpha}$,
${\mathfrak s}_{10}^{\alpha,\beta}$ and ${\mathfrak s}_{11}^{\alpha}$
in the list.

For the Lie algebra ${\mathfrak s}_{5}^{\alpha}$, $\alpha>0$, we consider $A$ and $B$ given by
$A=B=-1+i\,\alpha$.
These values of the parameters $A$ and $B$ lie in Table~\ref{tabla_imA=imBneq0},
since $\Imag A=\Imag B=\alpha\neq 0$ and $\Real A=\Real B=-1$.
Hence,
the (1,0)-forms $\omega^1=-\frac{\Imag A}{|\Imag A|}e^3+ie^4=-e^3+ie^4$,
$\omega^2=e^1+ie^2$,
$\omega^3=\frac 12\,e^5 + ie^6$ define a splitting-type complex structure
on ${\mathfrak s}_{5}^{\alpha}$
according to Table~\ref{tabla_imA=imBneq0}.

For the other Lie algebras the argument is similar.
We show below particular appropriate values of $A,B$
and the table where the corresponding basis
of (1,0)-forms is given:

\vskip.1cm

\noindent
- For ${\mathfrak s}_{6}^{\alpha,\beta}$, $\alpha>0$ with $\alpha \not=1$, $0<\beta<1$,
it suffices to take $A=\frac{-2}{1+\beta} + i \frac{1-\alpha^2}{\alpha(1+\beta)}$ and
$B=i \frac{1+\alpha^2}{\alpha(1+\beta)}$ in Table~\ref{tabla5};

\vskip.1cm

\noindent
- For ${\mathfrak s}_{6}^{1,\beta}$, $0<\beta<1$, we can take
$A=-\frac{1+\beta}{1+\beta^2} + i \frac{1-\beta}{1+\beta^2}$
and $B=-\frac{1-\beta}{1+\beta^2} + i \frac{1+\beta}{1+\beta^2}$ in Table~\ref{tabla5};

\vskip.1cm

\noindent
- For ${\mathfrak s}_{7}^{\alpha}$, $0<\alpha\leq 1$, we take $A=-B=\alpha$ in Table~\ref{tabla3};

\vskip.1cm

\noindent
- For ${\mathfrak s}_{8}^{\alpha}$, $\alpha > 0$, we take
$A=\frac{1}{1+\alpha^2} (1- \alpha^2 +2i \alpha)$
and $B=-1$ in Table~\ref{tabla5};

\vskip.1cm

\noindent
- For ${\mathfrak s}_{10}^{\alpha,\beta}$, $\alpha \not=0$, $\beta \in \mathbb{R}$, we can take
$A=-1-\bar B$
with $B=\frac{1}{2} (\alpha\beta -1 +i \beta)$ in Table~\ref{tabla2};

\vskip.1cm

\noindent
- Finally, for the Lie algebra ${\mathfrak s}_{11}^{\alpha}$, $\alpha\in (0,1)$, we take
$A=-1-B$
with $B=-\frac{1}{2} ( 1 + \alpha)$ in Table~\ref{tabla2}.
\hfill
$\Box$

\vskip.4cm

\begin{remark}\label{remark-comparacion}
{\rm
In view of Remark~\ref{remark-trivialbundle}, a 6-dimensional unimodular (non-nilpotent) solvable Lie algebra admits a complex structure of
splitting type with a non-zero closed $(3, 0)$-form if and only if $B=-\varepsilon$ in the structure equations~\eqref{split_C2}.
Looking at the tables above, it is easy to check that this condition is satisfied if and only if the
Lie algebra is isomorphic to $\mathfrak s_{4}$, $\mathfrak s_{7}^{1}$, $\mathfrak s_{8}^{\alpha}$ or $\mathfrak s_{12}$,
which is in accord with  \cite[Theorem 2.8]{fino-otal-ugarte} (notice that these Lie algebras correspond, respectively, to
the Lie algebras labeled as $\frg_1$, $\frg_2^{\alpha}$ and $\frg_8$ in \cite{fino-otal-ugarte}).
%%%That is to say, the unimodular (non-nilpotent) solvable Lie algebras of dimension
%%%6 admitting a complex structure of splitting type with a nonzero closed $(3, 0)$-form are
%%%$\mathfrak s_{4}$, $\mathfrak s_{8}^{\alpha}$ and $\mathfrak s_{12}$.

On the other hand, the Lie algebras obtained in Theorem~\ref{thm:main-thm}
appear with different notations in previous papers. Next, we make explicit
the correspondence with \cite{bock,turkowski}:

\begin{tabular}{llll}
$\mathfrak s_1=\frg_{4,9}^0\oplus \R^2$,&
$\mathfrak s_2=\frg_{3,5}^0\oplus \R^3$,&
$\mathfrak s_3=\frg_{3,5}^0\oplus\frg_{3,5}^0$,&
$\mathfrak s_4=\frg_{5,7}^{-1,-1,1}\oplus \R$,\\
$\mathfrak s_5^{\alpha}=\frg_{5,13}^{1,-1,\alpha}\oplus \R$,&
$\mathfrak s_6^{\alpha,\beta}=\frg_{5,17}^{\alpha, -\alpha, \beta}\oplus \R $,&
$\mathfrak s_7^{\alpha}=\frg_{5,17}^{0,0, \alpha}\oplus \R$,&
$\mathfrak s_8^{\alpha}=\frg_{5,17}^{\alpha, -\alpha, 1}\oplus \R$,\\
$\mathfrak s_{9}=N_{6,13}^{0,-1,0,-1}$,&
$\mathfrak s_{10}^{\alpha,\beta}=N_{6,15}^{-1,\alpha,\beta,-\beta}$,&
$\mathfrak s_{11}^{\alpha}=N_{6,18}^{0,\alpha,-1}$,&
$\mathfrak s_{12}=N_{6,18}^{0,-1,-1} $.
\end{tabular}

It turns out that the only Lie algebra that is completely solvable is $\mathfrak s_4$.}
\end{remark}

\vskip.2cm

\begin{remark}\label{remark-lattices} 
{\rm     
As regards solvmanifolds of splitting type, we notice that the condition (5) in Definition~\ref{def-splitting-type} is satisfied by the Kodaira-Thurston manifold;
see~\cite{CF,CFGU2,R1,R2} for general results on the Dolbeault cohomology of nilmanifolds.
Therefore, we need to study the existence of lattices in the connected and simply-connected solvable Lie groups $G_k$ corresponding
to the Lie algebras $\mathfrak{s}_{k}$ in Theorem~\ref{thm:main-thm}.  %If there is a lattice $\Gamma_k$  in $G_k$ then we get a compact complex solvmanifold $G_k/\Gamma_k$ endowed with a complex structure of splitting type.
The Lie groups $G_{1}$, $G_{2}$ and $G_3$ admit lattices (see \cite[Table 8]{bock}). Also $G_{4}$ admits lattices
by \cite{fino-otal-ugarte,otal-phd}. Moreover, by \cite[page 13]{console-macri} we have:
\begin{itemize}
 \item $G_{6}^{\alpha,\beta}$ admits lattices if and only if $\beta=\frac{r_1}{r_2}\in\Q$ and $\alpha$ satisfies $\exp(2\pi \alpha^{-1}r_2)+\exp(-2\pi \alpha^{-1}r_2)\in\Z$, that is, $\alpha$ is the form $\alpha_n:=\frac{2\pi r_2}{\log\left(\frac{1}{2}\left(n\pm\sqrt{n^2-4}\right)\right)}$ with $n\in\N$;
 \item $G_{7}^{\alpha}$ admits lattices if and only if $\alpha\in\Q$;
 \item $G_{8}^{\alpha}$ admits lattices if and only if $\exp(2\pi \alpha^{-1})+\exp(-2\pi p\alpha^{-1})\in\Z$, that is, for any $\alpha$ of the form $\alpha_n:=\frac{2\pi}{\log\left(\frac{1}{2}\left(n\pm\sqrt{n^2-4}\right)\right)}$ with $n\in\N$.
%%%%%% \item $G_{12}$: it admits lattices by \cite{yamada} (this case corresponds to the Nakamura manifold).
\end{itemize}

In Proposition~\ref{G5} below we show the existence of lattices for a countable family of $G_{5}^{\alpha}$.
Note that the results on the existence of lattices are consistent with \cite[Proposition 8.7]{Witte}, where it is shown that only
countably many non-isomorphic simply-connected solvable Lie groups admit a lattice.
Therefore, one cannot expect a lattice to exist on $G_{5}^{\alpha}$, $G_{6}^{\alpha,\beta}$, $G_{7}^{\alpha}$ or $G_{8}^{\alpha}$ for every value of $\alpha,\beta$, and so, in this sense, our proposition below completes the cases when the Lie algebra is decomposable.

The indecomposable case is more difficult to treat, but in Section~\ref{sec:nakamura} we will provide
explicit lattices
on the Lie group associated to $\mathfrak{s}_{12}$
(which is the Lie algebra underlying the Nakamura manifold~\cite{nakamura}, see also \cite{yamada})
with interesting properties with respect to
the $\partial\db$-Lemma.
}
\end{remark}

%%%%%%%%%%%%%%%%%%%%%%%%%%%%%%%%%%%%%%%%%%%
%%%%%%%%%%%%%%%%%%%%%%%%%%%%%%%%%%%%%%%
%%%%%%%%%%%%%%%%%%%%%%%%%%%%%%%%%%%%%%%%%%%

%%%\begin{comment}

%%%For the cases $\mathfrak{s}_{9}$, $\mathfrak{s}_{10}^{\beta.\gamma}$, and $\mathfrak{s}_{11}^{\beta}$,
%%%up to our knowledge, the existence of lattices is not known.

%%%As regards $\mathfrak{s}_{5}^{\gamma}$, it is possible to construct lattices for certain values of the parameter $\gamma$.

%For naming the Lie algebras we follow the notation in \cite{bock}.
%For naming the Lie algebras we follow the notation in \cite{bock,turkowski}.

\begin{proposition}\label{G5}
There is a countable family $\{\alpha_{s,n}\}\subset\R^+$   %%%$\{\alpha_{s,n}\}_{ s\neq0,n\geq3}\subset\R$
such that the connected and simply-connected Lie group $G_{5}^{\alpha_{s,n}}$ admits a lattice.
\end{proposition}

\begin{proof}
The Lie algebra of $G_{5}^{\alpha}$, $\alpha>0$, can be written as $\frs_{5}^{\alpha}=\frg_{5,13}^{1,-1,\alpha}\oplus\R$ with $\frg_{5,13}^{1,-1,\alpha}=\R\ltimes_{\text{ad}_{e_5}}\R^{4}$.
Since the simply-connected Lie group $H_{\alpha}$ corresponding to $\frg_{5,13}^{1,-1,\alpha} $ is almost-nilpotent \cite{bock},
it admits a lattice if and only if there exists
$\tau\neq 0$ such that the matrix $\exp(\tau\,\text{ad}_{e_5})$ belongs to the conjugation class of an integer matrix.
We have (see~\cite[p. 41]{bock}) that $\exp(t\,\text{ad}_{e_5})$ is given by
\begin{equation}\label{matriz_exp}
 \exp(t\,\text{ad}_{e_5})=\begin{pmatrix}
                 e^{-t}&0&0&0\\
                 0&e^{-t}&0&0\\
                 0&0&e^{t}\cos\alpha t&-e^{t}\sin\alpha t\\
                 0&0&e^{t}\sin\alpha t&e^{t}\cos\alpha t
                \end{pmatrix}.
\end{equation}
Let $\tau\neq 0$ be such that  $\sin\alpha \tau=0$, that is, $\tau=\frac{s\pi}{\alpha}$ with $0\neq s\in\Z$. In this case the matrix~\eqref{matriz_exp} is diagonal and its characteristic
polynomial is
\begin{equation}\label{pol}
p(\lambda)=\left(\lambda^2-(e^{-\tau}+(-1)^se^{\tau})\lambda +(-1)^s\right)^2.
\end{equation}
Now, if $\exp(\tau\,\text{ad}_{e_5})$ lies in the conjugation class of an integer matrix, then $p(\lambda)\in\Z[\lambda]$, that is, $e^{-\tau}+(-1)^se^{\tau}=n$, for some $n\in\Z$. Solving this equation, we get
$$
\tau_{s,n}=-\log\left(\frac{n+\sqrt{n^2-4(-1)^s}}{2}\right),\quad
\alpha_{s,n}=-\frac{s\pi}{\log\left(\frac{n+\sqrt{n^2-4(-1)^s}}{2}\right)},\text{ for }n\geq3.
$$
Substituting these values in \eqref{pol}, we get $p(\lambda)= \left(\lambda^2-n\lambda +(-1)^s\right)^2\in\Z[\lambda]$,
which is also the characteristic polynomial of the integer matrix
\begin{equation*}
B_s=\begin{pmatrix}
                 0&(-1)^{s+1}&0&0\\
                 1&n&0&0\\
                 0&0&0&(-1)^{s+1}\\
                 0&0&1&n
                \end{pmatrix}\in\GL(4,\Z).
\end{equation*}
In addition, it turns out that $Q \exp(\tau_{s,n}\,\text{ad}_{e_5})\, Q^{-1}=B_s$, where
$$
Q=\begin{pmatrix}
                 0&\beta_+&0&\beta_-\\
                 0&1&0&1\\
                 \beta_+&0&\beta_-&0\\
                 1&0&1&0
                \end{pmatrix}, \qquad \beta_{\pm}=\frac{1}{2}\left(-n\pm\sqrt{n^2-4(-1)^s}\right),$$
concluding the proof.
\end{proof}

%%%\end{comment}

%%%%%%%%%%%%%%%%%%%%%%%%%%%%%%%%%%%%%%%%%%%
%%%%%%%%%%%%%%%%%%%%%%%%%%%%%%%%%%%%%%%
%%%%%%%%%%%%%%%%%%%%%%%%%%%%%%%%%%%%%%%%%%%

\section{Hermitian geometry of splitting-type complex structures}\label{sec:hermitian}

\noindent
In this section we study the existence of special Hermitian metrics on solvmanifolds endowed with a complex structure of splitting type.
From now on, $F$ denotes the fundamental $(1,1)$-form associated to a Hermitian metric $g$, and $n$ is the complex dimension of the complex manifold.

It is well-known that the {\em K\"ahler condition} ``$dF=0$'' can be weakened in the
``{\em geometry with torsion}'' direction, and the main classes of Hermitian structures that arise are:
%\begin{itemize}
% \item {\em K\"ahler}, that is, $dF=0$;
% \item {\em Hermitian-symplectic} (or {\em holomorphic-tamed}), that is, $F$ is the $(1,1)$-component of a $d$-closed $2$-form;
% \item {\em SKT} (or {\em pluri-closed}), that is, $\partial\bar\partial F=0$;
% \item {\em balanced} (in the sense of Michelsohn \cite{Mi}), that is, $dF^{n-1}=0$;
% \item {\em strongly Gauduchon} ({\em sG} for short) \cite{Pop0}, that is, $F^{n-1}$ is the $(n-1,n-1)$-component of a $d$-closed $(2n-2)$-form; %equivalently, the $(n, n-1)$-form $\partial F^{n-1}$ is $\bar\partial$-exact;
% \item {\em Gauduchon}, that is, $\partial\bar\partial F^{n-1}=0$;
% \item {\em $k$-Gauduchon} \cite{FWW}, that is, $\partial\bar\partial F^k\wedge F^{n-k-1}=0$, where $k=1,\ldots, n-2$.
%\end{itemize}
\begin{itemize}
 \item {\em Hermitian-symplectic} (or {\em holomorphic-tamed}), that is, $F$ is the $(1,1)$-component of a $d$-closed $2$-form;
 \item {\em SKT} ({\em strong K\"ahler with torsion} or {\em pluri-closed}), that is, $\partial\bar\partial F=0$;
 \item {\em $k$-Gauduchon} \cite{FWW}, that is, $\partial\bar\partial F^k\wedge F^{n-k-1}=0$, where $k=1,\ldots, n-2$.
\end{itemize}

\vskip.1cm

The following implications are clear from the definitions:
$$
\text{K\"ahler} \quad \Rightarrow \quad \text{Hermitian-symplectic} \quad \Rightarrow \quad \text{SKT} \quad \Rightarrow \quad \text{$1$-Gauduchon}.
$$
So far, no example of compact complex non-K\"ahler manifold admitting Hermitian-symplectic structure is known, see \cite[page 678]{li-zhang}, \cite[Question 1.7]{streets-tian}.

\vskip.2cm

Other interesting and well-known classes of Hermitian metrics on compact complex manifolds are:
\begin{itemize}
 \item {\em balanced} (in the sense of Michelsohn \cite{Mi}), that is, $dF^{n-1}=0$;
 \item {\em strongly Gauduchon} \cite{Pop0}, that is, $F^{n-1}$ is the $(n-1,n-1)$-component of a $d$-closed $(2n-2)$-form; equivalently, the $(n, n-1)$-form $\partial F^{n-1}$ is $\bar\partial$-exact;
 \item {\em Gauduchon} \cite{gauduchon}, that is, $\partial\bar\partial F^{n-1}=0$.
\end{itemize}

\vskip.1cm

It is clear that
$$
\text{K\"ahler} \quad \Rightarrow \quad \text{balanced} \quad \Rightarrow \quad \text{strongly Gauduchon} \quad \Rightarrow \quad \text{Gauduchon}.
$$
We recall also that any conformal class of Hermitian structures admits a Gauduchon representative by the foundational theorem by Gauduchon \cite[Théorème 1]{gauduchon}. A recent conjecture of Fino and Vezzoni \cite{fino-vezzoni} states that in the
compact non-K\"ahler case it is never possible to find
an SKT metric and also a balanced one, and they prove the conjecture for nilmanifolds \cite{fino-vezzoni-nilmanifolds}
and for 6-dimensional solvmanifolds having holomorphically trivial canonical bundle \cite{fino-vezzoni}.
On the other hand, Popovici \cite{Pop-2015} proposes, for $\partial\db$-manifolds, a conjecture relating their balanced and Gauduchon cones,
and he observes that, if proved to hold, the conjecture would imply
the existence of a balanced structure on any $\partial\db$-manifold.  Recall that a  $\partial\db$-manifold is a compact complex manifold $X$ satisfying the $\partial\bar\partial$-Lemma, that is, 
if for any
$d$-closed form $\gamma$ of pure type on $X$, the following exactness properties are equivalent: 

\vskip.2cm

\hskip1cm
$\gamma$ is $d$-exact $\Longleftrightarrow$
$\gamma$ is $\partial$-exact
$\Longleftrightarrow$
$\gamma$ is $\bar\partial$-exact
$\Longleftrightarrow$
$\gamma$ is $\partial\bar\partial$-exact.

%Clearly,
%$$
%\xymatrix{
%& & \text{Hermitian-symplectic} \ar@{=>}[r] & \text{SKT} \\
%\text{K\"ahler} \ar@{=>}[rru] \ar@{=>}[rd] & & & \\
%& \text{balanced} \ar@{=>}[r] & \text{strongly-Gauduchon} \ar@{=>}[r] & \text{Gauduchon}
%}
%$$

% $$ \text{K\"ahler} \quad \Rightarrow \quad \text{Hermitian-symplectic} \quad \Rightarrow \quad \text{SKT} $$
% and
% $$ \text{K\"ahler} \quad \Rightarrow \quad \text{balanced} \quad \Rightarrow \quad \text{strongly-Gauduchon} \quad \Rightarrow \quad \text{Gauduchon}. $$

\medskip

We have the following general result.

\begin{proposition}\label{prop:balanced-N}
Let $X=G/\Gamma$ be a solvmanifold endowed with a complex structure of splitting type,
i.e., $G=\C\ltimes_\varphi N$, where $N$ is nilpotent.
Then, $X$ admits a balanced (respectively, strongly Gauduchon) Hermitian structure if and only if $N$ admits an invariant balanced (respectively, strongly Gauduchon) Hermitian structure.
\end{proposition}

\begin{proof}
First of all, by the well-known symmetrization process, $X$ admits a balanced (respectively, strongly Gauduchon) Hermitian structure
if and only if the Lie group $G$ admits an invariant balanced (respectively, strongly Gauduchon) Hermitian structure.
Let $n$ be the complex dimension of $X$, and denote by $\{\omega^n\}$ a co-frame of $(1,0)$-forms for the factor $\C$ in $G$. First, notice that, if we have an invariant Hermitian structure $F_G$ on $G$, (respectively, an invariant Hermitian structure $F_N$ on $N$) then we can construct an invariant Hermitian structure $F_N$ on $N$ (respectively, an invariant Hermitian structure $F_G$ on $G$) such that
$$ F_G^{n-1} \;=\;F_N^{n-1} + F_N^{n-2} \wedge \omega^{n\bar n} \;, $$
with abuse of notations.
Indeed, as a vector space, the Lie algebra $\mathfrak{g}$ of $G$ splits as $\mathfrak{g}=\mathfrak{n}\oplus\R^2$, where $\mathfrak{n}$ is the Lie algebra of $N$. Invariant structures on $G$ (respectively, on $N$) are identified with linear structures on $\mathfrak{g}$
(respectively, on $\mathfrak{n}$). If we start from a Hermitian structure $F_N$ on $N$, then we can take $F_G:=\sqrt[n-1]{F_N^{n-1}+F_N^{n-2}\wedge\omega^{n\bar n}}$, which is a Hermitian structure on $G$. On the other hand, if we start from a Hermitian structure $F_G$  on $G$, then it induces a Hermitian structure $F_N$ on $N$ and the Hermitian structure $\omega^{n\bar n}$ on $\R^2$, up to multiplicative positive constants, such that $F_G = F_N + \omega^{n\bar n}$, which yields the above identity.

Since $d\omega^n=0$, we have
\begin{equation}\label{eq:balanced}
dF_G^{n-1} \;=\; dF_N^{n-1} + dF_N^{n-2} \wedge \omega^{n\bar n} \;=\; dF_N^{n-1} + d_N F_N^{n-2} \wedge \omega^{n\bar n} \;,
\end{equation}
where $d_N$ denotes the differential over $N$.

We notice also that $dF_N^{n-1}=0$ by unimodularity. Otherwise, if $dF_N^{n-1}\neq 0$, then either $d\left(F_N^{n-1}\wedge\omega^n\right)$ or $d\left(F_N^{n-1}\wedge\omega^{\bar n}\right)$ would be non-trivial $d$-exact $2n$-forms.

Then, \eqref{eq:balanced} reduces to
$$ dF_G^{n-1} \;=\; d_NF_N^{n-2} \wedge \omega^{n\bar n} \;. $$
It follows that $dF_G^{n-1}=0$ if and only if $d_NF_N^{n-2}=0$. Analogously, it follows that $\partial F_G^{n-1}$ is $\bar\partial$-exact if and only if $\partial F_N^{n-2}$ is $\bar\partial$-exact.
\end{proof}

In \cite{FKV} it is studied the existence of Hermitian-symplectic structures on complex solvmanifolds
(see \cite[Theorem 1.1]{FKV} for case when $G$ is not of type (I) and \cite[Theorem 1.2]{FKV} for other cases).  We recall that a Lie group $G$ is said to be of type (I) if for any $X \in \mathfrak{g}$,
all the eigenvalues of
the adjoint operator $\mathrm{ad}_X$ are pure imaginary.   Some of the Lie algebras in the list of Theorem~\ref{thm:main-thm} are of type (I) but other not, however
for all of them (except $\frs_1$) the Lie group is of the form $G=\C\ltimes_\varphi\C^{n-1}$,
so we give in the following result an alternative direct proof about existence of special
Hermitian metrics in this concrete case.

\begin{proposition}\label{prop:skt-Cn}
Let $X=G/\Gamma$ be a solvmanifold endowed with a complex structure of splitting type,
such that $G=\C\ltimes_\varphi\C^{n-1}$.
Then, for $X$ it is equivalent: to admit SKT structures; to admit Hermitian-symplectic structures; to admit K\"ahler structures.
\end{proposition}

\begin{proof}
By the symmetrization process, $X$ admits SKT, Hermitian-symplectic or K\"ahler structure
if and only if the Lie group $G$ admits an invariant SKT, invariant Hermitian-symplectic or invariant K\"ahler structure.
Fix a co-frame $\{\omega^1,\ldots,\omega^{n-1}\}$ of $(1,0)$-forms on $\C^{n-1}$ and a co-frame $\{\omega^n\}$ of $(1,0)$-forms on $\C$, such that
the complex structure equations are of the form
 $$ \left\{\begin{array}{rcl}
            d\omega^j &=& A^j\,\omega^{nj}+B^j\,\omega^{\bar nj}, \quad \quad \text{  } j\in\{1,\ldots,n-1\}, \\
            d\omega^n &=& 0,
           \end{array} \right. $$
 for suitable $A^j,B^j\in\C$. Notice that the Jacobi identity is satisfied for any value of the structure constants,
 while the unimodularity condition corresponds to the requirement
 $$ \sum_{j=1}^{n-1} (A^j+\bar B^j) \;=\; 0 \;. $$

 Consider the general invariant metric on $G$ given by
 $$ F \;:=\; \sum_{h,k=1}^{n} \alpha_{h\bar k}\, \omega^{h\bar k} $$
 where $(\alpha_{h\bar k})_{h,k}$ is a Hermitian matrix with entries in $\C$.
 By noticing that
 $$\partial\bar\partial\omega^{h\bar k}=(B^h+\bar A^k)(A^h+\bar B^k)\omega^{n\bar n h\bar k},
 \quad\quad
 d\omega^{h\bar k}=(A^h+\bar B^k)\omega^{nh\bar k}+(\bar A^k+B^h)\omega^{\bar n h\bar k},$$
 we get
 $$ \partial\bar\partial F \;=\; \sum_{h,k=1}^{n-1} \alpha_{h\bar k}(B^h+\bar A^k)(A^h+\bar B^k)\omega^{n\bar n h\bar k} \;. $$
 So, if $F$ is SKT, then every coefficients must vanish. In particular, for any $j\in\{1,\ldots,n-1\}$,
 $$ |B^j+\bar A^j|^2=0 \;, $$
 since $\alpha_{j\bar j}\neq0$.
But this implies that the diagonal Hermitian structure
$\tilde F :=\frac i2 \sum_{h=1}^{n}  \omega^{h\bar h}$
is K\"ahler, since $2d\tilde F = i\, \sum_{h=1}^{n-1} \left((A^h+\bar B^h)\omega^{n h\bar h} + (\bar A^h + B^h)\omega^{\bar n h \bar h}\right)=0$.
\end{proof}

\subsection{Hermitian structures in dimension \texorpdfstring{$6$}{6}}%%%\label{...}
Next we consider the case when the (real) dimension of $X$ is 6. As we reminded in the proofs of Propositions~\ref{prop:balanced-N} and~\ref{prop:skt-Cn},
the existence of K\"ahler, Hermitian-symplectic, SKT, balanced and strongly Gauduchon structures is reduced to their existence
at the Lie algebra level, so we will study the spaces of such Hermitian structures on each $\frs_k$, for $1 \leq k \leq 12$.
We also study the existence of 1-Gauduchon structures on the Lie algebras $\frs_k$, although as it is pointed out in \cite{FU},
the symmetrization process does not hold for this kind of Hermitian structures on solvmanifolds, and so
our study covers only the space of \emph{invariant} 1-Gauduchon structures.
The existence results are summarized in Table~\ref{table:summary-metrics}.

A generic Hermitian structure on $\frs_k$ is given, with respect to any coframe $\{\omega^1,\omega^2,\omega^3\}$
of $(1,0)$-forms, by
$$
\begin{pmatrix}        
i r^2&u&z\\
-\bar{u} & i s^2  & v\\              
-\bar{z} & -\bar{v} & i t^2
\end{pmatrix}$$
or equivalently, by the expression
\begin{equation}\label{eq:general-metric}
2F \;=\; ir^2\omega^{1\bar1}+is^2\omega^{2\bar2}+it^2\omega^{3\bar3}+u\omega^{1\bar2}-\bar u\omega^{2\bar1}+v\omega^{2\bar3}-\bar v\omega^{3\bar2}+z\omega^{1\bar3}-\bar z\omega^{3\bar1},
\end{equation}
where $r,s,t\in\R\setminus\{0\}$ and $u,v,z\in\C$ satisfy the conditions that ensure that $F$ is positive-definite:
$r^2s^2>|u|^2$, $s^2t^2>|v|^2$, $r^2t^2>|z|^2$ and
$r^2s^2t^2+2\,\Real(i\bar u \bar v z)>t^2|u|^2+r^2|v|^2+s^2|z|^2$. %%%in order to let $F$ be positive-definite.

%With reference to Theorem \ref{thm:main-thm}, we have the following cases. As for notations, once fixed a coframe $\{\omega^1,\omega^2,\omega^3\}$ of $(1,0)$-forms, consider the general Hermitian form
% \begin{equation}\label{eq:general-metric}
% 2F \;=\; ir^2\omega^{1\bar1}+is^2\omega^{2\bar2}+it^2\omega^{3\bar3}+u\omega^{1\bar2}-\bar u\omega^{2\bar1}+v\omega^{2\bar3}-\bar v\omega^{3\bar2}+z\omega^{1\bar3}-\bar z\omega^{3\bar1}
% \end{equation}
% where $r,s,t\in\R$ and $u,v,z\in\C$ satisfy
% $$ r>0 \;,\quad s>0 \;, \quad t>0\;, $$
% $$ r^2s^2>|u|^2\;,\quad s^2t^2>|v|^2\;,\quad r^2t^2>|z|^2\;,$$
% $$ r^2s^2t^2+2\Re(i\bar u \bar v z)>t^2|u|^2+r^2|v|^2+s^2|z|^2 \;, $$
% in order to let $F$ be positive-definite. Consider also the generic real $2$-form $\Omega$ having $F$ as $(1,1)$-component:
% $$\Omega \;=\; F + \left( L \omega^{12}+M\omega^{23}+N\omega^{13} \right)+\left( \bar L \omega^{\bar1\bar2}+\bar M\omega^{\bar2\bar3}+\bar N\omega^{\bar1\bar3} \right) $$
% where $L,M,N\in\C$.
% Finally, consider the generic real $4$-form $\Omega_2$ having $F^2$ as $(1,1)$-component:
% $$\Omega_2 \;=\; F^2 + \left( \lambda \omega^{123\bar1}+ \mu \omega^{123\bar2}+ \nu \omega^{123\bar3} \right) - \left( \bar \lambda \omega^{1\bar1\bar2\bar3}+\bar \mu \omega^{2\bar1\bar2\bar3}+\bar \nu \omega^{3\bar1\bar2\bar3} \right) $$
% where $L,M,N\in\C$.

%In particular, we have the following:

\medskip

Let us consider first the Lie algebra $\frs_1$, which corresponds to the structure equations~\eqref{split_kt} for $\varepsilon=1$,
and for which we can apply Proposition~\ref{prop:balanced-N} because the Lie group $G$ is of the form $\C\ltimes_\varphi KT$.
By \cite[Observation 4.4]{Pop0}, every strongly Gauduchon compact complex surface is K\"ahler,
so in particular the Kodaira-Thurston manifold does not admit strongly Gauduchon structures.
 Hence, by Proposition 2.1, we conclude that 
$\frs_1$ does not admit either strongly Gauduchon or balanced structures.
A direct calculation shows that it does not admit Hermitian-symplectic structures.
However, there always exist SKT and $1$-Gauduchon structures, since for a metric $F$ given by~\eqref{eq:general-metric} we have
$$
2\, \partial \db F = u\, \omega^{13\bar2\bar3} - \bar u\, \omega^{23\bar1\bar3},
\quad\quad
2\, \partial \db F \wedge F= |u|^2\, \omega^{123\bar1\bar2\bar3}.
$$
More precisely, $F$ is SKT if and only if $F$ is $1$-Gauduchon, if and only if $u=0$.

\medskip

The remaining Lie algebras $\frs_k$, $2 \leq k \leq 12$, correspond to the complex structure equations~\eqref{split_C2},
and we can apply Proposition~\ref{prop:skt-Cn} because the Lie group $G$ is of the form $\C\ltimes_\varphi\C^2$.
As a matter of notation, let us denote such complex structures simply as $J=(A, B, \varepsilon) \in \mathbb{C}^2 \times \{0,1\}$.
%%%, where $\varepsilon\in\{0,1\}$
Given a generic Hermitian structure~\eqref{eq:general-metric},
we first note that
%%%if we consider the $(1,0)$-basis given by $\omega'^1 = r\,\omega^1,\,\,\omega'^2 = s\,\omega^2,\,\,\omega'^3 = \omega^3$, then the structure constants remain the same and we obtain structure coefficients  $r'^2=s'^2 =1$.
one can always normalize the metric coefficients $r$ and $s$, i.e. we can suppose $r=s=1$. Therefore, we will identify
the Hermitian structures simply by a tuple $F=(t^2, u, v, z)\in \mathbb{R}^+ \times \mathbb{C}^3$, where
$1>|u|^2$, $t^2>|v|^2$, $t^2>|z|^2$ and
$t^2+2\,\Real(i\bar u \bar v z)>t^2|u|^2+|v|^2+|z|^2$,
in order  for $F$ to be positive-definite.

Now, by Proposition \ref{prop:skt-Cn}, there exists a K\"ahler structure if and only if there is a Hermitian-symplectic structure,
if and only if there exists an SKT structure. A direct calculation from~\eqref{split_C2} shows that  the existence of one of these types of structures implies
$$
A + \bar B \;=\; 0,
$$ 
that is, the complex structure must be of the form $J=(A, -\bar A, \varepsilon)$, where $A\in\C$ and $\varepsilon\in\{0,1\}$.
%%%Moreover, $A$ can be chosen to have $|A|\leq1$.
According to the classification given in Section~\ref{subsec:classification-algebras}, the Lie algebras
admitting such a complex structure are $\mathfrak{s}_2,\, \mathfrak{s}_3,\, \mathfrak{s}_7^{\alpha}$.
Indeed,

\vskip.1cm

- if $\varepsilon=0$ then from Table~\ref{tabla2} we get $\mathfrak{s}_2$ (notice that we can take $A=1$ in this case);

\vskip.1cm

- if $\varepsilon=1$ and $A\in\R$, then by Table~\ref{tabla3} the possibilities are
$\mathfrak{s}_2$, $\mathfrak{s}_7^{\alpha}$;

\vskip.1cm

- if $\varepsilon=1$ and $\Imag A\not= 0$, then from Table~\ref{tabla_imA=imBneq0} we get $\mathfrak{s}_3$.

%  The parameters have to satisfy the following equations:
%  \begin{itemize}
%   \item in case K\"ahler: $(A-C)u = Az = (\bar A - \bar C) u = Cv = 0$; (take, e.g., the diagonal metric with $u=v=z=0$);
%   \item in case SKT: $u |B+\bar C|^2 = 0$ (take, e.g., $u=0$);
%   \item in case Hermitian-symplectic: $(A+C)L=(\bar A+ \bar C)L=(A-C)u=Az+2N\bar A=(\bar A-\bar C)u=Cv+2M\bar C=0$;
%   \item in case $1$-Gauduchon: $r^2s^2 |A+\bar B|^2+|u|^2|C+\bar B|^2$.
%  \end{itemize}
%  On the other side, the Lie group $\C\ltimes_\varphi\C^2$ always admit balanced metrics. In fact, the equations to be satisfied are $(A+C+\bar B)(ir^2v+z\bar u)=(uv-is^2z)\bar B=0$ (take, e.g., $v=z=0$).

\medskip

Next we give a detailed description of the spaces of K\"ahler structures.

\begin{proposition}\label{casosK}
Let $\frg$ be a $6$-dimensional solvable Lie algebra with a complex structure $J$ of splitting type.
Then, $\frg$ admits a K\"ahler structure if and only if $\frg$ is isomorphic to $\mathfrak{s}_2$, $\mathfrak{s}_3$ or $\mathfrak{s}_7^{\alpha}$, and the K\"ahler structures $(J,F)$ are the following:
\begin{enumerate}
\item[(K.i)] $(\frs_2, J, F)$, where $J=(1,-1,0)$ and $F=(t^2, 0, v, 0)$.\medskip
\item[(K.ii)] $(\frs_3, J, F)$, where $J=(A,-\bar A,1)$, $\Imag A\neq 0$, and $F=(t^2, 0, 0, 0)$.
\medskip
\item[(K.iii)] $(\frs_7^{\alpha}, J, F)$, where
$J=(A,-A,1),\,\, A\in\mathbb R\setminus\{0, -1\}$, and $F=(t^2, 0, 0, 0)$. (Notice that $\alpha = |A|$ or $\alpha=|\frac1A|$.)\medskip
%\item[(K.iv)] $(\frs_8^0, J, F)$, where $J=(1,-1,1)$ and $F=(t^2, 0, 0, 0)$.\medskip
\item[(K.iv)] $(\frs_7^1, J, F)$, where $J=(-1,1,1)$ and $F=(t^2, u, 0, 0)$.\medskip
\end{enumerate}
\end{proposition}

\begin{proof}
A direct computation shows that
\begin{equation}\label{delta_bar_F}
2\,\bar\partial F \;=\; (\bar A + \varepsilon)\,(u\,\omega^{1\bar2\bar3} + \bar u\,\omega^{2\bar1\bar3}) -\varepsilon \bar v\,\omega^{3\bar2\bar3} + \bar A\bar z\,\omega^{3\bar1\bar3} \;.
\end{equation}

Hence the conditions to be satisfied for $F$ being K\"ahler are
\begin{equation*}%\label{kahler-conditions}
u(A+\varepsilon)\;=\;0\;,\qquad \varepsilon v\;=\;0\;,\qquad Az \;=\; 0\;.
\end{equation*}

If $\varepsilon = 0$, then we may assume that $A=1$ (see the proof of Proposition~\ref{prop2} for details) and therefore $u=z=0$.
The K\"ahler structures are then given by $(t^2, 0, v, 0)$ and we obtain case {(K.i)}.

If $\varepsilon=1$, then $v=0$ and several cases appear:
\begin{itemize}
\item If $A=0$, it is equivalent to the previous case {(K.i)}.
\item If $A=-1$, then $z=0$. So, $J=(-1, 1, 1)$ and $F = (t^2, u, 0, 0)$, which corresponds to {(K.iv)}.
\item If $A\neq 0$ and $A\neq-1$,
%%%it can be supposed to verify $|A|\leq 1$ and
then $u=v=z=0$. Depending on the values of $A$ (see Tables~\ref{tabla3} and~\ref{tabla_imA=imBneq0}),
we get the remaining cases {(K.ii)} or {(K.iii)}.\qedhere
\end{itemize}
\end{proof}

\begin{remark}
{\rm In \cite[Example 4]{hasegawa},  it is shown that the complex structures corresponding to
cases (K.i) and (K.iv) in Proposition~\ref{casosK} admit K\"ahler metrics.
In the recent paper \cite{Andriot} it is shown that $\frs_3$ admits a K\"ahler structure and, moreover,
solvmanifolds constructed from the Lie algebra $\frs_7^1$ give rise to
new supersymmetric vacua. Notice that $\frs_2, \frs_3$ and $\frs_7^{\alpha}$ are the only (non abelian)
solvable Lie algebras in six dimensions admitting Ricci flat metrics (see \cite{Andriot} and the
references therein).  By Proposition~\ref{casosK} all these Lie algebras admit a K\"ahler structure,
although by \cite{fino-otal-ugarte} only $\frs_7^1$ with $J=(1,-1,1)$ admits
a Calabi-Yau structure.}
\end{remark}

In the following proposition we compare the spaces of Hermitian-symplectic, SKT and 1-Gauduchon structures
with the space of K\"ahler structures.

\begin{proposition}\label{casosSKT}
Let $\frg$ be a $6$-dimensional solvable Lie algebra with a complex structure~$J$ of splitting type
that admits K\"ahler structures.
Any Hermitian structure $(J,F)$ on $\frg$ is 1-Gauduchon if and only if it is Hermitian-symplectic, if and only if it is SKT.
Moreover, any SKT structure $(J,F)$ on $\frg$ is one of the following:
\begin{enumerate}
\item[(SKT.i)] $(\frs_2, J, F)$, where $J=(1,-1,0)$ and $F=(t^2, 0, v, z)$.\medskip
\item[(SKT.ii)] $(\frs_3, J, F)$, where $J=(A,-\bar A,1)$, $\Imag A\neq 0$, and $F=(t^2, 0, v, z)$.\medskip
\item[(SKT.iii)] $(\frs_7^{\alpha}, J, F)$, where
$J=(A,-A,1)$, $A\in\mathbb R\setminus\{0, -1\}$, and $F=(t^2, 0, v, z)$.\medskip
%\item[(SKT.iv)] $(\frs_8^0, J, F)$, where $J=(1,-1,1)$ and $F=(t^2, 0, v, z)$.\medskip
\item[(SKT.iv)] $(\frs_7^1, J, F)$, where $J=(-1,1,1)$ and $F=(t^2, u, v, z)$.
\medskip
\end{enumerate}
\end{proposition}

\begin{proof}
Using \eqref{delta_bar_F}, we have
$$2\,\partial \bar\partial F = |A + \varepsilon|^2 (u\,\omega^{13\bar2\bar3} -\bar u\, \omega^{23\bar1\bar3}),
\quad\quad
2\, \partial \db F \wedge F= |u|^2\, |A + \varepsilon|^2\, \omega^{123\bar1\bar2\bar3}.$$
Therefore, the SKT condition is equivalent to the 1-Gauduchon condition, and they are equivalent to $u(A+\varepsilon) \;=\; 0.$

On the other hand, the structure $F$ is Hermitian-symplectic if
$$
\bar\partial F \;=\; \partial \beta,\quad \bar\partial\beta=0,\quad\text{where  }\quad \beta\in\frg^{0,2}\;.
$$
Since $\partial\beta \in\langle A\,\omega^{3\bar1\bar3},\, \varepsilon\,\omega^{3\bar2\bar3}\rangle$,
it follows from~\eqref{delta_bar_F} that
$F$ is Hermitian-symplectic if and only if there exist $\lambda,\, \mu\in\mathbb C$ satisfying
$$
u(A+\varepsilon) \;=\; 0\;,\quad v\,\varepsilon \;=\; \lambda\,\varepsilon\;,\qquad z\,A \;=\; \mu\, \bar A.
$$
It is always possible to find $\lambda,\mu$ satisfying the last two equations.
The first one is precisely the SKT condition.

Now, depending on the vanishing of the metric coefficient $u$, the possibilities for
a Hermitian structure $(J, F)$ to satisfy the SKT condition are:
\begin{itemize}
\item $u\not=0$. Then, $\varepsilon =1$ and $A= -1$, which corresponds to the case {(SKT.iv)}.
\item $u=0$. If $\varepsilon=0$, then we can suppose $A=1$, which leads to the case {(SKT.i)}.
The remaining cases {(SKT.ii)} and {(SKT.iii)}  are obtained when $\varepsilon =1$.
\end{itemize}
\end{proof}

\begin{remark}%%%\label{...}
{\rm A complex structure $J$ as above admits SKT structures if and only if it admits K\"ahler ones,
however, for any fixed $J$, there exist SKT structures which are not K\"ahler.
Indeed, by Propositions~\ref{casosK} and~\ref{casosSKT}, any SKT structure with metric coefficient $z\neq 0$ is not K\"ahler.
Similarly, there exist Hermitian-symplectic structures and 1-Gauduchon structures which are not K\"ahler.}
\end{remark}

Now, with respect to balanced and strongly Gauduchon Hermitian structures,
we can apply Proposition \ref{prop:balanced-N} for $N=\C^2$ and so any complex structure corresponding to the equations~\eqref{split_C2}
admits balanced structures. Indeed, for any value of the tuple $(A, B, \varepsilon)\in \mathbb{C}^2 \times \{0,1\}$, the Hermitian structures
given by $(t^2, u, 0,0)$ are balanced. Notice that there exist strongly Gauduchon Hermitian structures that are not balanced,
for instance, consider a complex structure $J=(A, B, 1)$, i.e. with $\varepsilon=1$, and a Hermitian structure $F$
given by $(t^2, 0, v, z)$ with $v\neq 0$.

\medskip

We summarize all the results about Hermitian structures in Table~\ref{table:summary-metrics}.
Here, the symbol ``$\checkmark$'' means that the corresponding kind of Hermitian metrics exists \emph{for any complex structure of splitting type}
on the Lie algebra (see Tables~\ref{tabla2}--\ref{tabla6}),
whereas ``$-$'' means that none of the complex structures admits such kind of metrics. Here ``H-symplectic'' means Hermitian-symplectic
and ``sG'' refers to strongly Gauduchon metrics.

\begin{table}[!ht]
\renewcommand{\arraystretch}{1.2}
%\resizebox{\textwidth}{!}{
\begin{tabular}{|c|c|c|c|c|c|c|}
\hline
& K\"ahler & H-symplectic & SKT & invariant 1-G & balanced & sG\\
\hline
$\frs_1$ & $-$ & $-$ & $\checkmark$ & $\checkmark$ & $-$ & $-$ \\
\hline
$\frs_2$ & $\checkmark$ & $\checkmark$ & $\checkmark$ & $\checkmark$ & $\checkmark$ & $\checkmark$ \\
\hline
$\frs_3$ & $\checkmark$ & $\checkmark$ & $\checkmark$ & $\checkmark$ & $\checkmark$ & $\checkmark$ \\
\hline
$\frs_4$ & $-$ & $-$ & $-$ & $-$ & $\checkmark$ & $\checkmark$ \\
\hline
$\frs_5^{\alpha}$ & $-$ & $-$ & $-$ & $-$ & $\checkmark$ & $\checkmark$ \\
\hline
$\frs_6^{\alpha,\beta}$ & $-$ & $-$ & $-$ & $-$ & $\checkmark$ & $\checkmark$ \\
\hline
$\frs_7^{\alpha}$ & $\checkmark$ & $\checkmark$ & $\checkmark$ & $\checkmark$ & $\checkmark$ & $\checkmark$ \\
\hline
$\frs_8^{\alpha}$ & $-$ & $-$ & $-$ & $-$ & $\checkmark$ & $\checkmark$ \\
\hline
$\frs_{9}$ & $-$ & $-$ & $-$ & $-$ & $\checkmark$ & $\checkmark$ \\
\hline
$\frs_{10}^{\alpha, \beta}$ & $-$ & $-$ & $-$ & $-$ & $\checkmark$ & $\checkmark$ \\
\hline
$\frs_{11}^{\alpha}$ & $-$ & $-$ & $-$ & $-$ & $\checkmark$ & $\checkmark$ \\
\hline
$\frs_{12}$ & $-$ & $-$ & $-$ & $-$ & $\checkmark$ & $\checkmark$ \\
\hline
\end{tabular}
%}
\medskip
\caption{Existence of Hermitian metrics for any complex structure of splitting type.}
\label{table:summary-metrics}
\end{table}

\begin{remark}%%%\label{...}
{\rm
Note that the Lie algebra $\mathfrak s_1=\frg_{4,9}^0\oplus \R^2$
(see Remark~\ref{remark-comparacion}) admits SKT Hermitian structures because the 4-dimensional
Lie algebra $\frg_{4,9}^0$ admit them by \cite{madsen-swann}, and so
the product complex structure on $\mathfrak s_1$ admits SKT structures.
However, the Hermitian structures that we have obtained on $\mathfrak{s}_1$
are different because the splitting-type complex structure is not a product,
and in this sense, our study above provides a new example of SKT metrics in dimension 6.

Finally, we notice also that our results provide (up to our knowledge) new families of non-K\"ahler balanced solvmanifolds
(see also Remark~\ref{remark-lattices}). The $\frs_{12}$ case is especially rich, as Section~\ref{sec:nakamura} below shows.
}
\end{remark}

\bigskip

In relation to the conjectures in \cite{fino-vezzoni} and in \cite{Pop-2015} mentioned above, as a consequence
of the results of this section one has the following result.

\begin{corollary}\label{cor-consec}
Let $X=G/\Gamma$ be a $6$-dimensional solvmanifold endowed with a complex structure of splitting type.
We have:
\begin{enumerate}
\item[(i)] If $X$ has an SKT metric and also a balanced metric, then $X$ is K\"ahler.
\item[(ii)] If $X$ satisfies the $\partial\db$-Lemma, then $X$ is balanced.
\end{enumerate}
\end{corollary}

\begin{proof}
If $X$ has an SKT metric and also a balanced metric, then by symmetrization, there is an SKT structure
and also a balanced structure on the Lie algebra $\frg$ underlying $X$. Now, by Table~\ref{table:summary-metrics},
the Lie algebra is isomorphic to $\frs_2$, $\frs_3$ or $\frs_7^{\alpha}$. In any case, there is a
K\"ahler structure on $\frg$ and so $X$ is K\"ahler, which completes the proof of (i).

For the proof of (ii), in view of Table~\ref{table:summary-metrics} it is enough to prove that for any lattice
$\Gamma$ on the connected and simply-connected Lie group $G_1$ corresponding to $\frs_1$, the solvmanifold
$G_1/\Gamma$ does not satisfy the $\partial\db$-Lemma with respect to any complex structure of splitting type $J$.
In addition, by the symmetrization process, it suffices to check that the $\partial\db$-Lemma is not satisfied at
the Lie algebra level. Now, for any complex structure of splitting type $J$ we have a basis $\{\omega^1,\omega^2,\omega^3\}$
of (1,0)-forms satisfying~\eqref{split_kt} with $\varepsilon=1$, therefore the (1,1)-form
$$
\omega^{1\bar1}=d\omega^2=\partial(-\omega^{\bar2})=\db \omega^2
$$
is $d$-exact, $\partial$-exact and $\db$-exact, but it is not $\partial\db$-exact.
\end{proof}

%%%%%%%%%%%%%%%%%%%%%%%%%%%%%%%%%%%%%%%%%
%%%%%%%%%%%%%%%%%%%%%%%%%%%%%%%%%%%%%%%%%
%%%%%%%%%%%%%%%%%%%%%%%%%%%%%%%%%%%%%%%%%

\section{Complex structures on the Nakamura manifold}\label{sec:nakamura}

\noindent
In this section we focus on the complex geometry of splitting type on the Nakamura manifold \cite{nakamura},
whose underlying Lie algebra is $\mathfrak{s}_{12}$.
Firstly, we classify the complex structures of splitting type, which allows us
to produce analytic families of complex solvmanifolds with holomorphically trivial canonical bundle
satisfying interesting properties in relation to the
$\partial\bar\partial$-Lemma.

\subsection{Moduli of complex structures of splitting type on the Nakamura manifold}

%%%We recall that the Lie algebra underlying the Nakamura manifold \cite{nakamura} is, in our notations, $\mathfrak{s}_{12}$.

Next we study the space of complex structures of splitting type on the Lie algebra $\mathfrak{s}_{12}$ up to equivalence.

\begin{proposition}\label{prop:moduli_Nk}
On the Lie algebra $\frs_{12}$, there exist the following non-equivalent complex structures of splitting type:
\begin{itemize}
\item[(i)] $(\frs_{12}, \tilde{J}):$  $d\omega^1 = -\omega^{13},\quad d\omega^{2} = \omega^{23},\quad d\omega^3 =0$;
\medskip
\item[(ii)] $(\frs_{12}, J_A) :$ $
 \begin{cases}
\begin{array}{l}
d\omega^1 = A\,\omega^{13}-\omega^{1\bar 3},\\
d\omega^{2} = -A\,\omega^{23}+\omega^{2\bar3},\quad A \in \C, \quad |A|\neq 1,\\
d\omega^3 =0;
\end{array}
\end{cases}
$
\medskip
\item[(iii)] $(\frs_{12}, J_B):$  $\begin{cases}
\begin{array}{l}
d\omega^1 = -\omega^{13}+B\,\omega^{1\bar 3},\\
d\omega^{2} = -\bar B\,\omega^{23}+\omega^{2\bar3},\quad B \in \C, \quad |B|<1,\\
d\omega^3 =0.
\end{array}
\end{cases}$
\end{itemize}
\end{proposition}

\begin{proof}
Here the equivalence between the complex structures is in the usual sense:
two complex structures $J$
and $J'$ on a Lie algebra $\frg$ are equivalent if there exists an automorphism
$F\colon \frg\longrightarrow\frg$ such that
$J=F^{-1}\circ J'\circ F$.
We first observe the following property of the complex structures
defined by equations~\eqref{split_C2} with $A=-1$ and $\varepsilon=1$: if
we denote by $J_B$ such a complex structure, then, for $B \not=0$, $J_B$ is equivalent to $J_{1/B}$
(indeed, it suffices to multiply $\omega^3$ by $\bar B$, and change $\omega^1$ with $\omega^2$).
This property explains the condition $|B|<1$ in the equations (iii) above.

Now, according to our classification in Section~\ref{sec:classification} of complex structures of splitting type,
the Lie algebra $\frs_{12}$ appears only in some specific cases in the Tables~\ref{tabla2},
\ref{tabla3} and~\ref{tabla6}.
First, from Table~\ref{tabla2}, in the case $(A,B)=(-1,0)$ we obtain equations (i),
%%%which corresponds to the complex-parallelizable structure $\tilde{J}$ on the Nakamura manifold,
and in the case $(A,B)=(0,-1)$ equations (iii) for $B=0$, just considering a new basis $\{\tau^1 = \omega^2,\, \tau^2 = \omega^1,\, \tau^3 = -\omega^3\}.$
%%%which corresponds to the abelian complex structure.
The case $B=-1$, $A\in \mathbb R-\{\pm1\}$, of Table~\ref{tabla3} lies in equations (ii),
whereas the case $A=-1$, $B\in \mathbb R-\{\pm1\}$, of Table~\ref{tabla3} lies in the equations (iii).

With respect to Table~\ref{tabla6}, the complex structures on $\frs_{12}$ satisfy the conditions
$$
\Imag A \neq \Imag B,\quad\quad \Imag A (1+\Real B) = - \Imag B (1+\Real A),\quad\quad \Delta =\pm (|B|^2-|A|^2)\neq 0,
$$
where $\Delta = |A|^2 + |B|^2 + 2(\Real A + \Real B + \Real A\, \Real B - \Imag A\, \Imag B)$.
%%%and $\delta_{\pm} = \pm (|B|^2-|A|^2)$.
%%%\begin{equation}
%%%\begin{cases}
%%%\Imag A (1+\Real B) = - \Imag B (1+\Real A),\\
%%%\Delta = |A|^2 + |B|^2 + 2(\Real A + \Real B + \Real A\, \Real B - \Imag A\, \Imag B) = \delta_{+}\, \text{or}\, \delta_{-},\\
%%%\delta_{\pm} = \pm (|B|^2-|A|^2),
%%%\end{cases}
%%%\end{equation}
%%%where $\Delta\neq 0,\, |A|\neq |B|$ and $\Imag A \neq \Imag B$.
Next we study the solutions of this set of equations:

- If $\Imag A=0$, since $\Imag B\neq 0$, then $A=-1$ and $\Delta = |B|^2-1\neq 0$.  In this case, the structures belong to case (iii).
Similarly, if $\Imag B=0$ then $B=-1$ and $\Delta = 1-|A|^2\neq 0$, so we are in case (ii).

- Suppose now that $(\Imag A)(\Imag B)\neq 0$ and $\Real A = \Real B = -1$.  It is straightforward to verify that
$\Delta = (\Imag A - \Imag B)^2$. But $\Delta = \pm(|B|^2-|A|^2) = \pm(\Imag^2B - \Imag^2A)$ implies
$\Imag A = \Imag B$, which is a contradiction to $\Delta \neq 0$.

- Finally, if $(\Imag A)(\Imag B)\neq 0$ and $(1+\Real A)(1+\Real B)\neq 0$, then we can take
$\Imag B = -\Imag A \left(\frac{1+\Real B}{1+\Real A}\right)$. Now, the condition $\Delta = |B|^2-|A|^2$ is equivalent to
$B= -A \left(\frac{1+\bar A}{1+A}\right)$, which implies $|B| = |A|$.
The case $\Delta = -(|B|^2-|A|^2)$ is similar. In conclusion, we do not get
complex structures in these cases.

\medskip

Let us study now the equivalences of complex structures.
Observe first that all the complex structures in the cases (i) and (ii)
satisfy $\text{dim }H^{3,0}_{\bar\partial} (\frg) = 1$, but $\text{dim }H^{3,0}_{\bar\partial} (\frg)=0$ for
the complex structures in case (iii).
Therefore, there are no equivalences between the case (iii) and cases (i)-(ii).
A direct calculation allows to show that the complex structure (i) is not equivalent to
any complex structure in (ii), and moreover, two complex structures $J$ and $J'$ in (ii), respectively in (iii),
are equivalent if and only if $A=A'$, respectively $B=B'$.
\end{proof}

\begin{remark}\label{remark_s12_one}
{\rm Observe that $\tilde{J}$ given by (i) is the complex-parallelizable structure on the
Nakamura manifold \cite{nakamura}, and the complex structure given by $A = 0$ in the family (ii) corresponds to
the abelian complex structure, see \cite{andrada-barberis-dotti}. In addition, a
complex structure of splitting type on $\frs_{12}$
gives rise to a complex solvmanifold with holomorphically trivial canonical bundle if and only if
it belongs to (i) or (ii), accordingly to Remark~\ref{remark-trivialbundle}.}
\end{remark}

The following theorem reveals that the Nakamura manifold has a rich space of complex structures.
The result is based on an appropriate deformation of its abelian complex structure.

\begin{theorem}\label{canonico_no_abierta}
The property of having holomorphically trivial canonical bundle and the property of being of splitting type
are not stable under holomorphic deformations.
\end{theorem}

\begin{proof}
Although the first part of the theorem was firstly shown by Nakamura~\cite{nakamura},
we provide other proof based on the invariant complex geometry described in Proposition~\ref{prop:moduli_Nk}.

Let $\Gamma$ be any lattice on the Lie group $G_{12}$ corresponding to $\frs_{12}$, and consider the complex solvmanifold
$X_0=(G_{12}/\Gamma,J_0)$ endowed with its abelian complex structure $J_0$,
which is given by taking $A=0$ in the family (ii) of Proposition~\ref{prop:moduli_Nk}.
Consider an open disc $\Delta:=\Delta(0,\epsilon)$ around $0$ in $\mathbb{C}$ for $\epsilon>0$
small enough, and the family $\{X_t\}_{t\in\Delta}$ of complex solvmanifolds given
by the solvmanifold $G_{12}/\Gamma$ endowed with the complex structure $J_t$ defined by the (1,0)-co-frame $\{\omega_t^1:=\omega^1,\;\omega_t^2:=\omega^2,\;
\omega_t^3:=\omega^3-t\omega^{\bar1}\}$.
Notice that the form $\omega^{\bar1}$ defines a non-zero Dolbeault cohomology class on $X_0$, and so the previous family $X_t$
provides a small holomorphic deformation of $X_0$.
The complex structure equations of the invariant  complex structure $J_t$ are
\begin{equation}\label{def_Nak}
d\omega_t^1=-\omega_t^{1\bar3},\qquad d\omega_t^2=-\bar t\,\omega_t^{12}+\omega_t^{2\bar3},\qquad  d\omega_t^3=-t\,\omega_t^{3\bar1}.
\end{equation}

Now, since $d\omega_t^{123}=-t\,\omega_t^{123\bar1}\neq 0$ for any $t\in\Delta^*$, by~\cite[Proposition 2.1]{fino-otal-ugarte} the solvmanifold
$X_t$ does not have holomorphically trivial canonical bundle for any $t\neq0$.
Indeed, $J_t$ does not belong to (i) or (ii) for $t\neq0$, see Remark~\ref{remark_s12_one}.
Moreover, from the complex structure equations~\eqref{def_Nak} one also has that $J_t$ does not belong to the family (iii),
because there are not non-zero invariant holomorphic (1,0)-form for $t\neq0$.
In conclusion, $J_t$ is not of splitting type for any $t\neq0$.
%%%Moreover, $h_\db^{1,0}\geq1$ for any complex structure in family (iii) whereas
%%%$h_\db^{1,0}=0$ for any complex structure~\eqref{def_Nak} with $t\neq 0$, hence the deformation is not of splitting type.
\end{proof}

\begin{remark}
{\rm
 All the complex structures $J_t$ given in the proof of Theorem~\ref{canonico_no_abierta} admit balanced metrics.
 }
\end{remark}

\subsection{The \texorpdfstring{$\partial\bar\partial$}{partial-overline-partial}-Lemma on a family of splitting-type complex structures on the Nakamura manifold}
In \cite[Proposition 3.7]{fino-otal-ugarte} the complex structures on the Lie algebra  $\frs_{12}$ giving rise to complex solvmanifolds with holomorphically trivial canonical
bundle are classified.  There are two complex structures, denoted in the aforementioned paper by $J'$ and $J''$, and a family $J_C$ parametrized by $C \in \C$ with $\Imag C \neq~0$ which can be represented by a $(1,0)$-co-frame 
$\left\{\omega_C^1,\omega_C^2,\omega_C^3\right\}$ with structure equations:
\begin{equation}\label{htcb}
J_C : \quad \left\{ \begin{array}{l}
            d \omega_C^1 \;=\; - (C - i)\, \omega_C^{13} - (C + i)\, \omega_C^{1\bar3}, \\[5pt]
            d \omega_C^2 \;=\;   (C - i)\, \omega_C^{23} + (C + i)\, \omega_C^{2\bar3}, \\[5pt]
            d \omega_C^3 \;=\; 0.
           \end{array}\right. \;
\end{equation}

\noindent Observe that all the structures $J_C$ are of splitting type, whereas $J'$ and $J''$ are not.

Moreover, the family~\eqref{htcb} unifies
the complex structures (i) and (ii) of Proposition~\ref{prop:moduli_Nk}.
Concretely, if $C=-i$ in \eqref{htcb} then we obtain the complex-parallelizable structure $\tilde{J}$ in Proposition~\ref{prop:moduli_Nk},
whereas if $C\neq-i$ then the complex structure $J_C$ corresponds to the complex structure $J_A$ in the family (ii) of Proposition~\ref{prop:moduli_Nk} for $A=-(C-i)/(\bar C-i)$.
Thus, the connected and simply-connected solvable Lie group $G_{12}$
with Lie algebra $\frs_{12}$, endowed with a left-invariant complex structure $J_C$ given by~\eqref{htcb},
may be written as a semi-direct product
$(G_{12},J_C)=\C\ltimes_{\varphi_C}\C^2$,
%%%\footnote{{\color{red} Basta poner $G=\C\ltimes_{\varphi_C}\C^2$.  Al menos as\'i es como se ha escrito la secci\'on 1.}}
where the action $\varphi_{C}$ is described by a diagonal matrix~\eqref{phi_dos} and the characters $\alpha_1^C,\alpha_2^C\colon\C\to\C^*$ are
\begin{equation}\label{caracteres_g8}
\alpha_1^C(z_3)\;=\;e^{-(C-i)z_3-(C+i)\bar z_3},\qquad\alpha_2^C(z_3)\;=\;\alpha_1^C(z_3)^{-1}.
\end{equation}

Now, we are concerned with the construction of lattices $\Gamma$ in $(G_{12},J_C)$ compatible with the splitting. They are of the form $\Gamma=\Gamma'\ltimes_{\varphi_C}\Gamma''$, where
$\Gamma'$ and $\Gamma''$ are lattices of $\C$ and $\C^2$ respectively  and $\Gamma'$ is compatible with the splitting, in other words, $\varphi_C(z)\left(\Gamma''\right)\subseteq\Gamma''$ for any
$z\in\Gamma'$. The former condition implies that $\varphi_C|_{\Gamma'}$ must be in the conjugation class of a matrix in GL$(2,\Z)$.

\begin{lemma}\label{lema_lattice}
For every $C\in\C$ with $\Imag C\neq0$, the lattice $\Gamma_C':=\frac{\pi}{2\,\Imag C}(1-i\Real C)\Z\oplus \frac{i}{2}\log(\frac{3+\sqrt{5}}{2})\Z$ of $\C$ is compatible with the splitting
$\varphi_C$ given by the characters~\eqref{caracteres_g8}. Thus, the complex solvmanifold $X_C\;:=\;
(G_{12}/\Gamma_C,J_C)$ is of splitting type, where $\Gamma_C:=\Gamma_C'\ltimes_{\varphi_C}\Gamma''$ and $\Gamma''$ is
a lattice of $\C^2$.
\end{lemma}
\begin{proof}

After computing its characteristic polynomial, it turns out that the diagonal matrix~\eqref{phi_dos} with characters~\eqref{caracteres_g8} is in the conjugation class of a matrix in GL$(2,\Z)$, if the condition $(C-i)z_3+(C+i)\bar z_3=\log(\frac{n+\sqrt{n^2-4}}{2})$ holds for some $n\in\Z$. In particular, fixed $C\in\C$ with $\Imag C\neq 0$, we get $z_3=\frac{\pi}{2\,\Imag C}(1-i\Real C)$ for $n=-2$
  and
$z_3=\frac{i}{2}\log(\frac{3+\sqrt{5}}{2})$ for $n=3$. Therefore, $\Gamma_C'=\frac{\pi}{2\,\Imag C}(1-i\Real C)\Z\oplus \frac{i}{2}\log(\frac{3+\sqrt{5}}{2})\Z$ is a lattice of $\C$ compatible with the
splitting.
\end{proof}

As a consequence of the previous lemma, we have a family $\{X_C\}_{ \Imag C\neq0}$ of complex solvmanifolds of splitting type with underlying real Lie algebra $\frs_{12}$. We are interested in knowing
which of them satisfy the $\partial\db$-Lemma. The following result states a sufficient condition in order to satisfy the $\partial\db$-Lemma.  This condition is stated and proved in terms of the differential complexes
$(B_\Gamma^{\bullet,\bullet},\db)$ and
$(C_\Gamma^{\bullet,\bullet},\partial,\db)$ defined by Kasuya~\cite{kasuya-mathz}, respectively, by Kasuya and the first author~\cite{angella-kasuya-1}.  Recall that such complexes allow to compute the Dolbeault, respectively the Bott-Chern cohomology of complex solvmanifolds of
splitting type. 

\begin{lemma}\label{lema_B}
Let $X=(G/\Gamma,J)$ be a complex solvmanifold of splitting type. If
$\partial|_{B^{\bullet,\bullet}_\Gamma}=\db|_{B^{\bullet,\bullet}_\Gamma}=0$ and $B^{q,p}_\Gamma=\overline{B^{p,q}_\Gamma}$ for all $p,q\in\N$, then $X$ satisfies the $\partial\db$-Lemma.
\end{lemma}
\begin{proof}
If the complex $(B^{\bullet,\bullet}_\Gamma,\partial,\db)$ satisfies $B^{q,p}_\Gamma=\overline{B^{p,q}_\Gamma}$ for all $p,q\in\N$ then, since $C^{\bullet,\bullet}_{\Gamma} \;:=\; B^{\bullet,\bullet}_{\Gamma} + \overline{B^{\bullet,\bullet}_{\Gamma}}$, 
 it holds
$C^{\bullet,\bullet}_\Gamma=B^{\bullet,\bullet}_\Gamma$. 
Furthermore,  the condition $\partial|_{B^{\bullet,\bullet}_\Gamma}=\db|_{B^{\bullet,\bullet}_\Gamma}=0$ forces the natural isomorphisms
$$
H^{\bullet,\bullet}_\text{BC}(X)\cong H^{\bullet,\bullet}_\text{BC}(C_\Gamma)=C^{\bullet,\bullet}_\Gamma=B^{\bullet,\bullet}_\Gamma=H^{\bullet,\bullet}_\db(B_\Gamma)\cong
H^{\bullet,\bullet}_\db(X).
$$
Hence, $X$ satisfies the $\partial\db$-Lemma.
\end{proof}

 Now we recall the construction of the differential complex $(B_{\Gamma_C}^{\bullet,\bullet},\db)$ defined in \cite[Corollary
4.2]{kasuya-mathz} in order  to compute the
Dolbeault cohomology of $X_C$. For any complex solvmanifold in the family $\{X_C\}_{\Imag C\neq0}$, %we want to construct the differential complex $(B_{\Gamma_C}^{\bullet,\bullet},\db)$ defined in \cite[Corollary 4.2]{kasuya-mathz} in order  to compute the Dolbeault cohomology of $X_C$. 
consider a set $\{z_1,z_2\}$  of local coordinates on $\C^2$ and $z_3$ a local coordinate on $\C$.
We have the following basis  $\{\omega^1_C,\, \omega^2_C,\,\omega^3_C\}$ of
left-invariant $(1,0)$-forms, where $\omega_C^3=dz_3$ and
$$
\omega_C^1=\left(\alpha_1^C\right)^{-1}dz_1=e^{(C-i)z_3+(C+i)\bar z_3}dz_1,\quad \omega_C^2=\left(\alpha_2^C\right)^{-1}dz_2=e^{-(C-i)z_3-(C+i)\bar z_3}dz_2,
$$
satisfying the complex structure equations~\eqref{htcb}. Now, the unitary characters $\beta_1^C, \beta_2^C, \gamma_1^C, \gamma_2^C\colon\C\to\C^*$ satisfying that
$\alpha_1^C\left(\beta_1^C\right)^{-1}$,
$\alpha_2^C\left(\beta_2^C\right)^{-1}$, $\bar\alpha_1^C\left(\gamma_1^C\right)^{-1}$,
$\bar\alpha_2^C\left(\gamma_2^C\right)^{-1}$ are holomorphic and required to construct the double complex
$(B_{\Gamma_C}^{\bullet,\bullet},\db)$ are:
\begin{equation} \label{beta_g8}
\begin{array}{l}
\beta_1^C(z_3)\;=\;e^{(\bar C-i)z_3-(C+i)\bar z_3},\qquad \beta_2^C(z_3)\;=\;\beta_1^C(z_3)^{-1}\;=\;e^{-(\bar C-i)z_3+(C+i)\bar z_3},\\
\gamma_1^C(z_3)\;=\;e^{(C-i)z_3-(\bar C+i)\bar z_3},\qquad \gamma_2^C(z_3)\;=\;\gamma_1^C(z_3)^{-1}\;=\;e^{-(C-i)z_3+(\bar C+i)\bar z_3}.
\end{array}
\end{equation}

\noindent Following~\cite[Corollary 4.2]{kasuya-mathz}, \cite[Theorem 2.16]{angella-kasuya-1}, and defining for the sake of simplicity that $\beta_3^C=\gamma_3^C\equiv1$, we have that for $X_C$ the complexes 
$B_{\Gamma_C}^{\bullet,\bullet}$ and $C_{\Gamma_C}^{\bullet,\bullet}$ are generated by: 
\begin{equation}\label{complejo_B}
\begin{array}{lll}
B_{\Gamma_C}^{p,q}&=&\left\langle\beta_I^C\omega_C^I\wedge\gamma_J^C\bar\omega_C^J\,
 \middle \vert \ \begin{array}{l}\text{the restriction of }\beta_I\gamma_J\text{ on }\Gamma_C\text{ is trivial}\\
|I| = p,\,\, |J|=q
\end{array} \right\rangle,\\[15pt]
C_{\Gamma_C}^{p,q}&=&B_{\Gamma_C}^{p,q}+\bar B_{\Gamma_C}^{p,q},
\end{array}
\end{equation}
where $(p,q)\in\N^2$. Taking into account the expressions in~\eqref{beta_g8}, it turns out that the restrictions induced by the characters on the generators in~\eqref{complejo_B}  reduce in our case to satisfy one of the following conditions:
$$
\beta_1^C|_{\Gamma_C}=1,\quad\gamma_1^C|_{\Gamma_C}=1,\quad\left(\beta_1^C\gamma_1^C\right)|_{\Gamma_C}=1,\quad\left(\beta_1^C\left(\gamma_1^C\right)^{-1}\right)|_{\Gamma_C}=1.
$$

\noindent From now on, we will express the generators of the complexes $B_{\Gamma_C}^{\bullet,\bullet}$ and $C_{\Gamma_C}^{\bullet,\bullet}$  in terms of the following:

\begin{equation}\label{generadores_g8}
\left\{\begin{array}{l}
\varphi^1\;:=\;\beta_1^C\omega_C^1\;=\;e^{(C+\bar C-2i)z_3}dz_1,\\[5pt]
\varphi^2\;:=\;\beta_2^C\omega_C^2\;=\;e^{-(C+\bar C-2i)z_3}dz_2,\\[5pt]
\varphi^3\;:=\;dz_3,\end{array}\right.\qquad
\left\{\begin{array}{l}
\tilde\varphi^1\;:=\;\gamma_1^C\omega_C^{\bar1}\;=\;e^{(C+\bar C-2i)z_3}d\bar z_1,\\[5pt]
\tilde\varphi^2\;:=\;\gamma_2^C\omega_C^{\bar2}\;=\;e^{-(C+\bar C-2i)z_3}d\bar z_2,\\[5pt]
\tilde\varphi^3\;:=\; d\bar z_3,
\end{array}\right.
\end{equation}
where $ \varphi^1,\varphi^2,\varphi^3$ have bidegree $(1,0)$ and $\tilde\varphi^{1},\tilde\varphi^{2},\tilde\varphi^{3}$ have bidegree $(0,1)$.  The complex  structure equations  expressed in the
co-frame $\{\varphi^1,\varphi^2,\varphi^3, \tilde\varphi^{1},\tilde\varphi^{2},\tilde\varphi^{3}\}$ are:
\begin{equation}\label{eq_phi_g8}
\begin{array}{cc}
\left\{\begin{array}{l}
d\varphi^1\;=\;-(C+\bar C-2i)\varphi^{13},\\
d\varphi^2\;=\;(C+\bar C-2i)\varphi^{23},\\
d\varphi^3\;=\;0,
\end{array}\right.&
\left\{\begin{array}{l}
d\tilde\varphi^{1}\;=\;(C+\bar C-2i)\varphi^{3\tilde1},\\
d\tilde\varphi^{2}\;=\;-(C+\bar C-2i)\varphi^{3\tilde2},\\
d\tilde\varphi^{3}\;=\;0.
\end{array}\right.
\end{array}
\end{equation}  
In the tables below, we shorten, e.g., $\varphi^{1\tilde2}:=\varphi^1\wedge\tilde\varphi^2$.

\begin{proposition}\label{solv_g8}
Let $X_C=(G_{12}/\Gamma_C,J_C)$ be a complex solvmanifold according to Lemma~\ref{lema_lattice}.
Then, $X_C$ satisfies the
$\partial\db$-Lemma if and only if $C\neq\frac{i}{k}\in\C$, for  $0\neq k\in\Z$.
\end{proposition}

\begin{proof}
Let $C\in\C$ with $\Imag C\neq 0$ and $\Gamma'_C$ be the lattice of $\C$ provided in Lemma~\ref{lema_lattice}. The triviality of the products of the characters restricted to $\Gamma'_C$ behaves as
follows:
$$
\left\{\begin{array}{ll}
\left(\beta_1^C\gamma_1^C\right)|_{\Gamma'_C}=1,&\text{for any }C,\\[6pt]
\left(\beta_1^C\left(\gamma_1^C\right)^{-1}\right)|_{\Gamma'_C}=1,&\text{if and only if }C=\frac{i}{k}\text{ with }0\neq k\in\Z,\\[6pt]
\beta_1^C|_{\Gamma'_C}=\gamma_1^C|_{\Gamma'_C}=1,&\text{if and only if }C=\frac{i}{2k+1}\text{ with }k\in\Z.
\end{array}\right.
$$
The computations of the double complex $B_{\Gamma_C}^{\bullet,\bullet}$ and of the Hodge and the Betti numbers  for the solvmanifolds $X_C$ can be  found in Table \ref{tabla_g8}. The computations of
these numbers reveal that if
$C=\frac{i}{k}$ for $0\neq k\in\Z$ then $h_\db^{2,0}(X_C)+h_\db^{1,1}(X_C)+h_\db^{0,2}(X_C)\neq b_2(X_C)$, thus $X_C$ does not satisfy the $\partial\db$-Lemma. However, when $C\neq\frac{i}{k}$ it
turns out that the hypothesis of
Lemma~\ref{lema_B} are satisfied by using the relations
$$
\tilde \varphi^{1}\wedge\tilde \varphi^{2}\;=\; \bar\varphi^{1}\wedge\bar\varphi^{2},\quad \varphi^{2}\wedge\tilde \varphi^{1}\;=\; -\bar\varphi^{1}\wedge\bar{\tilde\varphi}^{2},\quad
\varphi^{1}\wedge\tilde\varphi^{2}\;=\; -\bar{\varphi}^{2}\wedge\bar{\tilde\varphi}^{1},
$$
of
the generators~\eqref{generadores_g8},
and the complex structure equations~\eqref{eq_phi_g8}.
Hence, all the corresponding complex solvmanifolds $X_C$ for $C\neq\frac{i}{k}$ satisfy the $\partial\db$-Lemma.
\end{proof}

\subsection{The \texorpdfstring{$\partial\db$}{partial-overline-partial}-Lemma under holomorphic deformations}\label{subsubsec:JA}\label{subsubsec:JAt}

In this section we construct complex solvmanifolds of splitting type with holomorphically trivial canonical bundle that satisfy the $\partial\bar\partial$-Lemma by deforming structures that do not satisfy this last
condition.

We consider the differential complexes $(B_{\Gamma,t}^{\bullet,\bullet},\db)$ and $(C_{\Gamma,t}^{\bullet,\bullet},\partial,\db)$ and the techniques introduced in~\cite{angella-kasuya-2} to compute
the Dolbeault and Bott-Chern cohomologies of  small deformations. In particular, by means of the computation of the cohomologies of the
complex-parallelizable structure on the Nakamura manifold, the non-closedness of the $\partial\db$-Lemma property under holomorphic
deformations is proved in~\cite[Corollary 6.1]{angella-kasuya-2}.  Using the splitting-type complex geometry on $\mathfrak{s}_{12}$, we extend this result to the following:

\begin{theorem}\label{thm:non-closedness}
 There is an infinite family of complex solvmanifolds $\{X_k\}_{k\in\Z}$  not satisfying the
$\partial\bar\partial$-Lemma and admitting a small holomorphic deformation $\{(X_{k})_t\}_{t\in\Delta_k}$ such that $(X_{k})_t$  does satisfy the
$\partial\bar\partial$-Lemma for every  $t\neq0$.

Moreover, the solvmanifolds $\{(X_{k})_t\}_{t\in\Delta_k},\, k\in\Z$ have holomorphically trivial canonical bundle and are balanced.
\end{theorem}

%\begin{theorem}\label{thm:non-closedness}
% There is an infinite family of complex solvmanifolds $\{X_k\}_{k\in\Z}$  with underlying real Lie algebra isomorphic to the one of the Nakamura manifold such that:
% \begin{itemize}
%  \item for any $k\in\Z$, the complex manifold $X_k$ is of splitting type, has holomorphically trivial canonical bundle and does not satisfy the
%$\partial\bar\partial$-Lemma;
%
%  \item for any $k\in\Z$, $X_k$ admits a small holomorphic deformation $\{X_{k,t}\}_{t\in\Delta_k}$ such that $X_{k,t}$ is of
%splitting type, has holomorphically trivial canonical bundle and does satisfy the
%$\partial\bar\partial$-Lemma for any $t\neq0$.
% \end{itemize}
%\end{theorem}

\begin{proof}
Consider the infinite family $\{X_k\}_{k\in\Z}$ where $X_k:=X_{C_k}$, $C_k=\frac{i}{2k+1}$ and $X_C$ is the complex solvmanifold described in Lemma~\ref{lema_lattice}. By
Proposition~\ref{solv_g8}, $X_k$ does not satisfy the $\partial\db$-Lemma for any $k\in\Z$.

We consider an open disc $\Delta_k:=\Delta(0,\epsilon_{k})\subset\C$ for $\epsilon_{k}>0$ small
enough,
and the family $\left\{(X_{k})_t\right\}_{t\in\Delta_k}$, $k\in\Z$, of holomorphic deformations of $X_k$  given by the $(1,0)$-co-frame
$ \left\{ \omega^1_{C_k,t} := \omega_{C_k}^1
\;,\;\omega^2_{C_k,t} := \omega^2_{C_k} \;,\;\omega^3_{C_k,t} := \omega^3_{C_k} + t\, \bar\omega^3_{C_k} \right\} $. For simplicity, we will denote $\omega^i_{C_k,t}$ as $\omega^i_{k,t}$.  The structure equations become:
$$ \left\{ \begin{array}{l}
            d \omega^1_{k,t} \;=\; - \frac{(C_k - i) + (C_k + i)\, \bar t}{1 - |t|^2}\, \omega^{13}_{k,t} - \frac{(C_k + i) + (C_k - i)\, t}{1 - |t|^2}\, \omega^{1\bar3}_{k,t},  \\[5pt]
            d \omega^2_{k,t} \;=\;   \frac{(C_k - i) + (C_k + i)\, \bar t}{1 - |t|^2}\, \omega^{23}_{k,t} + \frac{(C_k + i) + (C_k - i)\, t}{1 - |t|^2}\, \omega^{2\bar3}_{k,t},\\[5pt]
            d \omega^3_{k,t} \;=\; 0.
           \end{array}\right. \;$$

It is easy to see that the previous complex structures are of splitting type, and therefore, there exist balanced metrics (see Table~\ref{table:summary-metrics}).  Moreover, since $d\omega^{123}_{k,t} = 0$, the solvmanifolds have holomorphically trivial canonical bundle.

Taking into account the characters $\alpha_1^C$, $\alpha_2^C$, $\beta_1^C$, $\beta_2^C$, $\gamma_1^C$, $\gamma_2^C$ described in \eqref{caracteres_g8} and \eqref{beta_g8}, we define the generators of
the complex $B_{\Gamma_{C_k},t}^{\bullet,\bullet}=\wedge^{\bullet,\bullet}\langle\varphi^1_t,\varphi^2_t,\varphi^3_t,\tilde\varphi^1_t,\tilde\varphi^2_t,\tilde\varphi^3_t\rangle$ associated to the
complex solvmanifold $(X_{k})_t$:
$$ \left\{ \begin{array}{l}
        \varphi^1_{t} \;:=\; \beta_1^{C_k}\, \omega^1_{k,t} \;=\; \exp\left( - 2\,i\, z_3\right) \, d z_1, \\[5pt]
        \varphi^2_{t} \;:=\; \beta_2^{C_k}\, \omega^2_{k,t} \;=\; \exp\left(   2\,i\, z_3\right) \, d z_2, \\[5pt]
        \varphi^3_{t} \;:=\; \omega^3_{k,t} \;=\; d z_3 + t\, d \bar z_3,
       \end{array} \right.\qquad \left\{ \begin{array}{l}
        \tilde\varphi^1_{t} \;:=\; \gamma_1^{C_k}\, \omega^{\bar1}_{k,t} \;=\; \exp\left( - 2\,i\, z_3\right) \, d \bar z_1, \\[5pt]
        \tilde\varphi^2_{t} \;:=\; \gamma_2^{C_k}\, \omega^{\bar2}_{k,t} \;=\; \exp\left(   2\,i\, z_3\right) \, d \bar z_2, \\[5pt]
        \tilde\varphi^3_{t} \;:=\; \omega^{\bar3}_{k,t} \;=\; d \bar z_3 + \bar t\, d z_3,
       \end{array} \right. \;$$
where $\varphi_{t}^1$, $\varphi_{t}^2$, and $\varphi_{t}^3$ have bi-degree $(1,0)$ and $\tilde\varphi_{t}^1$, $\tilde\varphi_{t}^2$, and $\tilde\varphi_{t}^3$ have bi-degree $(0,1)$, as
explicitly described in Table \ref{table:B-JtA}. Consider also the bi-differential bi-graded double
complex
$$ C^{\bullet,\bullet}_{\Gamma_{C_k},t} \;:=\; B^{\bullet,\bullet}_{\Gamma_{C_k},t} + \overline{B^{\bullet,\bullet}_{\Gamma_{C_k},t}} $$
of vector spaces, where
$$
\begin{array}{llll}
\bar \varphi_{t}^{3}=\tilde \varphi_{t}^{3}, \quad&
\tilde\varphi_{t}^1\wedge\tilde\varphi_{t}^2=\bar\varphi_{t}^1\wedge\bar\varphi_{t}^2,\quad&
\varphi_{t}^1\wedge\bar{\tilde\varphi}_{t}^1=0, \quad&
\varphi_{t}^2\wedge\bar{\tilde\varphi}_{t}^2=0, \\
\varphi_{t}^1\wedge\tilde\varphi_{t}^2=\bar{\tilde\varphi}_{t}^1\wedge\bar\varphi_{t}^2, \quad&
\varphi_{t}^2\wedge\tilde\varphi_{t}^1=\bar{\tilde\varphi}_{t}^2\wedge\bar\varphi_{t}^1, \quad&
\varphi_{t}^1\wedge\bar\varphi_{t}^1=\bar{\tilde\varphi}_{t}^1\wedge\tilde\varphi_{t}^1, \quad&
\varphi_{t}^2\wedge\bar\varphi_{t}^2=\bar{\tilde\varphi}_{t}^2\wedge\tilde\varphi_{t}^2,
\end{array}
$$
as explicitly described in Table \ref{table:B-JtA}.  We compute the structure equations:

\begin{center}
\resizebox{\textwidth}{!}{
\begin{tabular}{>{$}l<{$}>{$}l<{$}}
\left\{ \begin{array}{l}
d\varphi_{t}^1 \;=\;   \frac{2\, i}{1-|t|^2}\, \varphi_{t}^1\wedge\varphi_{t}^3 - \frac{2\, t\,i}{1-|t|^2}\, \varphi_{t}^1\wedge\bar\varphi_{t}^3, \\[5pt]
d\varphi_{t}^2 \;=\; - \frac{2\, i}{1-|t|^2}\, \varphi_{t}^2\wedge\varphi_{t}^3 + \frac{2\, t\, i}{1-|t|^2}\, \varphi_{t}^2\wedge\bar\varphi_{t}^3,\\[5pt]
d\varphi_{t}^3 \;=\; 0, \\[5pt]
d\bar{\tilde\varphi}_{t}^1 \;=\; - \frac{2\, i}{1-|t|^2}\, \bar{\tilde\varphi}_{t}^1\wedge\bar\varphi_{t}^3 + \frac{2\, \bar t\, i}{1-|t|^2}\, \bar{\tilde\varphi}_{t}^1 \wedge \varphi^3_{t},
\\[5pt]
d\bar{\tilde\varphi}_{A_k,t}^2 \;=\;   \frac{2\, i}{1-|t|^2}\, \bar{\tilde\varphi}_{t}^2\wedge\bar\varphi_{t}^3 - \frac{2\, \bar t\, i}{1-|t|^2}\, \bar{\tilde\varphi}_{t}^2 \wedge
\varphi_{t}^3,
\end{array} \right. &
\left\{ \begin{array}{l}
d\tilde\varphi_{t}^1 \;=\; - \frac{2\, i}{1-|t|^2}\, \varphi_{t}^3\wedge\tilde\varphi_{t}^1 - \frac{2\, t\, i}{1-|t|^2}\, \tilde\varphi_{t}^1\wedge\bar\varphi_{t}^3, \\[5pt]
d\tilde\varphi_{t}^2 \;=\;   \frac{2\, i}{1-|t|^2}\, \varphi_{t}^3\wedge\tilde\varphi_{t}^2 + \frac{2\, t\, i}{1-|t|^2}\, \tilde\varphi_{t}^2\wedge\bar\varphi_{t}^3, \\[5pt]
d\tilde\varphi_{t}^3 \;=\; 0, \\[5pt]
d\bar\varphi_{t}^1 \;=\; - \frac{2\, i}{1-|t|^2}\, \bar\varphi_{t}^1\wedge\bar\varphi_{t}^3 - \frac{2\, \bar t\, i}{1-|t|^2}\, \varphi_{t}^3\wedge\bar\varphi_{t}^1, \\[5pt]
d\bar\varphi_{t}^2 \;=\;   \frac{2\, i}{1-|t|^2}\, \bar\varphi_{t}^2\wedge\bar\varphi_{t}^3 + \frac{2\, \bar t\, i}{1-|t|^2}\, \varphi_{t}^3\wedge\bar\varphi_{t}^2.
\end{array} \right.
\end{tabular}
}\\
\end{center}

By \cite[Corollary 1.3]{kasuya-mathz}, \cite[Theorem 2.16]{angella-kasuya-1} and \cite[Theorem 1.1, Theorem 1.2]{angella-kasuya-2} (with respect to the Hermitian metric
$g := \varphi_{t}^1\odot\tilde\varphi_{t}^1 + \varphi_{t}^2\odot\tilde\varphi_{t}^2 + \varphi^3_{t}\odot\bar\varphi^3_{t}$), such complexes allow to compute the Dolbeault cohomology and
the Bott-Chern cohomology of $(X_{k})_t$.

Note that the differential bi-graded algebra $\left( B^{\bullet,\bullet}_{\Gamma_{C_k},t},\, \bar\partial \right)$ and the bi-differential double complex
$\left( C^{\bullet,\bullet}_{\Gamma_{C_k},t},\, \partial,\, \bar\partial \right)$ of vector spaces do not depend on $C_k$; in particular, for any $C_k=\frac{i}{2k+1}$ varying $k\in\Z$, they are
isomorphic to the
corresponding object with $k=-1$, that is, $C=-i$. Hence, it follows that the computations and the results in \cite[\S4]{angella-kasuya-2} still hold for any $C_k$.
More precisely, we recall in Table~\ref{table:BC-JAt} the harmonic representatives in the Dolbeault cohomology, respectively Bott-Chern cohomology, with
respect to the metric $g$.

In Table \ref{table:nak-def-cohomologies}, we summarize the results of the computations by giving the Betti, Hodge, and Bott-Chern numbers of the complex solvmanifolds $X_k$ and of its
 small deformations $(X_{k})_t$. From the results summarized in Table~\ref{table:BC-JAt}, we get that $(X_{k})_t$ satisfies the
$\partial\bar\partial$-Lemma for any $t\neq0$.
\end{proof}

\begin{table}[!hb]
 \centering
\begin{tabular}{||>{$\mathbf\bgroup}c<{\mathbf\egroup$} || >{$}c<{$} | >{$}c<{$} >{$}c<{$} | >{$}c<{$} >{$}c<{$} ||}
\toprule
\multirow{2}{*}{\textnormal{$\dim_\C H_{\sharp}^{\bullet,\bullet}(X_{k})_t$} }& & \multicolumn{2}{c|}{$t = 0$} & \multicolumn{2}{c|}{$t \neq 0$} \\
& dR & \bar\partial & BC & \bar\partial & BC \\
\toprule
(0,0) & 1 & 1 & 1 & 1 & 1 \\
\midrule[0.02em]
(1,0) & \multirow{2}{*}{2} & 3 & 1 & 1 & 1 \\[3pt]
(0,1) & & 3 & 1 & 1 &  1 \\
\midrule[0.02em]
(2,0) & \multirow{3}{*}{5} & 3 & 3 & 1 & 1 \\[3pt]
(1,1) & & 9 & 7 & 3 &  3 \\[3pt]
(0,2) & & 3 & 3 & 1 &  1 \\
\midrule[0.02em]
(3,0) & \multirow{4}{*}{8} & 1 & 1 & 1 & 1 \\[3pt]
(2,1) & & 9 & 9 & 3 &  3 \\[3pt]
(1,2) & & 9 & 9 & 3 &  3 \\[3pt]
(0,3) & & 1 & 1 & 1 &  1 \\
\midrule[0.02em]
(3,1) & \multirow{3}{*}{5} & 3 & 3 & 1 & 1 \\[3pt]
(2,2) & & 9 & 11 & 3 & 3 \\[3pt]
(1,3) & & 3 & 3 & 1 &  1 \\
\midrule[0.02em]
(3,2) & \multirow{2}{*}{2} & 3 & 5 & 1 & 1 \\[3pt]
(2,3) & & 3 & 5 & 1 & 1 \\
\midrule[0.02em]
(3,3) & 1 & 1 & 1 & 1 & 1 \\
\bottomrule
\end{tabular}
\medskip
\caption{Summary of the dimensions of the de Rham, Dolbeault, and Bott-Chern cohomologies of the complex solvmanifolds $(X_{k})_t$.}
\label{table:nak-def-cohomologies}
\end{table}

%%%%%%%%%%%%%%%%%%%%%%%%%%%%%%%%%%%%%%%%%%%%
%%%%%%%%%%%%%%%%%%%%%%%%%%%%%%%%%%%%%%%%%%%%%%%%%%%%%%%%%%%%%%%%%%%%%%%%%%%%%%%%%%%%

%\appendix
%\section{Tables}
\begin{landscape}
\begin{table}[!hb]
 \centering
\scalebox{0.75}{
\begin{tabular}{||>{$\mathbf\bgroup}l<{\mathbf\egroup$} || >{$}l<{$} >{$\mathbf\bgroup}l<{\mathbf\egroup$} >{$\mathbf\bgroup}l<{\mathbf\egroup$} ||
>{$}l<{$}>{$\mathbf\bgroup}l<{\mathbf\egroup$}>{$\mathbf\bgroup}l<{\mathbf\egroup$}||>{$}l<{$}>{$\mathbf\bgroup}l<{\mathbf\egroup$}>{$\mathbf\bgroup}l<{\mathbf\egroup$}||}
\toprule
 & B_{\Gamma_C}^{\bullet,\bullet} (C=\frac{i}{2k+1},\, k\in\Z) & h^{\bullet,\bullet}_\db &b_\bullet& B_{\Gamma_C}^{\bullet,\bullet}  (C=\frac{i}{2k},\, 0\neq k\in\Z)&h^{\bullet,\bullet}_\db&b_\bullet&
B_{\Gamma_C}^{\bullet,\bullet}  (C\neq\frac{i}{k},\, 0\neq k\in\Z)&h^{\bullet,\bullet}_\db&b_\bullet\\
\toprule
(1,0) & \C \left\langle \varphi^1,\varphi^2,\varphi^3 \right\rangle &3&\multirow{2}{*}{2}& \C \left\langle \varphi^3 \right\rangle&1&\multirow{2}{*}{2}& \C \left\langle \varphi^3
\right\rangle&1&\multirow{2}{*}{2} \\[5pt]
(0,1) & \C \left\langle \tilde\varphi^1,\tilde\varphi^2,\varphi^{\bar3} \right\rangle &3&& \C \left\langle \varphi^{\bar 3} \right\rangle&1&& \C \left\langle \varphi^{\bar3} \right\rangle&1&
\\[5pt]\hline
(2,0) & \C\left\langle \varphi^{12},\varphi^{13},\varphi^{23}\right\rangle &3&\multirow{3}{*}{5}&
\C\left\langle\varphi^{12}\right\rangle&1&\multirow{3}{*}{5}&\C\left\langle\varphi^{12}\right\rangle&1&\multirow{3}{*}{5}\\[5pt]
(1,1) & \C\left\langle   \varphi^{1\tilde1},\varphi^{1\tilde2},\varphi^{1\bar3}, \varphi^{2\tilde1},\varphi^{2\tilde2},\varphi^{2\bar3},
\varphi^{3\tilde1},\varphi^{3\tilde2},\varphi^{3\bar3}\right\rangle &9&&\C\left\langle
\varphi^{1\tilde1},\varphi^{1\tilde2},\varphi^{2\tilde1},\varphi^{2\tilde2},\varphi^{3\bar3}\right\rangle&5&&\C\left\langle
\varphi^{1\tilde2},\varphi^{2\tilde1},\varphi^{3\bar3}\right\rangle&3&\\[5pt]
(0,2) & \C\left\langle \varphi^{\tilde 1\tilde 2},\varphi^{\tilde 1\bar3},\varphi^{\tilde 2\bar3}\right\rangle &3&& \C\left\langle \varphi^{\tilde 1\tilde 2}\right\rangle&1&& \C\left\langle
\varphi^{\tilde 1\tilde 2}\right\rangle&1&\\[5pt]\hline
(3,0) & \C\left\langle \varphi^{123} \right\rangle &1&\multirow{4}{*}{8}& \C\left\langle  \varphi^{123} \right\rangle&1&\multirow{4}{*}{8}& \C\left\langle  \varphi^{123}
\right\rangle&1&\multirow{4}{*}{8}\\[5pt]
(2,1) & \C\left\langle
\varphi^{12\tilde1},\varphi^{12\tilde2},\varphi^{12\bar3},\varphi^{13\tilde1},\varphi^{13\tilde2},\varphi^{13\bar3},\varphi^{23\tilde1},\varphi^{23\tilde2},\varphi^{23\bar3}\right\rangle &9&&
\C\left\langle \varphi^{12\bar3},\varphi^{13\tilde1},\varphi^{13\tilde2},\varphi^{23\tilde1},\varphi^{23\tilde2}\right\rangle&5&& \C\left\langle
\varphi^{12\bar3},\varphi^{13\tilde2},\varphi^{23\tilde1}\right\rangle&3&\\[5pt]
(1,2) & \C\left\langle
\varphi^{1\tilde1\tilde2},\varphi^{1\tilde1\bar3},\varphi^{1\tilde2\bar3},\varphi^{2\tilde1\tilde2},\varphi^{2\tilde1\bar3},\varphi^{2\tilde2\bar3},\varphi^{3\tilde1\tilde2},\varphi^{3\tilde1\bar3},
\varphi^{3\tilde2\bar3}\right\rangle &9&& \C\left\langle \varphi^{1\tilde1\bar3},\varphi^{1\tilde2\bar3},\varphi^{2\tilde1\bar3},\varphi^{2\tilde2\bar3},\varphi^{3\tilde1\tilde2}\right\rangle&5&&
\C\left\langle \varphi^{1\tilde2\bar3},\varphi^{2\tilde1\bar3},\varphi^{3\tilde1\tilde2}\right\rangle&3&\\[5pt]
(0,3) & \C\left\langle  \varphi^{\tilde1\tilde2\bar3} \right\rangle &1&& \C\left\langle  \varphi^{\tilde1\tilde2\bar3} \right\rangle&1&& \C\left\langle  \varphi^{\tilde1\tilde2\bar3}
\right\rangle&1&\\[5pt]\hline
(3,1) & \C\left\langle \varphi^{123\tilde1},\varphi^{123\tilde2},\varphi^{123\bar3} \right\rangle &3&\multirow{3}{*}{5}& \C\left\langle \varphi^{123\bar3}  \right\rangle&1&\multirow{3}{*}{5}&
\C\left\langle \varphi^{123\bar3}  \right\rangle&1&\multirow{3}{*}{5}\\[5pt]
(2,2) & \C\left\langle
\varphi^{12\tilde1\tilde2},\varphi^{12\tilde1\bar3},\varphi^{12\tilde2\bar3},\varphi^{13\tilde1\tilde2},\varphi^{13\tilde1\bar3},\varphi^{13\tilde2\bar3},\varphi^{23\tilde1\tilde2},\varphi^{
23\tilde1\bar3},\varphi^{23\tilde2\bar3}\right\rangle&9&& \C\left\langle \varphi^{12\tilde1\tilde2}, \varphi^{13\tilde1\bar3}, \varphi^{13\tilde2\bar3}, \varphi^{23\tilde1\bar3},
\varphi^{23\tilde2\bar3} \right\rangle&5&& \C\left\langle \varphi^{12\tilde1\tilde2}, \varphi^{13\tilde2\bar3}, \varphi^{23\tilde1\bar3}\right\rangle&3&\\[5pt]
(1,3) & \C\left\langle  \varphi^{1\tilde1\tilde2\bar3},\varphi^{2\tilde1\tilde2\bar3},\varphi^{3\tilde1\tilde2\bar3} \right\rangle &3&& \C\left\langle \varphi^{3\tilde1\tilde2\bar3} \right\rangle&1&&
\C\left\langle \varphi^{3\tilde1\tilde2\bar3} \right\rangle&1&\\[5pt]\hline
(3,2) & \C\left\langle  \varphi^{123\tilde1\tilde2},\varphi^{123\tilde1\bar3},\varphi^{123\tilde2\bar3}\right\rangle &3&\multirow{2}{*}{2}& \C\left\langle
\varphi^{123\tilde1\tilde2}\right\rangle&1&\multirow{2}{*}{2}& \C\left\langle   \varphi^{123\tilde1\tilde2}\right\rangle&1&\multirow{2}{*}{2}\\[5pt]
(2,3) & \C\left\langle   \varphi^{12\tilde1\tilde2\bar3},\varphi^{13\tilde1\tilde2\bar3},\varphi^{23\tilde1\tilde2\bar3}\right\rangle &3&& \C\left\langle
\varphi^{12\tilde1\tilde2\bar3}\right\rangle&1&& \C\left\langle   \varphi^{12\tilde1\tilde2\bar3}\right\rangle&1&\\[5pt]\hline
(3,3) & \C\left\langle   \varphi^{123\tilde1\tilde2\bar3}\right\rangle &1&1& \C\left\langle    \varphi^{123\tilde1\tilde2\bar3}\right\rangle&1&1& \C\left\langle
\varphi^{123\tilde1\tilde2\bar3}\right\rangle&1&1\\[5pt]
\bottomrule
\end{tabular}
}
\medskip
\caption{The double complex $B_{\Gamma_C}^{\bullet,\bullet}$ for computing the Dolbeault cohomology of the complex solvmanifolds $X_k$,
described in Lemma~\ref{lema_lattice}.}\label{tabla_g8}
\end{table}
\end{landscape}

\newpage
\begin{landscape}

\begin{table}[!hb]
 \centering
 \resizebox{\textwidth}{!}{
\begin{tabular}{||>{$\mathbf\bgroup}l<{\mathbf\egroup$} || >{$}l<{$}|| >{$}l<{$}||}
\toprule
 & B^{\bullet,\bullet}_{\Gamma_{C_k}, t}& C^{\bullet,\bullet}_{\Gamma_{C_k}, t}\\
\toprule
(0,0) & \C \left\langle 1 \right\rangle & \C \left\langle 1 \right\rangle \\
\midrule[0.02em]
(1,0) & \C \left\langle \varphi_t^1,\; \varphi_t^2,\; \varphi_t^3 \right\rangle  & \C \left\langle \varphi_t^1,\; \varphi_t^2,\; \varphi_t^3,\; \varphi_t^{\bar{\tilde1}},\; \varphi_t^{\bar{\tilde2}} \right\rangle  \\[5pt]
(0,1) & \C \left\langle \varphi_t^{\tilde1},\; \varphi_t^{\tilde2},\; \varphi_t^{\bar3} \right\rangle & \C \left\langle \varphi_t^{\tilde1},\; \varphi_t^{\tilde2},\; \varphi_t^{\bar3},\; \varphi_t^{\bar1},\; \varphi_t^{\bar2} \right\rangle \\
\midrule[0.02em]
(2,0) & \C \left\langle \varphi_t^{12},\; \varphi_t^{13},\; \varphi_t^{23} \right\rangle & \C \left\langle \varphi_t^{12},\; \varphi_t^{13},\; \varphi_t^{23},\; \varphi_t^{\bar{\tilde1}3},\; \varphi_t^{\bar{\tilde2}3} \right\rangle \\[5pt]
(1,1) & \C \left\langle \varphi_t^{1\tilde1},\; \varphi_t^{1\tilde2},\; \varphi_t^{1\bar3},\; \varphi_t^{2\tilde1},\; \varphi_t^{2\tilde2},\; \varphi_t^{2\bar3},\; \varphi_t^{3\tilde1},\;
\varphi_t^{3\tilde2},\; \varphi_t^{3\bar3} \right\rangle &
\C \left\langle \varphi_t^{1\tilde1},\;\varphi_t^{1\tilde2},\; \varphi_t^{1\bar3},\; \varphi_t^{2\tilde1},\; \varphi_t^{2\tilde2},\; \right.\\[5pt]
&&\left.
\quad\varphi_t^{2\bar3},\; \varphi_t^{3\tilde1},\; \varphi_t^{3\tilde2},\; \varphi_t^{3\bar3},\; \varphi_t^{\bar{\tilde1}\bar1},\; \varphi_t^{3\bar1},\; \varphi_t^{\bar{\tilde2}\bar2},\; \varphi_t^{3\bar2},\; \varphi_t^{\bar{\tilde1}\bar3},\;
\varphi_t^{\bar{\tilde2}\bar3} \right\rangle \\[5pt]
(0,2) & \C \left\langle \varphi_t^{\tilde1\tilde2},\; \varphi_t^{\tilde1\bar3},\; \varphi_t^{\tilde2\bar3} \right\rangle & \C \left\langle \varphi_t^{\tilde1\tilde2},\; \varphi_t^{\tilde1\bar3},\; \varphi_t^{\tilde2\bar3},\; \varphi_t^{\bar1\bar3},\; \varphi_t^{\bar2\bar3} \right\rangle \\
\midrule[0.02em]
(3,0) & \C \left\langle \varphi_t^{123} \right\rangle & \C \left\langle \varphi_t^{123} \right\rangle \\[5pt]
(2,1) & \C \left\langle \varphi_t^{12\tilde1},\; \varphi_t^{12\tilde2},\; \varphi_t^{12\bar3},\; \varphi_t^{13\tilde1},\; \varphi_t^{13\tilde2},\; \varphi_t^{13\bar3},\; \varphi_t^{23\tilde1},\;
\varphi_t^{23\tilde2},\; \varphi_t^{23\bar3} \right\rangle & \C \left\langle \varphi_t^{12\tilde1},\; \varphi_t^{12\tilde2},\; \varphi_t^{12\bar3},\; \varphi_t^{13\tilde1},\; \varphi_t^{13\tilde2},\; \varphi_t^{13\bar3},\; \varphi_t^{23\tilde1},\;
\varphi_t^{23\tilde2},\; \right.\\[5pt]
&&\left.
\quad
\varphi_t^{23\bar3},\; \varphi_t^{\bar{\tilde1}\bar{\tilde2}\bar1},\; \varphi_t^{\bar{\tilde1}3\bar1},\; \varphi_t^{\bar{\tilde1}\bar{\tilde2}\bar2},\;
\varphi_t^{\bar{\tilde2}3\bar2},\;  \varphi_t^{\bar{\tilde1}3\bar3},\; \varphi_t^{\bar{\tilde2}3\bar3} \right\rangle \\[5pt]
(1,2) & \C \left\langle \varphi_t^{1\tilde1\tilde2},\; \varphi_t^{1\tilde1\bar3},\; \varphi_t^{1\tilde2\bar3},\; \varphi_t^{2\tilde1\tilde2},\; \varphi_t^{2\tilde1\bar3},\;
\varphi_t^{2\tilde2\bar3},\; \varphi_t^{3\tilde1\tilde2},\; \varphi_t^{3\tilde1\bar3},\; \varphi_t^{3\tilde2\bar3} \right\rangle & \C \left\langle \varphi_t^{1\tilde1\tilde2},\; \varphi_t^{1\tilde1\bar3},\; \varphi_t^{1\tilde2\bar3},\; \varphi_t^{2\tilde1\tilde2},\; \varphi_t^{2\tilde1\bar3},\;
\varphi_t^{2\tilde2\bar3},\; \varphi_t^{3\tilde1\tilde2},\; \varphi_t^{3\tilde1\bar3},\; \varphi_t^{3\tilde2\bar3},\; \right.\\[5pt]
&&\left.
\quad\varphi_t^{\bar{\tilde1}\bar1\bar2},\; \varphi_t^{\bar{\tilde2}\bar1\bar2},\;
\varphi_t^{\bar{\tilde1}\bar1\bar3},\; \varphi_t^{3\bar1\bar3},\; \varphi_t^{\bar{\tilde2}\bar2\bar3},\; \varphi_t^{3\bar2\bar3} \right\rangle \\[5pt]
(0,3) & \C \left\langle \varphi_t^{\tilde1\tilde2\bar3} \right\rangle & \C \left\langle \varphi_t^{\tilde1\tilde2\bar3} \right\rangle \\
\midrule[0.02em]
(3,1) & \C \left\langle \varphi_t^{123\tilde1},\; \varphi_t^{123\tilde2},\; \varphi_t^{123\bar3} \right\rangle & \C \left\langle \varphi_t^{123\tilde1},\; \varphi_t^{123\tilde2},\; \varphi_t^{123\bar3},\; \varphi_t^{\bar{\tilde1}\bar{\tilde2}3\bar1},\; \varphi_t^{\bar{\tilde1}\bar{\tilde2}3\bar2}
\right\rangle \\[5pt]
(2,2) & \C \left\langle \varphi_t^{12\tilde1\tilde2},\; \varphi_t^{12\tilde1\bar3},\; \varphi_t^{12\tilde2\bar3},\; \varphi_t^{13\tilde1\tilde2},\; \varphi_t^{13\tilde1\bar3},\;
\varphi_t^{13\tilde2\bar3},\; \varphi_t^{23\tilde1\tilde2},\; \right. & \C \left\langle \varphi_t^{12\tilde1\tilde2},\; \varphi_t^{12\tilde1\bar3},\; \varphi_t^{12\tilde2\bar3},\; \varphi_t^{13\tilde1\tilde2},\; \varphi_t^{13\tilde1\bar3},\;
\varphi_t^{13\tilde2\bar3},\; \varphi_t^{23\tilde1\tilde2},\; \varphi_t^{23\tilde1\bar3},\; \right.\\[5pt]
&\quad\left.\varphi_t^{23\tilde1\bar3},\; \varphi_t^{23\tilde2\bar3} \right\rangle&\left.
\quad\varphi_t^{23\tilde2\bar3},\; \varphi_t^{\bar{\tilde1}3\bar1\bar2},\;
\varphi_t^{\bar{\tilde2}3\bar1\bar2},\; \varphi_t^{\bar{\tilde1}\bar{\tilde2}\bar1\bar3},\; \varphi_t^{\bar{\tilde1}\bar{\tilde2}\bar2\bar3},\; \varphi_t^{\bar{\tilde1}3\bar1\bar3},\;
\varphi_t^{\bar{\tilde2}3\bar2\bar3} \right\rangle \\[5pt]
(1,3) & \C \left\langle \varphi_t^{1\tilde1\tilde2\bar3},\; \varphi_t^{2\tilde1\tilde2\bar3},\; \varphi_t^{3\tilde1\tilde2\bar3} \right\rangle & \C \left\langle \varphi_t^{1\tilde1\tilde2\bar3},\; \varphi_t^{2\tilde1\tilde2\bar3},\; \varphi_t^{3\tilde1\tilde2\bar3},\; \varphi_t^{\bar{\tilde1}\bar1\bar2\bar3},\;
\varphi_t^{\bar{\tilde2}\bar1\bar2\bar3} \right\rangle \\
\midrule[0.02em]
(3,2) & \C \left\langle \varphi_t^{123\tilde1\tilde2},\; \varphi_t^{123\tilde1\bar3},\; \varphi_t^{123\tilde2\bar3} \right\rangle & \C \left\langle \varphi_t^{123\tilde1\tilde2},\; \varphi_t^{123\tilde1\bar3},\; \varphi_t^{123\tilde2\bar3},\; \varphi_t^{\bar{\tilde1}\bar{\tilde2}3\bar1\bar3},\;
\varphi_t^{\bar{\tilde1}\bar{\tilde2}3\bar2\bar3} \right\rangle \\[5pt]
(2,3) & \C \left\langle \varphi_t^{12\tilde1\tilde2\bar3},\; \varphi_t^{13\tilde1\tilde2\bar3},\; \varphi_t^{23\tilde1\tilde2\bar3} \right\rangle & \C \left\langle \varphi_t^{12\tilde1\tilde2\bar3},\; \varphi_t^{13\tilde1\tilde2\bar3},\; \varphi_t^{23\tilde1\tilde2\bar3},\; \varphi_t^{\bar{\tilde1}3\bar1\bar2\bar3},\;
\varphi_t^{\bar{\tilde2}3\bar1\bar2\bar3} \right\rangle \\
\midrule[0.02em]
(3,3) & \C \left\langle \varphi_t^{123\tilde1\tilde2\bar3} \right\rangle & \C \left\langle \varphi_t^{123\tilde1\tilde2\bar3} \right\rangle \\
\bottomrule
\end{tabular}}
\bigskip
\caption{The (double-)complexes $B^{\bullet,\bullet}_{\Gamma_{C_k}, t}$ in \cite[Corollary 4.2]{kasuya-mathz} and  $C^{\bullet,\bullet}_{\Gamma_{C_k}, t}$ in \cite[Theorem 2.16]{angella-kasuya-1} for computing the Dolbeault and the Bott-Chern cohomology of the complex solvmanifolds $(X_k)_t$.}
\label{table:B-JtA}
\end{table}

\end{landscape}

\begin{landscape}

\begin{table}[!hb]
 \centering
% \resizebox{\textwidth}{!}{
\scalebox{0.63}{
\begin{tabular}{||>{$\mathbf\bgroup}l<{\mathbf\egroup$} || >{$}l<{$}  || >{$}l<{$} || >{$}l<{$}  || >{$}l<{$}  ||}
\toprule
% \multirow{2}{*}{$H^{\bullet,\bullet}_{{BC}_{J_{A,t}}}(X)$} & t = 0 & t \neq 0 & t = 0 & t \neq 0 \\
%  &  & &   & \\
& H^{\bullet,\bullet}_{{\db}_{J_{C_k,t=0}}}(X) & H^{\bullet,\bullet}_{{\db}_{J_{C_k,t\neq0}}}(X) & H^{\bullet,\bullet}_{{BC}_{J_{C_k,t=0}}}(X) & H^{\bullet,\bullet}_{{BC}_{J_{C_k,t\neq0}}}(X) \\

\toprule
(0,0) & \C \left\langle 1 \right\rangle  & \C \left\langle 1 \right\rangle & \C \left\langle 1 \right\rangle  & \C \left\langle 1 \right\rangle  \\
\midrule[0.02em]
(1,0) &   \C \left\langle \varphi_0^{1},\, \varphi_0^{2},\, \varphi_0^{3} \right\rangle &
\C \left\langle \varphi_t^{3} \right\rangle & \C \left\langle \varphi_0^{3} \right\rangle &  \C \left\langle \varphi_t^{3} \right\rangle   \\[5pt]
(0,1) & \C \left\langle \varphi_0^{\tilde1},\, \varphi_0^{\tilde2},\, \varphi_0^{\bar3} \right\rangle
& \C \left\langle \varphi_t^{\bar3} \right\rangle & \C \left\langle \varphi_0^{\bar3} \right\rangle  & \C \left\langle \varphi_t^{\bar3} \right\rangle   \\
\midrule[0.02em]
(2,0) & \C \left\langle \varphi_0^{12},\,
\varphi_0^{13},\, \varphi_0^{23} \right\rangle  & \C \left\langle \varphi_t^{12} \right\rangle & \C \left\langle \varphi_0^{12},\, \varphi_0^{13},\, \varphi_0^{23} \right\rangle  & \C \left\langle \varphi_t^{12} \right\rangle  \\[5pt]
(1,1) & \C \left\langle \varphi_0^{1\tilde1},\,
\varphi_0^{1\tilde2},\, \varphi_0^{1\bar3},\, \varphi_0^{2\tilde1},\, \varphi_0^{2\tilde2},\, \varphi_0^{2\bar3},\, \varphi_0^{3\tilde1},\,
\varphi_0^{3\tilde2},\, \varphi_0^{3\bar3} \right\rangle  & \C \left\langle \varphi_t^{1\tilde2},\, \varphi_t^{2\tilde1},\, \varphi_t^{3\bar3} \right\rangle & \C \left\langle \varphi_0^{1\tilde2},\, \varphi_0^{2\tilde1},\, \varphi_0^{3\tilde1},\, \varphi_0^{3\tilde2},\, \varphi_0^{3\bar3},\, \varphi_0^{\bar{\tilde1}\bar3},\,
\varphi_0^{\bar{\tilde2}\bar3} \right\rangle  & \C \left\langle \varphi_t^{1\tilde2},\, \varphi_t^{2\tilde1},\, \varphi_t^{3\bar3} \right\rangle  \\[5pt]
(0,2) & \C \left\langle \varphi_0^{\tilde1\tilde2},\, \varphi_0^{\tilde1\bar3},\, \varphi_0^{\tilde2\bar3} \right\rangle &  \C \left\langle \varphi_t^{\tilde1\tilde2} \right\rangle & \C \left\langle \varphi_0^{\tilde1\tilde2},\, \varphi_0^{\bar1\bar3},\, \varphi_0^{\bar2\bar3} \right\rangle  & \C \left\langle \varphi_t^{\tilde1\tilde2} \right\rangle  \\
\midrule[0.02em]
(3,0) & \C \left\langle \varphi_0^{123} \right\rangle &  \C \left\langle \varphi_t^{123} \right\rangle  & \C \left\langle \varphi_0^{123} \right\rangle &  \C \left\langle \varphi_t^{123} \right\rangle \\[5pt]
(2,1) & \C \left\langle \varphi_0^{12\tilde1},\, \varphi_0^{12\tilde2},\, \varphi_0^{12\bar3},\, \varphi_0^{13\tilde1},\, \varphi_0^{13\tilde2},\, \varphi_0^{13\bar3},\, \varphi_0^{23\tilde1},\,
\varphi_0^{23\tilde2},\, \varphi_0^{23\bar3} \right\rangle &  \C \left\langle \varphi_t^{12\bar3},\, \varphi_t^{13\tilde2},\, \varphi_t^{23\tilde1} \right\rangle & \C \left\langle \varphi_0^{12\bar3},\, \varphi_0^{13\tilde1},\, \varphi_0^{13\tilde2},\, \varphi_0^{13\bar3},\, \varphi_0^{23\tilde1},\, \varphi_0^{23\tilde2},\,
\varphi_0^{23\bar3},\,
\varphi_0^{\bar{\tilde1}3\bar3},\, \varphi_0^{\bar{\tilde2}3\bar3} \right\rangle  & \C \left\langle \varphi_t^{12\bar3},\, \varphi_t^{13\tilde2},\, \varphi_t^{23\tilde1} \right\rangle \\[5pt]
(1,2) & \C \left\langle \varphi_0^{1\tilde1\tilde2},\, \varphi_0^{1\tilde1\bar3},\, \varphi_0^{1\tilde2\bar3},\, \varphi_0^{2\tilde1\tilde2},\, \varphi_0^{2\tilde1\bar3},\,
\varphi_0^{2\tilde2\bar3},\, \varphi_0^{3\tilde1\tilde2},\, \varphi_0^{3\tilde1\bar3},\, \varphi_0^{3\tilde2\bar3} \right\rangle  & \C \left\langle \varphi_t^{1\tilde2\bar3},\,
\varphi_t^{2\tilde1\bar3},\, \varphi_t^{3\tilde1\tilde2} \right\rangle & \C \left\langle \varphi_0^{1\tilde2\bar3},\, \varphi_0^{2\tilde1\bar3},\, \varphi_0^{3\tilde1\tilde2},\, \varphi_0^{3\tilde1\bar3},\, \varphi_0^{3\tilde2\bar3},\,
\varphi_0^{\bar{\tilde1}\bar1\bar3},\, \varphi_0^{3\bar1\bar3},\, \varphi_0^{\bar{\tilde2}\bar2\bar3},\, \varphi_0^{3\bar2\bar3} \right\rangle  & \C \left\langle \varphi_t^{1\tilde2\bar3},\,
\varphi_t^{2\tilde1\bar3},\, \varphi_t^{3\tilde1\tilde2} \right\rangle  \\[5pt]
(0,3) & \C \left\langle \varphi_0^{\tilde1\tilde2\bar3} \right\rangle  & \C \left\langle \varphi_t^{\tilde1\tilde2\bar3} \right\rangle & \C \left\langle \varphi_0^{\tilde1\tilde2\bar3} \right\rangle  & \C \left\langle \varphi_t^{\tilde1\tilde2\bar3} \right\rangle  \\
\midrule[0.02em]
(3,1) & \C \left\langle \varphi_0^{123\tilde1},\, \varphi_0^{123\tilde2},\, \varphi_0^{123\bar3} \right\rangle  & \C \left\langle \varphi_t^{123\bar3} \right\rangle & \C \left\langle \varphi_0^{123\tilde1},\, \varphi_0^{123\tilde2},\, \varphi_0^{123\bar3} \right\rangle &  \C \left\langle \varphi_t^{123\bar3} \right\rangle   \\[5pt]
(2,2) & \C \left\langle \varphi_0^{12\tilde1\tilde2},\, \varphi_0^{12\tilde1\bar3},\, \varphi_0^{12\tilde2\bar3},\, \varphi_0^{13\tilde1\tilde2},\, \varphi_0^{13\tilde1\bar3},\,
\varphi_0^{13\tilde2\bar3},\, \varphi_0^{23\tilde1\tilde2},\, \varphi_0^{23\tilde1\bar3},\, \varphi_0^{23\tilde2\bar3} \right\rangle  & \C \left\langle \varphi_t^{12\tilde1\tilde2},\,
\varphi_t^{13\tilde2\bar3},\, \varphi_t^{23\tilde1\bar3} \right\rangle &\C \left\langle \varphi_0^{12\tilde1\tilde2},\, \varphi_0^{13\tilde1\tilde2},\, \varphi_0^{13\tilde1\bar3},\, \varphi_0^{13\tilde2\bar3},\, \varphi_0^{23\tilde1\tilde2},\,
\varphi_0^{23\tilde1\bar3},\, \varphi_0^{23\tilde2\bar3},\, \varphi_0^{\bar{\tilde1}\bar{\tilde2}\bar1\bar3},\, \varphi_0^{\bar{\tilde1}\bar{\tilde2}\bar2\bar3},\,
\varphi_0^{\bar{\tilde1}3\bar1\bar3},\, \varphi_0^{\bar{\tilde2}3\bar2\bar3} \right\rangle  & \C \left\langle \varphi_t^{12\tilde{1}\tilde{2}},\, \varphi_t^{13\tilde{2}\bar{3}},\,
\varphi_t^{23\tilde{1}\bar{3}} \right\rangle   \\[5pt]
(1,3) & \C \left\langle \varphi_0^{1\tilde1\tilde2\bar3},\, \varphi_0^{2\tilde1\tilde2\bar3},\, \varphi_0^{3\tilde1\tilde2\bar3} \right\rangle  & \C \left\langle \varphi_t^{3\tilde1\tilde2\bar3}
\right\rangle &\C \left\langle \varphi_0^{1\tilde1\tilde2\bar3},\, \varphi_0^{2\tilde1\tilde2\bar3},\, \varphi_0^{3\tilde1\tilde2\bar3} \right\rangle &  \C \left\langle \varphi_t^{3\tilde1\tilde2\bar3}
\right\rangle   \\
\midrule[0.02em]
(3,2) & \C \left\langle \varphi_0^{123\tilde1\tilde2},\, \varphi_0^{123\tilde1\bar3},\, \varphi_0^{123\tilde2\bar3} \right\rangle  & \C \left\langle \varphi_t^{123\tilde1\tilde2}
\right\rangle  & \C \left\langle \varphi_0^{123\tilde1\tilde2},\, \varphi_0^{123\tilde1\bar3},\, \varphi_0^{123\tilde2\bar3},\, \varphi_0^{\bar{\tilde1}23\bar1\bar3},\, \varphi_0^{\bar{\tilde1}23\bar2\bar3}
\right\rangle  & \C \left\langle \varphi_t^{123\tilde1\tilde2} \right\rangle
\\[5pt]
(2,3) & \C \left\langle \varphi_0^{12\tilde1\tilde2\bar3},\, \varphi_0^{13\tilde1\tilde2\bar3},\, \varphi_0^{23\tilde1\tilde2\bar3} \right\rangle  & \C \left\langle
\varphi_t^{12\tilde1\tilde2\bar3} \right\rangle  & \C \left\langle \varphi_0^{12\tilde1\tilde2\bar3},\, \varphi_0^{13\tilde1\tilde2\bar3},\, \varphi_0^{23\tilde1\tilde2\bar3},\, \varphi_0^{\bar{\tilde1}3\bar1\bar2\bar3},\,
\varphi_0^{\bar{\tilde2}3\bar1\bar2\bar3} \right\rangle  & \C \left\langle \varphi_t^{12\tilde1\tilde2\bar3} \right\rangle \\
\midrule[0.02em]
(3,3) & \C \left\langle \varphi_0^{123\tilde1\tilde2\bar3} \right\rangle  & \C \left\langle \varphi_t^{123\tilde1\tilde2\bar3} \right\rangle & \C \left\langle \varphi_0^{123\tilde1\tilde2\bar3} \right\rangle  & \C \left\langle \varphi_t^{123\tilde1\tilde2\bar3} \right\rangle   \\
\bottomrule
\end{tabular}
}
\medskip
\caption{The harmonic representatives with respect to the metric $g := \varphi_t^1\odot\tilde\varphi_t^1 + \varphi_t^2\odot\tilde\varphi_t^2 + \varphi_t^3\odot\tilde\varphi_t^3$ of the Dolbeault and the Bott-Chern
cohomology of the complex solvmanifolds $(X_k)_t$.}
\label{table:BC-JAt}
\end{table}

\end{landscape}

%%%%%%%%%%%%%%%%%%%%%%%%%%%%%%%%%%%%%%%%%%%%%%%
%%%%%%%%%%%%%%%%%%%%%%%%%%%%%%%%%%%%%%%%%%%%%%%
\appendix\section{Reduction of parameters}\label{apendice} 
%%%%%%%%%%%%%%%%%%%%%%%%%%%%%%%%%%%%%%%%%%%%%%%
%%%%%%%%%%%%%%%%%%%%%%%%%%%%%%%%%%%%%%%%%%%%%%%

\noindent In this appendix we show how to reduce the value of the parameters in the algebras $\frs_5^{\alpha},\, \frs_6^{\alpha,\,\beta},\, \frs_7^{\alpha},\, \frs_8^{\alpha},\, \frs_{10}^{\alpha, \beta},\, \frs_{11}^{\alpha}$ according to Theorem~\ref{thm:main-thm}.

\medskip

Let us consider the following changes from a basis $\{e^1,\ldots, e^6\}$ to another real basis $\{f^1,\ldots, f^6\}$, where $\lambda$ is a non-zero real number:

\begin{enumerate}
\item[\textbf{ChA}] $f^i = e^i,\, i=1,3,5,6,\quad f^2 = e^4,\quad f^4 = e^2.$ \medskip

\item[\textbf{ChB}] $f^i = e^i,\, i=1,2,5,6,\quad f^3 = e^4,\quad f^4 = e^3.$ \medskip

\item[\textbf{ChC}] $f^i = e^i,\, i=1,3,6,\quad f^2 = -e^2,\quad f^4 = -e^4,\quad  f^5 = -e^5.$ \medskip

\item[\textbf{ChD}] $f^i = e^i,\, i=1,2,3,5,6,\quad f^4 = -e^4.$ \medskip

\item[\textbf{ChE}] $f^1 = e^3,\quad f^2 = -e^4,\quad f^3 = e^1,\quad f^4 = -e^2,\quad f^5 = -\lambda e^5,\quad f^6 = e^6.$ \medskip

\item[\textbf{ChF}] $f^i = e^i,\, i=2,4,6,\quad f^1 = -e^5,\quad f^3 = e^1,\quad f^5 = e^3.$ \medskip

\item[\textbf{ChG}]  $f^1 = e^3,\quad f^2 = e^4,\quad f^3 = e^1,\quad f^4 = e^2,\quad  f^i = e^i,\quad i=5,6.$ \medskip

\item[\textbf{ChH}] $f^1 = e^3,\quad  f^2 = e^4,\quad f^3 = e^1,\quad f^4 = e^2,\quad f^5 = \lambda e^5,\quad f^6 = -e^6.$ \medskip

\end{enumerate}

\medskip

\noindent   \textbf{Case }  $\frs_5^{\alpha}$:  Consider $\frs_5^{\alpha}$ where $\alpha\in\mathbb R$, with structure equations  \begin{equation*}\label{s5}
\frs_5^{\alpha} = (e^{15}, e^{25}, -e^{35} + \alpha\,e^{45}, -\alpha\,e^{35}-e^{45},0,0),\quad \alpha\in \mathbb R.
\end{equation*}
Then:
\begin{itemize}
\item If $\alpha =0$, change ChA gives the isomorphism $\frs_5^{0}\cong\frs_4$. \medskip
\item Change ChB gives the isomorphism $\frs_5^{\alpha}\cong\frs_5^{-\alpha}$.
\end{itemize}

\medskip

Therefore, we can suppose $\alpha>0$.

\bigskip

\noindent   \textbf{Case }  $\frs_6^{\alpha, \beta}$:  Consider $\frs_6^{\alpha, \beta}$ where $\alpha, \beta\in\mathbb R$, with structure equations  \begin{equation*}\label{s6}
\frs_6^{\alpha, \beta} = (\alpha\,e^{15} + e^{25}, -e^{15}+\alpha\,e^{25}, -\alpha\,e^{35} + \beta\,e^{45}, -\beta\,e^{35}-\alpha\,e^{45},0,0),\quad \alpha, \beta\in \mathbb R.
\end{equation*}
Then:
\begin{itemize}
\item Change ChC gives the isomorphism $\frs_6^{\alpha, \beta}\cong\frs_6^{-\alpha, \beta}$. \medskip
\item Change ChD gives the isomorphism $\frs_6^{\alpha, \beta}\cong\frs_6^{\alpha, -\beta}$. \medskip
\item If $\alpha=0$, $\frs_6^{0,\beta}\cong\frs_7^{\beta}$. \medskip
\item If $\beta=0$ and $\alpha\neq 0$, change ChE with $\lambda = \alpha$ gives the isomorphism $\frs_6^{\alpha, 0}\cong\frs_5^{\frac1\alpha}$. \medskip
\item If $\beta=1$, $\frs_6^{\alpha,1}\cong\frs_8^{\alpha}$. \medskip
\item If $\beta\neq 0, 1$, change ChE with $\lambda = \beta$ gives the isomorphism $\frs_6^{\alpha, \beta}\cong\frs_6^{\frac{\alpha}{\beta}, \frac1\beta}$. \medskip
\end{itemize}

\medskip

Therefore, we can suppose $\alpha>0$ and $\beta\in (0,1)$.

\bigskip

\noindent   \textbf{Case }  $\frs_7^{\alpha}$: Consider $\frs_7^{\alpha}$ where $\alpha\in\mathbb R$, with structure equations  \begin{equation*}\label{s7}
\frs_7^{\alpha} = (e^{25}, -e^{15}, \alpha\,e^{45}, -\alpha\,e^{35},0,0),\quad \alpha\in \mathbb R.
\end{equation*}
Then: 
\begin{itemize}
\item If $\alpha=0$, change ChF gives the isomorphism $\frs_7^0\cong\frs_2$. \medskip
\item Change ChD gives the isomorphism $\frs_7^{\alpha}\cong\frs_7^{-\alpha}$. \medskip
\item Change ChE with $\lambda = \alpha$ gives the isomorphism $\frs_7^{\alpha}\cong\frs_7^{\frac1\alpha}$. \medskip
\end{itemize}

\medskip

Therefore, we can suppose $0<\alpha\leq 1$.

\bigskip

\noindent   \textbf{Case }  $\frs_8^{\alpha}$:  Consider $\frs_8^{\alpha}$ where $\alpha\in\mathbb R$, with structure equations  \begin{equation*}\label{s8}
\frs_8^{\alpha} = (\alpha e^{15} + e^{25}, -e^{15} + \alpha e^{25}, -\alpha e^{35} + e^{45}, -e^{35} - \alpha e^{45},0,0),\quad \alpha\in \mathbb R.
\end{equation*}
Then: 
\begin{itemize}
\item Observe that $\frs_8^0\cong\frs_{7}^1$. \medskip
\item Change ChC gives the isomorphism $\frs_8^{\alpha}\cong\frs_8^{-\alpha}$. \medskip
\end{itemize}

\medskip

Therefore, we can suppose $\alpha>0$.

\bigskip

\noindent   \textbf{Case } $\frs_{11}^{\alpha}$:  Consider $\frs_{11}^{\alpha}$ where $\alpha\in\mathbb R$, with structure equations  \begin{equation*}\label{s11}
\frs_{11}^{\alpha} = (e^{16} - e^{25}, e^{15} + e^{26}, -e^{36} - \alpha e^{45}, \alpha e^{35} - e^{46},0,0),\quad \alpha\in \mathbb R.
\end{equation*}
Then: 
\begin{itemize}
\item If $\alpha =0$, change ChG gives the isomorphism $\frs_{11}^0\cong\frs_9$. \medskip
\item Change ChB gives the isomorphism $\frs_{11}^{\alpha}\cong\frs_{11}^{-\alpha}$. \medskip
\item Change ChH with $\lambda = \alpha$ gives the isomorphism $\frs_{11}^{\alpha}\cong\frs_{11}^{\frac1\alpha}$. \medskip
\item Change ChB gives the isomorphism $\frs_{11}^{1}\cong\frs_{12}$. \medskip
\end{itemize}

\medskip

Therefore, we can suppose $\alpha\in (0,1)$.

\end{document}